\documentclass{article}
\usepackage[english]{babel}
\usepackage[T1]{fontenc}
\usepackage[utf8]{inputenc}
\usepackage{amsthm}
\usepackage{amsmath}
\usepackage{amssymb}
\usepackage{palatino}
\usepackage{graphicx}
\usepackage[colorlinks=true, citecolor =blue]{hyperref}
\usepackage{dsfont}
\usepackage[all]{xy}
\usepackage[top=2.5cm,bottom=2.5cm,left=1cm,right=1cm]{geometry}
\usepackage{enumitem}
\usepackage{mathtools}
\usepackage{subcaption}
\usepackage{authblk}
\usepackage{bbm}
\usepackage{upgreek}
\usepackage{mathrsfs}

\graphicspath{{Counting/}{Intro/}}
\newcommand{\red}{\textcolor{red}}
\newcommand{\blue}{\textcolor{blue}}
\usepackage{makecell}
\usepackage{float}
\newcommand{\p}{\partial}
\newcommand{\eps}{\varepsilon}
\newcommand{\pl}{\mathtt{p}_l}
\newcommand{\pr}{\mathtt{p}_r}
\newcommand{\plq}{Q_{\mathtt{p}_l}^{\ell}}
\newcommand{\prq}{Q_{\mathtt{p}_r}^{\ell}}
\newtheorem{theorem}{Theorem}[section]
\newtheorem{lemma}[theorem]{Lemma}
\newtheorem{definition}[theorem]{Definition}
\newtheorem{remark}[theorem]{Remark}
\newtheorem{proposition}[theorem]{Proposition}

\newtheorem{conditionbase}{Condition}

\newenvironment{condition}[1]{%
  \renewcommand{\theconditionbase}{(#1)}%
  \begin{conditionbase}%
}{%
  \end{conditionbase}%
}

\usepackage{tikz}

\newcommand{\dist}{\textup{dist}}

\numberwithin{equation}{section}

\renewcommand{\l}{\mathfrak{l}}

\newcommand{\new}{\textup{new}}

\usepackage[linesnumbered, ruled, vlined]{algorithm2e}

\renewcommand{\o}{\mathfrak{o}}
\newcommand{\T}{\mathcal{T}}
\newcommand{\n}{\mathfrak{n}}
\newcommand{\m}{\mathfrak{m}}

\newcommand{\chard}{\mathtt{C}_{\textup{hard}}}

\newcommand{\fl}{f}

\newcommand{\xp}{\mathbf{x}^{\mathtt{p}}}
\newcommand{\kplq}{Q_{\mathtt{p}_l}^{\mathtt{K}\ell}}
\newcommand{\kprq}{Q_{\mathtt{p}_r}^{\mathtt{K}\ell}}

\newcommand{\vepst}{\boldsymbol{\mu}_{\eps,T}}
\newcommand{\indicator}{\mathds{1}_{\lambda\textup{ is $\eps$-cluster on $[0,T]$}}}
\newcommand{\indicatork}[1]{\mathds{1}_{\lambda_{#1}\textup{ is $\eps$-cluster on $[0,T]$}}}
\newcommand{\indicatormink}[1]{\mathds{1}_{\lambda_{#1}\textup{ is minimal $\eps$-cluster on $[0,T]$}}}
\newcommand{\ldif}{\mathcal{L}_{\textup{dif}}}

\makeatletter

\newcommand{\binarytree@draw}[3]{%
  \vcenter{\hbox{%
    \begin{tikzpicture}[
      x=1em,
      y=1em,
      scale=#1,
      transform shape,
      line cap=round,
      line join=round,
      tree edge/.style={
        draw=#3,
        line width=.055em
      },
      tree node/.style={
        circle,
        draw=#3,
        fill=white,
        line width=.055em,
        minimum size=.23em,
        inner sep=0pt,
        outer sep=0pt
      }
    ]
      \node[tree node,#2] (root)  at (0,.33) {};
      \node[tree node]    (left)  at (-.36,-.33) {};
      \node[tree node]    (right) at ( .36,-.33) {};

      \draw[tree edge]
        (root) -- (left)
        (root) -- (right);
    \end{tikzpicture}%
  }}%
}

\newcommand{\binarytree@choice}[2]{%
  \mathchoice
    {\binarytree@draw{1}{#1}{#2}}%
    {\binarytree@draw{1}{#1}{#2}}%
    {\binarytree@draw{.75}{#1}{#2}}%
    {\binarytree@draw{.60}{#1}{#2}}%
}

\newcommand{\OverlapFork}[1][black]{%
  \mathord{\binarytree@choice{fill=#1}{#1}}%
}

\newcommand{\CollisionFork}[1][black]{%
  \mathord{\binarytree@choice{fill=white}{#1}}%
}

\makeatother

\newcommand{\nnew}{\mathcal{N}_{\textup{new}}}

\newcommand{\timeldelta}{L^{\infty}\mathbf{L}_{\gamma,\delta}^1}
\newcommand{\timeldeltastar}{L^{\infty}\mathbf{L}_{\gamma_*,\delta_*}^1}
\newcommand{\ldelta}{\mathbf{L}_{\gamma,\delta}^1}
\newcommand{\ldeltastar}{\mathbf{L}_{\gamma_*,\delta_*}^1}

\newcommand{\tprime}{\mathcal{T}_0^+}
\newcommand{\x}{\mathbf{x}}

\usepackage{multirow}
\usepackage{arydshln}
\newcommand{\onew}{\mathfrak{o}[\textup{new}]}
\newcommand{\oold}{\mathfrak{o}[\textup{old}]}

\newcommand{\revcar}[1]{d^{\uparrow}[v]_{\mathcal{N}^+(#1)}}

\makeatletter

\newcommand{\binarytree@double@draw}[2]{%
  \vcenter{\hbox{%
    \begin{tikzpicture}[
      x=1em,
      y=1em,
      scale=#1,
      transform shape,
      line cap=round,
      line join=round,
      tree edge/.style={
        draw=#2,
        line width=.055em
      },
      tree node/.style={
        circle,
        draw=#2,
        fill=white,
        line width=.055em,
        minimum size=.23em,
        inner sep=0pt,
        outer sep=0pt
      }
    ]
      \node[tree node] (root) at (0,.95) {};

      \node[tree node] (L) at (-.65,.20) {};
      \node[tree node] (R) at ( .65,.20) {};

      \node[tree node] (LL) at (-1.00,-.55) {};
      \node[tree node] (LR) at (-.30,-.55) {};
      \node[tree node] (RL) at ( .30,-.55) {};
      \node[tree node] (RR) at (1.00,-.55) {};

      \draw[tree edge]
        (root) -- (L)
        (root) -- (R)
        (L) -- (LL)
        (L) -- (LR)
        (R) -- (RL)
        (R) -- (RR);
    \end{tikzpicture}%
  }}%
}

\newcommand{\binarytree@double@choice}[1]{%
  \mathchoice
    {\binarytree@double@draw{1}{#1}}%
    {\binarytree@double@draw{1}{#1}}%
    {\binarytree@double@draw{.75}{#1}}%
    {\binarytree@double@draw{.60}{#1}}%
}

\newcommand{\DoubleBinaryTree}[1][black]{%
  \mathord{\binarytree@double@choice{#1}}%
}

\makeatother

\makeatletter

\newcommand{\lefttree@draw}[2]{%
  \vcenter{\hbox{%
    \begin{tikzpicture}[
      x=1em,
      y=1em,
      scale=#1,
      transform shape,
      line cap=round,
      line join=round,
      tree edge/.style={
        draw=#2,
        line width=.055em
      },
      tree node/.style={
        circle,
        draw=#2,
        fill=white,
        line width=.055em,
        minimum size=.23em,
        inner sep=0pt,
        outer sep=0pt
      }
    ]
      \node[tree node] (root) at (0,.65) {};

      \node[tree node] (L) at (-.50,-.05) {};
      \node[tree node] (R) at ( .50,-.05) {};

      \node[tree node] (LL) at (-.86,-.75) {};
      \node[tree node] (LR) at (-.14,-.75) {};

      \draw[tree edge]
        (root) -- (L)
        (root) -- (R)
        (L) -- (LL)
        (L) -- (LR);
    \end{tikzpicture}%
  }}%
}

\newcommand{\lefttree@choice}[1]{%
  \mathchoice
    {\lefttree@draw{1}{#1}}%
    {\lefttree@draw{1}{#1}}%
    {\lefttree@draw{.75}{#1}}%
    {\lefttree@draw{.60}{#1}}%
}

\newcommand{\LeftBranchTree}[1][black]{%
  \mathord{\lefttree@choice{#1}}%
}

\makeatother

\makeatletter

\newcommand{\leftoverlap@draw}[2]{%
  \vcenter{\hbox{%
    \begin{tikzpicture}[
      x=1em,
      y=1em,
      scale=#1,
      transform shape,
      line cap=round,
      line join=round,
      tree edge/.style={
        draw=#2,
        line width=.055em
      },
      tree node/.style={
        circle,
        draw=#2,
        fill=white,
        line width=.055em,
        minimum size=.23em,
        inner sep=0pt,
        outer sep=0pt
      }
    ]
      \node[tree node,fill=#2] (root) at (0,.65) {};

      \node[tree node] (L) at (-.50,-.05) {};
      \node[tree node] (R) at ( .50,-.05) {};

      \node[tree node] (LL) at (-.86,-.75) {};
      \node[tree node] (LR) at (-.14,-.75) {};

      \draw[tree edge]
        (root) -- (L)
        (root) -- (R)
        (L) -- (LL)
        (L) -- (LR);
    \end{tikzpicture}%
  }}%
}

\newcommand{\leftoverlap@choice}[1]{%
  \mathchoice
    {\leftoverlap@draw{1}{#1}}%
    {\leftoverlap@draw{1}{#1}}%
    {\leftoverlap@draw{.75}{#1}}%
    {\leftoverlap@draw{.60}{#1}}%
}

\newcommand{\LeftOverlapTree}[1][black]{%
  \mathord{\leftoverlap@choice{#1}}%
}

\makeatother

\author{Chenjiayue Qi\thanks{chenjiayue@ihes.fr}}
\affil{Institut des Hautes Études Scientifiques (IHES), Université Paris-Saclay, 91440, Bures-sur-Yvette, France}

\date{\today}

\begin{document}

\title{Boltzmann Equation and Hard Sphere Gas with Singular Initial Data: Well-posedness, Kinetic Limit, and Weak Stability}

\maketitle
\begin{abstract}
We develop a local well-posedness theory for the spatially inhomogeneous 3-d Boltzmann equation with hard sphere cross-section, accommodating a class of large initial data. This class is an interpolation between the space $L_{x}^{\infty}L_v^{1}$ and $L_{v}^1L_{x}^{\infty}$ (with weight in $v$). It is notable for including data that can have infinite $L_{x,v}^\infty$-norm, entropy density. The core idea of the proof is a homogenization method based on collision geometries. As a key application, we prove the kinetic limit of hard sphere gases with initial data in this space, having certain extra conditions. This is the first proof on the kinetic limit of hard sphere gases with singular initial data. We also derive weak stability estimates of Boltzmann equation and hard sphere gas within the same functional spaces.
\end{abstract}

\paragraph{Ackownledgement:} The author is grateful to Thierry Bodineau, Yu Deng, Laure Saint-Raymond, and Sergio Simonella for enlightening discussions. Part of this work was conducted during a visit to Sapienza University of Rome in November 2025, and the author is grateful to Sergio Simonella for the invitation and the warm hospitality.

\tableofcontents

\section{Introduction}\label{sec:intro}
Ever since its introduction, the Boltzmann equation describing the evolution of dilute gases
\begin{equation*}
\p_tf+v\cdot\nabla_xf=\mathbf{Q}(f,f),\quad \textup{where $\mathbf{Q}(\cdot,\cdot)$ is a bilinear collision operator,}
\end{equation*}
has stood as a cornerstone of statistical physics. The study of its well-posedness and its rigorous derivation from hard sphere gas dynamics have been central problems in the field.

For the Boltzmann equation, DiPerna-Lions in their seminal work \cite{diperna1989cauchy} established the existence of global renormalized solutions for the inhomogeneous Boltzmann equation with general $L^1$ initial data, having \textit{finite mass, energy, and entropy}. However, the \textit{uniqueness} and \textit{energy conservation} of these solutions remain major open problems. Although a unified resolution of these open problems is still absent today, substantial progress has been achieved along three major lines of development: Perturbative theory around homogeneous solutions \cite{gualdani2010factorization}, \cite{duan2017global}, \cite{gressman2011global}; Near-vacuum theory \cite{illner1984boltzmann}, \cite{chaturvedi2021stability}; Local-in-time theory \cite{kaniel1978boltzmann}, \cite{chen2019local}, \cite[Section 5]{gallagher2014newton}.

For the hard sphere gas, the foundational framework of Lanford \cite{lanford2005time} was continued by \cite{cercignani2013mathematical,spohn2006boltzmann,uchiyama1988derivation}. Then Gallagher-Saint-Raymond-Texier \cite{gallagher2014newton} established the rigorous derivation of Boltzmann equation from hard sphere gas. In the following years, there have been various works on the short-time non-equilibrium fluctuations \cite{bodineau2022cluster,Bodineau_Laure_Gallagher_Simonella_2023}, the long-time equilibrium fluctuations \cite{bodineau2023long,bodineau2024long,bodineau2024dynamics}, and the long-time behavior of the related Hamilton-Jacobi structure \cite{qi2024global}. A major breakthrough on the long-time derivation of Boltzmann equation \cite{Deng_Hani_Ma_2024,deng2025hilbert} is established recently. A common feature of these works is that the initial data $f^0$ is always smooth in $x$ and bounded from above by Gaussian distributions: there exists $C>0$ and $\beta>0$ such that $f^0(x,v)\leq Ce^{-\beta|v|^2}$.

In this paper, we go beyond the assumption of point-wise Gaussian upper bound. We establish the local well-posedness of Boltzmann equation and the kinetic limit of hard sphere gas with very singular initial data $f^0$, allowing \textit{infinite $L^{\infty}$-norm and entropy}. We also prove the weak stability of Boltzmann equation and hard sphere gas within the class of singular initial data. 

The kinetic limit of hard sphere gas with singular and unbounded initial data carries a physical significance. As formulated in \cite{hilbert2019mathematical}, hard sphere gas and other Newtonian models are considered as the atomistic description of gases and fluids. The gas and fluid dyanmics are described by Boltzmann/Euler/Navier-Stokes equations, whose behavior with singular initial data have been extensively studied. This work is a first step towards the study of hard sphere gas with singular initial data, providing a microscopic perspective of PDEs describing gases and fluids. In this paper, we expand the solution of Boltzmann equations into pseudo-trajectories (see Remark \ref{rmk:pseudo-trajectory_and_xfunction}) and prove the singularity of initial data can be \textit{smoothed by collisions} (Proposition \ref{prop:toy_example} and Proposition \ref{prop:point_homogenization}), which in principle also applies to the estimates of particle trajectories of hard sphere gas.

More precisely \textit{at the level of partial differential equations},  we study the Boltzmann equation with hard sphere cross-section. Namely for $(x,v)\in \mathds{R}^3\times \mathds{R}^3$ if we study Boltzmann equation on the full space, or for $(x,v)\in \mathds{T}^3\times \mathds{R}^3$ if we study the Boltzmann equation on the torus $\mathds{T}^3:=[0,1]^3$
\begin{equation}\label{eq:Boltzmann}
\p_tf+v\cdot\nabla_xf=\mathbf{Q}(f,f),\quad \mathbf{Q}(f_1,f_2):=\int \big((v_*-v)\cdot\omega\big)_+ \Big(f_1(v')f_2(v_*')-f_1(v)f_2(v_*)\Big)d\omega dv_*,
\end{equation}
where $(v',v_*')$ is the pre-collisional configuration
\begin{equation}\label{eq:pre_collision}
v'=v-(v-v_*)\cdot\omega\omega,\quad v_*'=v_*+(v-v_*)\cdot\omega\omega.
\end{equation}
\begin{figure}[h]
    \centering
    \includegraphics[width=0.3\linewidth]{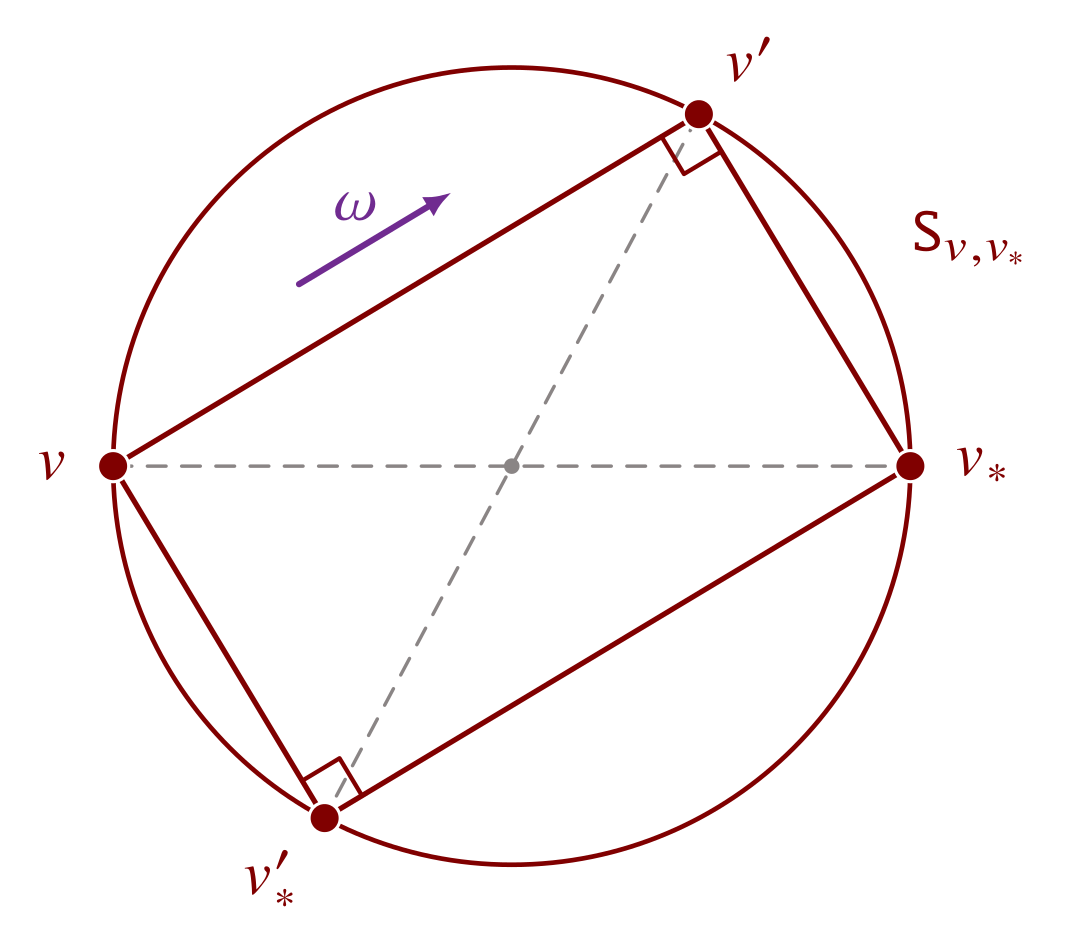}
    \caption{An illustration of collision sphere $\mathtt{S}_{v,v_*}$ defined in \eqref{eq:collision_sphere}. The variables $(v,v_*)$ exchange their velocity in the $\omega$-direction. Notice $(v-v')\perp (v-v_*')$.}
    \label{fig:collision_sphere}
\end{figure}
Notice that given $(v,v_*)$, the pre-collisional configuration $v'$ and $v_*'$ in \eqref{eq:pre_collision} always belong to the collision sphere $\mathtt{S}_{v,v_*}$ (Figure \ref{fig:collision_sphere}) defined as
\begin{equation}\label{eq:collision_sphere}
\mathtt{S}_{v,v_*}:=\Big\{v^{\dagger}\in\mathds{R}^3\Big|\ |v^{\dagger}-\frac{v+v_*}{2}|=|\frac{v-v_*}{2}|\Big\}.
\end{equation}
For the operator $\mathbf{Q}$, we frequently use the decomposition of $\mathbf{Q}$ into gain part and loss part
\begin{equation}\label{eq:gain_and_loss_decomposition}
\begin{split}
&\mathbf{Q}(f_1,f_2)=\mathbf{Q}^+(f_1,f_2)-\mathbf{Q}^-(f_1,f_2),\\
&\mathbf{Q}^+(f_1,f_2):=\int \big((v_*-v)\cdot\omega\big)_+ f_1(v')f_2(v_*')d\omega dv_*,\quad \mathbf{Q}^-(f_1,f_2):=\int \big((v_*-v)\cdot\omega\big)_+f_1(v)f_2(v_*)d\omega dv_*
\end{split}
\end{equation}

\textit{At the level of microscopic dynamics,} in this paper we work with the hard sphere gas on the torus $\mathds{T}^3=[0,1]^3$. Hard sphere gas is the microscopic particle system corresponding to the Boltzmann equation \eqref{eq:Boltzmann} with hard-sphere cross-section. Fix the total number $N$ of hard sphere particles, as well as the diameter $\eps>0$ of these hard spheres. The evolution for the positions $(\mathbf{x}_1^\eps,...,\mathbf{x}_N^\eps)\in \mathds{T}^{3N}$ and velocities $(\mathbf{v}_1^\eps,...,\mathbf{v}_N^\eps)\in \mathds{R}^{3N}$ of the $N$ particles satisfies a system of ordinary differential equations (Newton's laws)
\begin{equation}
\frac{d\mathbf{x}_i^\eps}{dt}=\mathbf{v}_i^{\eps},\quad \frac{d\mathbf{v}_i^\eps}{dt}=0\quad \textup{as long as $|\mathbf{x}_i^\eps-\mathbf{x}_j^\eps|>\eps$ for arbitrary $j$ with $1\leq i\not=j\leq N$},
\label{eq:hard_sphere_law_1}
\end{equation}
with specular reflection at collisions if $|\mathbf{x}_i^{\eps}-\mathbf{x}_j^{\eps}|=\eps$
\begin{equation}\label{eq:hard_sphere_law_2}
(\mathbf{v}_i^{\eps})':=\mathbf{v}_i^{\eps}-(\mathbf{v}_i^{\eps}-\mathbf{v}_j^{\eps})\cdot\omega\omega,\quad (\mathbf{v}_j^{\eps})':=\mathbf{v}_j^{\eps}+(\mathbf{v}_i^{\eps}-\mathbf{v}_j^{\eps})\cdot\omega\omega,\quad \omega:=\frac{\mathbf{x}_i^{\eps}-\mathbf{x}_j^{\eps}}{\eps}.
\end{equation}
This means that after a collision with $|\mathbf{x}_i^{\eps}-\mathbf{x}_j^{\eps}|=\eps$, the velocities $(\mathbf{v}_i^{\eps},\mathbf{v}_j^{\eps})$ of the two particles will be changed into $\big((\mathbf{v}_i^{\eps})',(\mathbf{v}_j^{\eps})'\big)$. This induces a well-defined dynamics for initial conditions of full Lebesgue measure in the canonical phase space $\mathcal{D}_N^{\eps}$
\begin{equation}\label{eq:exclusion_domain}
\mathcal{D}_{N}^{\eps}:=\Big\{(\mathbf{x}_1^\eps,...,\mathbf{x}_N^\eps,\mathbf{v}_1^\eps,...,\mathbf{v}_N^\eps)\in\mathds{T}^{3N}\times \mathds{R}^{3N}:\ \forall i\not=j,\ |\mathbf{x}_i^\eps-\mathbf{x}_j^\eps|>\eps\Big\},
\end{equation}
excluding multiple collisions and accumulation of collision times \cite{alexander1975infinite}. This microscopic dynamics \eqref{eq:hard_sphere_law_1}-\eqref{eq:hard_sphere_law_2} induces a Liouville equation for the probability density $W_N^{\eps}$ of the $N$ particles, where $W_{N}^{\eps}(t,X_N,V_N)$ refers to the probability density of finding $N$ hard spheres with configuration $(X_N,V_N)$ at time $t$. Here the variable $X_N$ refers to the positions of the $N$ particles $X_N:=(x_1,...,x_N)$, and the variable $V_N$ refers to the velocities of the $N$ particles $V_N:=(v_1,...,v_N)$. The Liouville equation for $W_{N}^{\eps}$ is
\begin{equation*}
\p_t W_{N}^\eps+V_N\cdot\nabla_{X_N}W_{N}^\eps=0,
\end{equation*}
with boundary condition corresponding to specular reflection \eqref{eq:hard_sphere_law_2}. 

One can further consider the grand canonical formulation of a hard sphere gas: instead of fixing the total number of particles, we assume the total number $\mathbf{N}$ of particles to be random with a modified Poisson distribution law. For each diameter $\eps$, we fix $\mu_\eps$ as the parameter for the modified Poisson distribution of the total number of particles. We assume that at initial time $t=0$ the probability density of having $N$ particles and configuration $(X_N,V_N)$ is the following with $z_i:=(x_i,v_i)$
\begin{equation}\label{eq:grand_canonical_measure}
\frac{1}{\mathcal{Z}^{\eps}}\frac{\mu_{\eps}^N}{N!}\prod_{i=1}^{N}f^0(z_i)\mathbbm{1}_{\mathcal{D}_{N}^{\eps}},\quad \mathcal{Z}^{\eps}:=1+\sum_{N\geq 1}\frac{\mu_{\eps}^N}{N!}\int_{\mathds{T}^{3N}\times \mathds{R}^{3N}}\prod_{i=1}^{N}f^0(z_i)\mathbbm{1}_{\mathcal{D}_{N}^{\eps}}dV_NdX_N.
\end{equation}
It is clear that we have two sources of randomness: the number $\mathbf{N}$ of particles is random, while given the total number $N$, the configuration $(X_{N},V_{N})$ is also random. Each sample $(X_{N},V_{N})$ follows the evolution law given by equations \eqref{eq:hard_sphere_law_1} and \eqref{eq:hard_sphere_law_2}. Thus at each time $t\geq 0$, this defines the distribution law for $(X_{\mathbf{N}},V_{\mathbf{N}})$. 

In the Boltzmann-Grad scaling, we impose $\mu_{\eps}=\eps^{-2}$ to ensure the number of collisions per particle is of order $1$ per unit time \cite{Grad_1949}. This scaling implies the asymptotic \eqref{eq:Bol_Grad} for $\mathds{E}_{\eps}(\mathbf{N})$, where $\mathds{E}_\eps$ is the expectation upon the probability measure given by \eqref{eq:grand_canonical_measure}
\begin{equation}\label{eq:Bol_Grad}
\lim_{\eps\rightarrow 0}\mathds{E}_{\eps}(\mathbf{N})\eps^2=O(1).
\end{equation}
Given the probability measure $\mathds{P}_{\eps}$ of the hard sphere gas $(X_{\mathbf{N}},V_{\mathbf{N}})$ at different times, we can define the empirical measure $\pi_{t}^{\eps}$ as a random measure
\begin{equation}\label{eq:empirical_measure}
\pi_t^{\eps}:=\frac{1}{\mu_{\eps}}\sum_{1\leq i\leq \mathbf{N}}\delta_{(\x_i^{\eps}(t),\mathbf{v}_i^{\eps}(t))},\quad \pi_t^{\eps}(h):=\frac{1}{\mu_{\eps}}\sum_{1\leq i\leq \mathbf{N}}h\big(\x_i^{\eps}(t),\mathbf{v}_i^{\eps}(t)\big),\ \textup{for any function $h$ defined on $\mathds{T}^3\times \mathds{R}^3$}.
\end{equation}

Given initial data $f^0$ and the hard sphere gas with initial data \eqref{eq:grand_canonical_measure}, the kinetic limit of the hard sphere gas during time $[0,T]$ refers to the convergence in probability of the empirical measure $\pi_t^{\eps}$ towards $f(t)$, which is the unique mild solution of Boltzmann equation \eqref{eq:Boltzmann} with initial data $f^0$
\begin{equation}
\lim_{\eps\rightarrow 0}\mathds{P}_{\eps}\Big[\ \Big|\pi_t^{\eps}(h)-\int_{\mathds{T}^3\times \mathds{R}^3}f(t,x,v)h(x,v)dvdx\Big|>\delta\ \Big]=0,\quad \textup{for any $t\in[0,T]$, any test function $h$, and any $\delta>0$}.
\end{equation}

\subsection{Statement of Well-posedness Result, and Examples of Singular Initial Data}

Given arbitrary functions $f(x,v)$ on $(x,v)\in\mathds{R}^3\times \mathds{R}^3$, and an arbitrary parameter $0\leq \delta\leq +\infty$, define $[f]_{\delta}(x,v)$ as
\begin{equation}\label{eq:delta_sup_def}
[f]_{\delta}(x,v):=\sup_{x':|x'-x|\leq \delta}|f(x',v)|.
\end{equation}
Notice that $[f]_{\delta}(x,v)=|f(x,v)|$ when $\delta=0$ and $[f]_{\delta}(x,v)=\sup_{x\in\mathds{R}^3}|f(x,v)|$ when $\delta=+\infty$. Now we define the function space of initial data for the Boltzmann equation \eqref{eq:Boltzmann} used in this paper.
\begin{definition}\label{def:Ldeltasup}
For arbitrary parameters $\gamma>0$ and $0\leq \delta\leq +\infty$, define $\mathbf{L}_{\gamma,\delta}^1(\mathds{R}^3\times \mathds{R}^3)$ as the following space of functions with variables $(x,v)$
\begin{equation}\label{eq:Ldeltasup}
\ldelta:=\Big\{f:\mathds{R}^3\times \mathds{R}^3\rightarrow \mathds{R}\Big|\ \lVert f\rVert_{\ldelta(\mathds{R}^3\times \mathds{R}^3)}<+\infty\Big\},\quad  \lVert f\rVert_{\ldelta(\mathds{R}^3\times \mathds{R}^3)}:=\sup_{x\in\mathds{R}^3}\Bigg(\int [f]_{\delta}(x,v)e^{\gamma|v|^2}dv\Bigg).
\end{equation}

\end{definition}

For the study of homogeneous Boltzmann equation, the well-posedness theory of initial data $f^0(v)$ in certain weighted $L^1$ spaces has been successfully established, for example see \cite{mischler1999spatially,mouhot2004regularity}, where $f^0$ is allowed to have infinite $L^{\infty}$-norm. For inhomogeneous Boltzmann equation having initial data with infinite $L_{x,v}^{\infty}$-norm, the study of the well-posedness is more limited. The landmark paper \cite{gualdani_mischler_mouhot} proved the global-in-time well-posedness of inhomogeneous Boltzmann equations on $\mathds{T}^3$, allowing initial data as a \textit{small perturbation} around \textit{spatially homogeneous} distributions. The perturbation in \cite{gualdani2010factorization} belongs to the space $L_v^1L_{x}^{\infty}(1+|v|^k)$
\begin{equation*}
\lVert f\rVert_{L_v^1L_{x}^{\infty}(1+|v|^k)}:=\int (1+|v|^k)\Big(\sup_{x\in\mathds{T}^3}|f(x,v)|\Big)dv.
\end{equation*}
If we wish to establish a well-posedness result in a larger functional space, a natural first thought could be the space $L_x^{\infty}L_v^{1}(1+|v|^k)$. However, this space is \textit{not even invariant} under the transport semigroup $\mathbf{S}_{t}$
\begin{equation}\label{eq:transport_semigroup}
(\mathbf{S}_tf)(x,v):=f(x-tv,v).
\end{equation}
For example take function $f(x,v)=\exp\big(-|x+v|^2\big)(1+|v|^k)^{-1}$, we have $\mathbf{S}_{t=1}f(x,v)=e^{-|x|^2}(1+|v|^k)^{-1}$.

The choice of the space $\ldelta$ in Definition \ref{def:Ldeltasup} is exactly due to this observation, although we work in a setting different from \cite{gualdani2010factorization}: we will establish the \textit{local well-posedness} of Boltzmann equation in \textit{full space} $\mathds{R}^3$, allowing \textit{non-perturbative} initial data $f^0$ which can have arbitrarily large $\lVert f^0\rVert_{\ldelta}$. The time length of the local well-posedness depends on $(\gamma,\delta)$ and $\lVert f^0\rVert_{\ldelta}$. As we will see, the space $\ldelta$ serves as an interpolation between the Gaussian-weighted spaces similar to $L_{x}^{\infty}L_v^1\big(1+|v|^k\big)$ and $L_{v}^1L_{x}^{\infty}(1+|v|^k)$. For functions $f:\mathds{R}^3\rightarrow \mathds{R}$, define the Gaussian-weighted norm $\lVert f\rVert_{L_{\gamma}^1}$
\begin{equation}\label{eq:Lgamma1_def}
\lVert f\rVert_{L_{\gamma}^1}:=\int_{\mathds{R}^3}|f(v)|e^{\gamma|v|^2}dv.
\end{equation}
Given functions $f:\mathds{R}^3\times \mathds{R}^3\rightarrow \mathds{R}$, define spaces 
\begin{equation}\label{eq:function_space_compare} 
\lVert f\rVert_{L_{x}^{\infty}L_{\gamma,v}^1}:=\sup_{x\in\mathds{R}^3}\lVert f(x,\cdot)\rVert_{L_{\gamma}^1},\quad \lVert f\rVert_{L_{\gamma,v}^1L_{x}^{\infty}}:=\Big\lVert \sup_{x\in\mathds{R}^3}|f(x,\cdot)|\Big\rVert_{L_{\gamma}^1}
\end{equation}
The space $\ldelta$ can be seen as an interpolation between $L_{x}^{\infty}L_{\gamma,v}^{1}$ and $L_{\gamma,v}^{1}L_{x}^{\infty}$
\begin{equation}
L_{\gamma,v}^{1}L_{x}^{\infty}=\mathbf{L}_{\gamma,+\infty}^1\subsetneq \ldelta \subsetneq \mathbf{L}_{\gamma,0}^1=L_{x}^{\infty}L_{\gamma,v}^1,\quad \delta\in (0,+\infty).
\end{equation}

We say $f(t)$ is a mild solution of the Boltzmann equation \eqref{eq:Boltzmann} with initial data $f^0$ during $[0,T]$ if and only if for any $0\leq t\leq T$
\begin{equation}\label{eq:mild_formulation}
f(t)=\mathbf{S}_{t}f^0+\int_0^t\mathbf{S}_{t-s}\mathbf{Q}\big(f(s),f(s)\big)ds.
\end{equation}

\begin{theorem}\label{th:well-posedness}
\textbf{\textup{[Well-posedness]}} Consider any parameters $(\gamma>0,\delta>0)$ and any non-negative initial data $f^0\in \ldelta(\mathds{R}^3\times \mathds{R}^3)$. There exists a time threshold $T_*>0$ and constants $(\gamma_*>0,\delta_*>0)$ determined only by $(\gamma,\delta,\lVert f^0\rVert_{\ldelta(\mathds{R}^3\times \mathds{R}^3)})$ such that: there exists a unique non-negative mild solution of the Boltzmann equation with the initial data $f^0$ on the time interval $[0,T_*]$, within the space $L^{\infty}\ldeltastar([0,T_*]\times \mathds{R}^3\times \mathds{R}^3)$ in the sense of \eqref{eq:mild_formulation}. This unique mild solution satisfies the local conservation laws of density, momentum, and energy.

\textup{\textbf{[Weak Stability]}} Given a sequence of non-negative initial data $\{f_n^0\}_{n<+\infty}\subset \ldelta(\mathds{R}^3\times \mathds{R}^3)$ with uniformly bounded $\ldelta(\mathds{R}^3\times \mathds{R}^3)$-norm, according to the well-posedness result, there exists $(\gamma_*,\delta_*,T_*)$ such that for each $n<+\infty$ we have a unique non-negative mild solution $f_n(t)$ in space $L^{\infty}\ldeltastar([0,T_*]\times \mathds{R}^3\times \mathds{R}^3)$ with initial data $f_n^0$. 

Assume there exists non-negative $f^0\in \ldelta$ such that for any $x\in\mathds{R}^3$ we have
\begin{equation}\label{assump:initial_weak_convergence}
f_n^0(x,\cdot)\rightarrow_{\textup{weak}} f^0(x,\cdot),\quad n\rightarrow +\infty.
\end{equation}
Further assume some spatial regularity, namely that there exists a continuous function $\mathfrak{C}(x,v,\eta):\mathds{R}^3\times \mathds{R}^3\times \mathds{R}^3\rightarrow \mathds{R}_{\geq 0}$ such that
\begin{equation}\label{assump:equi_continuity}
\begin{split}
&\lim_{|\eta|\rightarrow 0}\mathfrak{C}(x,v,\eta)=0,\ \textup{for $(x,v)\in\mathds{R}^3\times \mathds{R}^3$}\\
&\sup_{n<+\infty}\Big|f_n^0(x+\eta,v)-f_n^0(x,v)\Big|\leq \mathfrak{C}(x,v,\eta),\quad \Big|f^0(x+\eta,v)-f^0(x,v)\Big|\leq \mathfrak{C}(x,v,\eta) \quad \textup{for $(x,v,\boldsymbol{\eta})\in\mathds{R}^3\times \mathds{R}^3\times \mathds{R}^3$ with small enough $|\boldsymbol{\eta}|$ }.
\end{split}
\end{equation}
Then the unique solution $f_n(t)$ with initial data $f_n^0$ weakly converges to the unique solution $f(t)$ with initial data $f^0$ in the following sense
\begin{equation}\label{eq:mixing_PDE_statement}
f_n(t,x,\cdot)\rightarrow_{\textup{weak}}f(t,x,\cdot),\quad n\rightarrow +\infty,\quad\textup{ for arbitrary $(t,x)\in [0,T_*]\times \mathds{R}^3$}.
\end{equation}
\end{theorem}

Now we construct two examples of singular initial data $f^0\in \ldelta(\mathds{R}^3\times \mathds{R}^3)$. We denote $\varphi:\mathds{R}^3\mapsto\mathds{R}_{\geq 0}$ as an arbitrary smooth and compactly supported function.

\textit{Two clouds of gases:}

For the first example of singular initial data $f^0\in \ldelta(\mathds{R}^3\times \mathds{R}^3)$, let us consider two clouds of gases with opposite velocities, moving towards each other. Take arbitrary $\zeta>0$ and construct the initial data
\begin{equation}\label{eq:two_cloud}
f^0(x,v)=\underbrace{\mathds{1}_{|x-(-1,0,0)|\leq \frac{1}{2}}\frac{\mathds{1}_{|v-(1,0,0)|\leq \zeta}}{\zeta^3}}_{\textup{\blue{cloud of gas at left going right}}}+\underbrace{\mathds{1}_{|x-(1,0,0)|\leq \frac{1}{2}}\frac{\mathds{1}_{|v-(-1,0,0)|\leq \zeta}}{\zeta^3}}_{\textup{\blue{cloud of gas at right going left}}}.
\end{equation}
Here the parameter $\zeta>0$ can be arbitrarily small, describing the degree of concentration of the velocity distribution. For small $\zeta>0$, the velocities are more concentrated around $(-1,0,0)$ and $(1,0,0)$. Theorem \ref{th:well-posedness} gives the local well-posedness of Boltzmann equation \eqref{eq:Boltzmann} with initial data $f^0$ in \eqref{eq:two_cloud} on a time interval \textit{independent} from the concentration parameter $\zeta>0$. A simulation of a hard sphere gas with initial data similar to \eqref{eq:two_cloud} is given in Figure \ref{fig:eight_figures}. 

\begin{figure}[h]
    \centering
    \begin{tabular}{cccc}
        \includegraphics[width=0.23\textwidth]{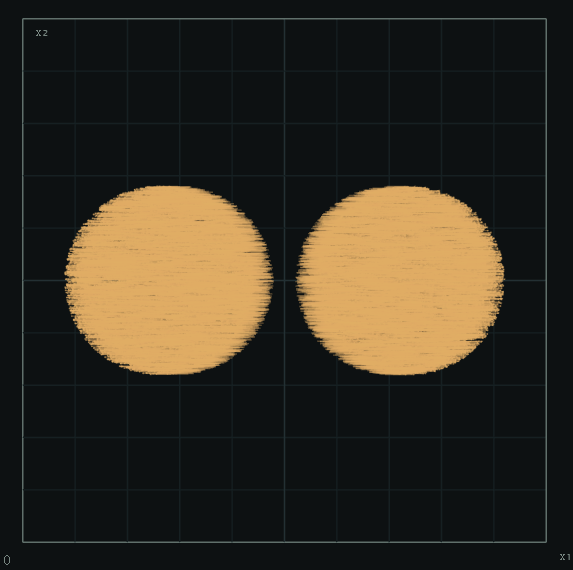}
        &
        \includegraphics[width=0.23\textwidth]{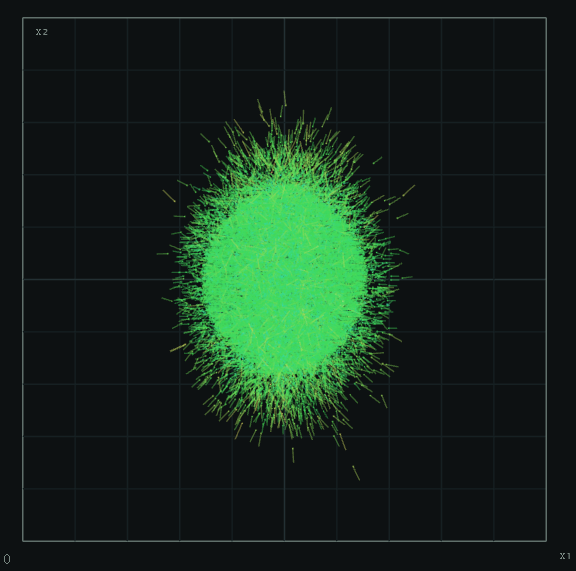}
        &
        \includegraphics[width=0.23\textwidth]{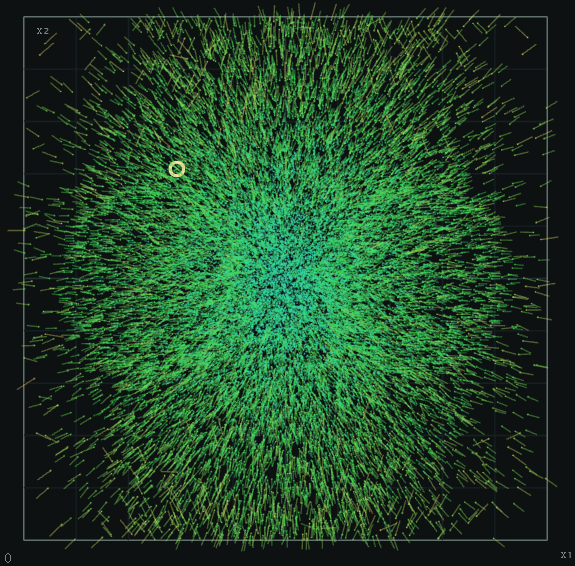}
        &
        \includegraphics[width=0.23\textwidth]{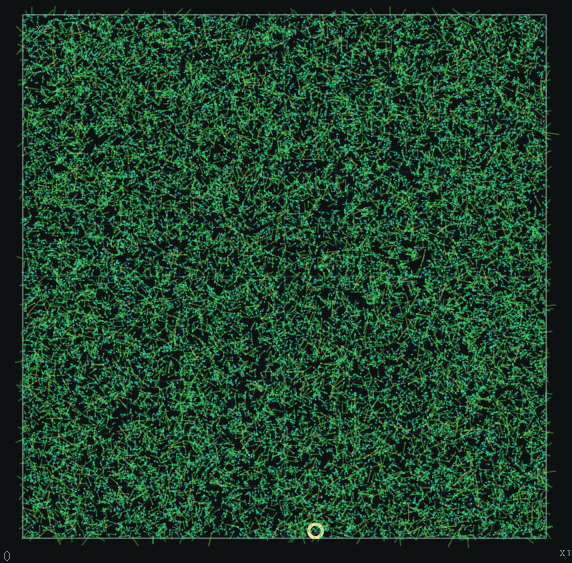}
        \\[2mm]

        \includegraphics[width=0.23\textwidth]{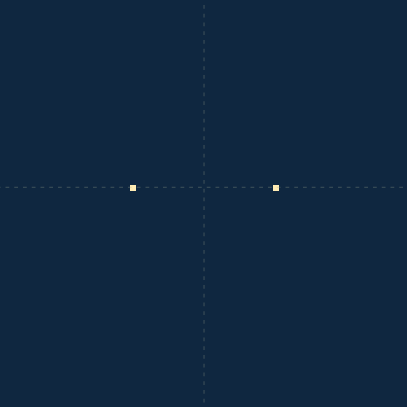}
        &
        \includegraphics[width=0.23\textwidth]{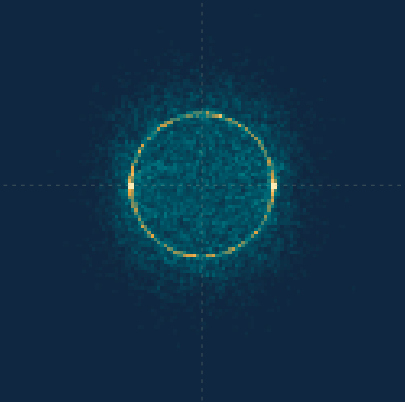}
        &
        \includegraphics[width=0.23\textwidth]{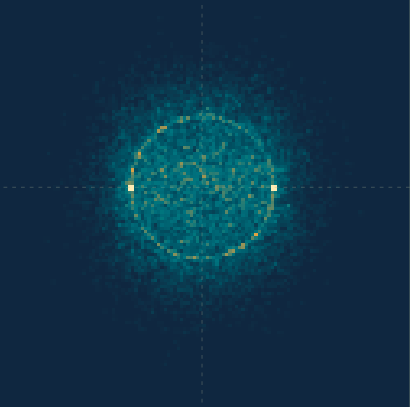}
        &
        \includegraphics[width=0.23\textwidth]{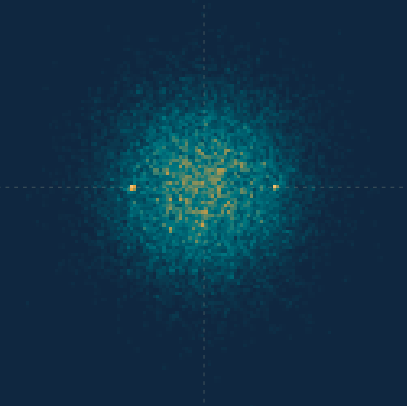}
        \\[1mm]

        \footnotesize Column 1
        &
        \footnotesize Column 2
        &
        \footnotesize Column 3
        &
        \footnotesize Column 4
    \end{tabular}

    \caption{This group of pictures is from a numerical simulation of $20000$ hard sphere particles on a two-dimensional torus. The initial data of the hard sphere gas is the two-dimensional analogue of \eqref{eq:two_cloud}, with two clouds of gases having the opposite velocities. The first row is the visualization of spatial configurations of particles, and the second row is the visualization of velocity configurations of particles. Time increases from column 1 to column 4. As time increases, the velocity distribution is spread from two singularities $(-1,0)$ and $(1,0)$ onto the collision sphere $\mathtt{S}_{(-1,0),(1,0)}$, then to the full space $\mathds{R}^2$.}
    \label{fig:eight_figures}
\end{figure}

\textit{Infinite clouds of gases:}

The example in \eqref{eq:two_cloud} can be further generalized. Suppose we have a sequence of positions $\{p_n\}_{n<+\infty}\subset\mathds{R}^3$, a sequence of velocities $\{u_n\}_{n<+\infty}\subset \mathds{R}^3$, and a sequence of positive numbers $\{\zeta_n\}_{n<+\infty}$ with $\lim_{n\rightarrow +\infty}\zeta_n=0$. We can construct the initial data $f^0$ as an infinite superposition of clouds of gases, with the $n$-th cloud supported on $|x-p_n|\leq 1$ and its velocity concentrated around $u_n$
\begin{equation}\label{eq:infinite_cloud}
f^0(x,v)=\sum_{n<+\infty}\mathds{1}_{|x-p_n|\leq 1}\frac{1}{\zeta_n^3}\varphi\Big(\frac{v-u_n}{\zeta_n}\Big)e^{-\gamma|v|^2}.
\end{equation}
The sequence $\{p_n\}_{n<+\infty}$ can be chosen such that $\inf_{n\not=n'<+\infty}|p_{n}-p_{n'}|>2$, to make sure different clouds are spatially separated at initial time. The other parameters and function $\varphi$ can also be properly chosen such that different clouds have different velocities. We first notice that $f^0$ in \eqref{eq:infinite_cloud} does not belong to the space $L_{\gamma,v}^{1}L_{x}^{\infty}$ in \eqref{eq:function_space_compare}
\begin{equation}
\sup_{x\in\mathds{R}^3}|f^0(x,v)|=\sum_{n<+\infty}\frac{1}{\zeta_n^3}\varphi\Big(\frac{v-u_n}{\zeta_n}\Big)e^{-\gamma|v|^2},\quad \int_{\mathds{R}^3}\sup_{x\in\mathds{R}^3}|f^0(x,v)|e^{\gamma|v|^2}dv=\int_{\mathds{R}^3}\sum_{n<+\infty}\frac{1}{\zeta_{n}^3}\varphi\Big(\frac{v-u_n}{\zeta_n}\Big)dv=+\infty.
\end{equation}
While we can verify $f^0$ in \eqref{eq:infinite_cloud} belongs to the space $\mathbf{L}_{\gamma,\delta}^1(\mathds{R}^3\times \mathds{R}^3)$ in Definition \ref{def:Ldeltasup} with $\delta<+\infty$, it is an example showing $\mathbf{L}_{\gamma,+\infty}^1(\mathds{R}^3\times \mathds{R}^3)\subsetneq \mathbf{L}_{\gamma,\delta}^1(\mathds{R}^3\times \mathds{R}^3)$.

\subsection{Statement of Kinetic Limit Result}

Theorem \ref{th:well-posedness} establishes the local well-posedness result of Boltzmann equation \eqref{eq:Boltzmann} with initial data $f^0\in \ldelta$, together with the weak stability estimate \eqref{eq:mixing_PDE_statement} at PDE level. In this subsection, we state the kinetic limit from hard sphere gas to Boltzmann equation, for the same class of singular initial data $f^0$ with certain natural additional conditions. In order to have finite mass, we work with initial data $f^0$ on the torus $\mathds{T}^3$. The space $\ldelta$ can be also defined on the torus $\mathds{T}^3$ with the norm $\lVert \cdot\rVert_{\ldelta(\mathds{T}^3\times \mathds{R}^3)}$
\begin{equation}\label{eq:L1_torus}
\lVert f\rVert_{\ldelta(\mathds{T}^3\times \mathds{R}^3)}:=\sup_{x\in\mathds{T}^3}\int_{\mathds{R}^3}[f]_{\delta}(x,v)e^{\gamma|v|^2}dv,\quad [f]_{\delta}(x,v):=\sup_{x'\in\mathds{T}^3:|x'-x|\leq \delta}|f(x',v)|.
\end{equation}
All the results in Theorem \ref{th:well-posedness} are still true if we change from full space $x\in\mathds{R}^3$ to torus $x\in\mathds{T}^3$, because a function on $\mathds{T}^3$ can be seen as a periodic function on $\mathds{R}^3$. Notice that on the torus, the space $\mathbf{L}_{\gamma,\delta}^1(\mathds{T}^3\times \mathds{R}^3)$ is equivalent to $\mathbf{L}_{\gamma,+\infty}^1(\mathds{T}^3\times \mathds{R}^3)$ when $\delta\not=0$. Therefore in the study of hard sphere gas, we assume $f^0\in \mathbf{L}_{\gamma,+\infty}^{1}(\mathds{T}^3\times \mathds{R}^3)$.

\begin{theorem}\label{th:kinetic_limit}
\textbf{\textup{[Kinetic Limit]}} Consider any parameter $\gamma>0$ and non-negative initial data $f^0\in \mathbf{L}_{\gamma,+\infty}^1(\mathds{T}^3\times \mathds{R}^3)$, and the hard sphere gas with initial data in \eqref{eq:grand_canonical_measure}. Under the Boltzmann-Grad scaling $\mu_{\eps}\eps^{2}\equiv 1$, the empirical measure $\pi_t^{\eps}$ converges in probability to the unique solution $f(t)$ with initial data $f^0$ of Boltzmann equation \eqref{eq:Boltzmann} derived in Theorem \ref{th:well-posedness}.

\textup{\textbf{[Weak Stability]}} Take any sequence $\{f_n^0\}_{n<+\infty}\subset \mathbf{L}_{\gamma,+\infty}^1(\mathds{T}^3\times \mathds{R}^3)$ of continuous-in-space non-negative initial data with uniformly bounded $\mathbf{L}_{\gamma,+\infty}^1(\mathds{T}^3\times \mathds{R}^3)$-norm, and any single continuous-in-space non-negative function $f^0\in \ldelta$. Assume $\{f_n^0\}$ and $f^0$ satisfy the initial weak convergence condition \eqref{assump:initial_weak_convergence} and the equi-continuity condition \eqref{assump:equi_continuity}. Then there exists a function $n(\eps):\mathds{R}_{\geq 0}\rightarrow \mathds{N}$ such that for any test function $\varphi(x,v)$, any $t\in [0,T_*]$
\begin{equation}
\lim_{\eps\rightarrow 0}\mathds{P}_{\eps,f_{n(\eps)}^0}\Big[\ \Big|\pi_t^{\eps}(\varphi)-\int f(t,x,v)\varphi(x,v)dvdx\Big|>\delta\ \Big]=0,
\end{equation}
where $\mathds{P}_{\eps,f_{n(\eps)}^{0}}$ is the probability measure of hard sphere gas with diameter $\eps$ and initial data $f_{n(\eps)}^0$.
\end{theorem}

\subsection{Idea of the Proof of Main Results}

Given initial data $f^0\in \ldelta$, the unique mild solution of Boltzmann equation \eqref{eq:Boltzmann} in Theorem \ref{th:well-posedness} will be constructed as the summation of a series-expansion. Each term in this series-expansion is indexed by a \textit{decorated binary tree}. To this end, we define several necessary combinatorial notions.

\begin{definition}\label{def:dec_binary_tree}
\textup{\textbf{[Decorated binary tree]}} Let $\T$ be a proper binary tree (each node has either $0$ or $2$ child nodes).
\vspace{-0.3cm}
\begin{itemize}
    \item Define $\mathcal{N}(\T)$ as the set of nodes of $\T$, and $\mathfrak{r}(\T)$ as the root, and $\mathcal{N}_{\new}(\T)$ as $\mathcal{N}(\T)\setminus\{\mathfrak{r}(\T)\}$. Define $\mathcal{B}(\T)$ as the set of branching nodes, and $\mathcal{L}(\T)$ as the set of leaves. For any non-root node $\n\in \nnew(\T)$, define $\mathcal{P}\mathfrak{n}$ as the parent node of $\n$.
    \vspace{-0.2cm}
    \item Furthermore, each branching node $\n\in\mathcal{B}(\T)$ is decorated as either "collision" or "overlap". Define $\mathcal{O}(\T)$ as the set of branching nodes with decoration "overlap". Define $\mathcal{B}^+(\T)$ as the set of branching nodes with decoration "collision", and $\mathcal{N}^+(\T)$ as $\mathcal{N}^+(\T):=\mathcal{B}^+(\T)\cup\mathcal{L}(\T)$. A collisional decorated tree is a decorated proper binary tree without any overlap nodes.
\end{itemize}
From now on, we may refer to decorated proper binary trees simply by 'decorated trees'.
\end{definition}

Now we can define the terms in the series expansion indexed by decorated trees, and the terms will be called as Boltzmann integrals. 

\begin{definition}\label{def:C_expansion}
\textbf{\textup{[Boltzmann Integral]}} Consider any function $h(t,x,v)$. The Boltzmann integrals $\mathbf{C}_{h}[\T]$ are defined recursively for all decorated trees $\T$. If $\T$ is just a single node $\mathcal{T}=\circ$, then $\mathbf{C}_{h}[\T](t)=h(t)$. If $\T$ has at least one branching node, then denote the left sub-tree of the root $\mathfrak{r}(\T)$ as $\T_l$, and the right sub-tree as $\T_r$
\begin{equation}\label{eq:expansion_term}
\mathbf{C}_{h}[\T](t)=
\left\{
\begin{array}{lr}
\int_0^t\mathbf{S}_{t-s}\mathbf{Q}^+\Big(\mathbf{C}_{h}[\T_l](s),\mathbf{C}_{h}[\T_r](s)\Big)ds     & \textup{ if the root of $\T$ is 'collision',} \\
\int_0^t\mathbf{S}_{t-s}\mathbf{Q}^-\Big(\mathbf{C}_{h}[\T_l](s),\mathbf{C}_{h}[\T_r](s)\Big)ds     & \textup{ if the root of $\T$ is 'overlap'.}
\end{array}
\right.
\end{equation}
Throughout the paper if we take $h(t,x,v)$ as $\big(\mathbf{S}_tf^0\big)(x,v)$ with $f^0$ as a time-independent function, we simply abbreviate $\mathbf{C}_{\mathbf{S}_{\cdot}f^0}[\mathcal{T}]$ as $\mathbf{C}_{f^0}[\T]$.
\end{definition}

To further motivate the definition of Boltzmann integrals, recall the mild formulation of Boltzmann equation in \eqref{eq:mild_formulation}. We can insert the formulation \eqref{eq:mild_formulation} into itself
\begin{equation}\label{eq:duhamel_itself_intro}
\begin{split}
f(t)=\mathbf{S}_tf^0+\int_0^t\mathbf{S}_{t-s}\mathbf{Q}\big(f(s),f(s)\big)ds&=\mathbf{S}_tf^0+\int_0^t\mathbf{S}_{t-s}\mathbf{Q}\big(\mathbf{S}_sf^0,f(s)\big)ds+\int_0^t\mathbf{S}_{t-s}\mathbf{Q}\Big(\int_0^s\mathbf{S}_{s-s_1}\mathbf{Q}(f(s_1),f(s_1))ds_1,f(s)\Big)ds\\
&=\mathbf{S}_tf^0+\int_0^t\mathbf{S}_{t-s}\mathbf{Q}\big(\mathbf{S}_sf^0,S_{s}f^0\big)ds+\int_0^t\mathbf{S}_{t-s}\mathbf{Q}\Big(\int_0^s\mathbf{S}_{s-s_1}\mathbf{Q}(f(s_1),f(s_1))ds_1,f(s)\Big)ds+...\\
&=\mathbf{C}_{f^0}[\circ](t)+\mathbf{C}_{f^0}[\CollisionFork](t)-\mathbf{C}_{f^0}[\OverlapFork](t)+\mathbf{C}_{f(\cdot)}\Big[\LeftBranchTree\Big](t)-\mathbf{C}_{f(\cdot)}\Big[\LeftOverlapTree\Big](t)+...
\end{split}
\end{equation}
Continuing to insert the mild formulation \eqref{eq:mild_formulation}, the solution $f(t)$ can be expressed as a summation of Boltzmann integrals $\mathbf{C}_{f^0}[\T]$, with some other Boltzmann integrals $\mathbf{C}_{f(\cdot)}[\T]$ involving the time-dependent function $f(\cdot)$ itself (and there are some other mixed terms). The Boltzmann integral in Definition \ref{def:C_expansion} is constructed to allow time-dependent functions $h$ allow us to control these 'error' terms and prove them converging to $0$ as we iterate more and more.

Under the notation of decorated trees and given $f^0\in \ldelta$, we aim to prove there exists $(\gamma_*>0,\delta_*>0,T_*>0)$ such that the following series expansion absolutely converges under the norm $L^{\infty}\ldeltastar([0,T_*]\times \mathds{R}^3\times \mathds{R}^3)$. Then we show it is the unique mild solution of the Boltzmann equation \eqref{eq:Boltzmann} with initial data $f^0$
\begin{equation}\label{eq:formal_summation}
f(t):=\sum_{\T:\textup{ decorated tree}}(-1)^{|\mathcal{O}(\T)|}\mathbf{C}_{f^0}[\T](t).
\end{equation}
The proof of Theorem \ref{th:well-posedness} is across Sections \ref{sec:combinatorial_objects}-\ref{sec:proof_of_main}. In this section we explain the core idea of the proof by a toy example. Suppose $f^0\in \ldelta$ is a space-independent function, which can have infinite $L^{\infty}$-norm according to the Definition \ref{def:Ldeltasup} of $\ldelta$. We can show after several collisions, the distribution becomes uniformly bounded. The following proposition only deals with space-independent initial data $f^0$ for the convenience of exposition, and the same result is true for any $f^0\in \ldelta$.

Across the paper, we use $C_{\gamma,\delta}$ as generic constants depending on $(\gamma,\delta)$, which may change term by term. We use $\mathtt{C}$ as generic constants that do not depend on $\gamma$ or $\delta$.

\begin{proposition}\label{prop:toy_example}
For any parameters $(\gamma>0,\delta>0)$, and any space-independent initial data $f^0\in \ldelta$ with $\lVert f^0\rVert_{\ldelta}\leq 1$, for any $(t,x,v)\in [0,+\infty)\times \mathds{R}^3\times \mathds{R}^3$ we have
\begin{equation}
\sup_{(x,v)\in \mathds{R}^3\times \mathds{R}^3}\Big|\mathbf{C}_{f^0}\Big[\DoubleBinaryTree\Big](t,x,v)\Big|<+\infty.
\end{equation}
\end{proposition}

\begin{proof}
Since $f^0$ is space-independent, the Boltzmann integral $\mathbf{C}_{f^0}[\DoubleBinaryTree]$ is also space-independent, and there exists a function $H:\mathds{R}\mapsto\mathds{R}_{\geq 0}$ such that
\begin{equation}
\mathbf{C}_{f^0}\Big[\DoubleBinaryTree\Big](t,x,v)=H(t)\mathbf{Q}^+\Big(\mathbf{Q}^+(f^0,f^0),\mathbf{Q}^+(f^0,f^0)\Big)(v).
\end{equation}
Consequently it is sufficient to prove
\begin{equation}\label{eq:QQQ_Linfty}
\sup_{v\in\mathds{R}^3}\Big|\mathbf{Q}^+\Big(\mathbf{Q}^+(f^0,f^0),\mathbf{Q}^+(f^0,f^0)\Big)(v)\Big|<+\infty.
\end{equation}
First expand the collision integral according to its definition in \eqref{eq:gain_and_loss_decomposition}
\begin{equation}
\begin{split}
\mathbf{Q}^+\Big(\mathbf{Q}^+\big(f^0,f^0\big),\mathbf{Q}^+\big(f^0,f^0\big)\Big)(v)&=\int \mathbf{Q}^+\big(f^0,f^0\big)(v')\mathbf{Q}^+\big(f^0,f^0\big)(v_*')\big((v_*-v)\cdot\omega\big)_+d\omega dv_*
\end{split}
\end{equation}
The Carleman change of variables in Lemma \ref{lem:Carleman_change_of_variables} allows us to change $(\omega,v_*)$ into $(v',v_*')$
\begin{equation}\label{eq:QQQ_1}
\mathbf{Q}^+\Big(\mathbf{Q}^+\big(f^0,f^0\big),\mathbf{Q}^+\big(f^0,f^0\big)\Big)(v)=\int \mathbf{Q}^+\big(f^0,f^0\big)(v')\mathbf{Q}^+\big(f^0,f^0\big)(v_*')|v'-v|^{-1}d\pi_{v+(v'-v)^{\perp}}(v_*')dv',
\end{equation}
where $v'$ is integrated as a free variable, and $v_*'$ falls on the orthogonal plane $v+(v'-v)^{\perp}$ determined by $v$ and $v'$
\begin{equation}\label{eq:orthogonal_plane_def}
v+(v'-v)^{\perp}:=\Big\{v_*'\in\mathds{R}^3\Big|\ (v'-v)\perp(v_*'-v)\Big\},\quad \textup{$d\pi_{\mathtt{PL}}$ is the uniform measure on the plane $\mathtt{PL}$.}
\end{equation}
A crucial observation is that there exists a parameter $\gamma(0)>0$ and a constant $C_{\gamma}$ such that
\begin{equation}\label{eq:plane_homogenization_toy_example}
\int \mathbf{Q}^+(f^0,f^0)(v)d\pi_{\mathtt{PL}}(v)\leq C_{\gamma}\int \mathbf{Q}^+(M_{\gamma(0)},M_{\gamma(0)})(v)d\pi_{\mathtt{PL}}(v),\quad \textup{for any plane $\mathtt{PL}\subset \mathds{R}^3$}.\tag{PL}
\end{equation}
Since \eqref{eq:plane_homogenization_toy_example} is true for any plane $\mathtt{PL}\subset \mathds{R}^3$, it is of course also true when $\mathtt{PL}=v+(v'-v)^{\perp}$. Equation \eqref{eq:plane_homogenization_toy_example} will be proved later, and now we show \eqref{eq:plane_homogenization_toy_example} actually proves \eqref{eq:QQQ_Linfty}, and therefore concludes the proof of Proposition \ref{prop:toy_example}. Apply \eqref{eq:plane_homogenization_toy_example} to \eqref{eq:QQQ_1}
\begin{equation}
\begin{split}
\mathbf{Q}^+\Big(\mathbf{Q}^+\big(f^0,f^0\big),\mathbf{Q}^+\big(f^0,f^0\big)\Big)(v)&\leq C_{\gamma}\int \mathbf{Q}^+\big(f^0,f^0\big)(v')\mathbf{Q}^+\big(M_{\gamma(0)},M_{\gamma(0)}\big)(v_*')|v'-v|^{-1}d\pi_{v+(v'-v)^{\perp}}(v_*')dv'\\
&=C_{\gamma}\mathbf{Q}^+\Big(\mathbf{Q}^+\big(f^0,f^0\big),\mathbf{Q}^+\big(M_{\gamma(0)},M_{\gamma(0)}\big)\Big)(v).
\end{split}
\end{equation}
Until now two singular initial data in \eqref{eq:QQQ_Linfty} have been \textit{homogenized} into Maxwellians $M_{\gamma(0)}$, and the same argument applies to the two remaining $f^0$
\begin{equation}\label{eq:toy_replacement}
\mathbf{Q}^+\Big(\mathbf{Q}^+\big(f^0,f^0\big),\mathbf{Q}^+\big(f^0,f^0\big)\Big)(v)\leq C_{\gamma}\mathbf{Q}^+\Big(\mathbf{Q}^+\big(M_{\gamma(0)},M_{\gamma(0)}\big),\mathbf{Q}^+\big(M_{\gamma(0)},M_{\gamma(0)}\big)\Big)(v).
\end{equation}
This proves \eqref{eq:QQQ_Linfty} and therefore Proposition \ref{prop:toy_example}.

Now we prove \eqref{eq:plane_homogenization_toy_example}. Expand the LHS of \eqref{eq:plane_homogenization_toy_example}, and apply the change of variables in Lemma \ref{lem:Carleman_change_of_variables}
\begin{equation}\label{eq:plane_toy_2}
\begin{split}
\textup{LHS of \eqref{eq:plane_homogenization_toy_example}}&=\int \big((v-v_*)\cdot\omega\big)_+f^0(v')f^0(v_*')d\omega dv_*d\pi_{\mathtt{PL}}(v)\\
&=\int f^0(v')f^0(v_*')|v'-v|^{-1}d\pi_{v+(v'-v)^{\perp}}(v_*')dv'd\pi_{\mathtt{PL}}(v).
\end{split}
\end{equation}
The core idea is that we can transfer the degrees of freedom of $v\in \mathtt{PL}$ onto the variables $(v',v_*')$ by Lemma \ref{lem:geometric_change_of_variable}, recovering $(v',v_*')$ as free variables
\begin{equation}
|v'-v|^{-1}d\pi_{v+(v'-v)^{\perp}}(v_*')dv'd\pi_{\mathtt{PL}}(v)=\mathtt{C}\overline{d\pi}_{\mathtt{S}_{v',v_*'}\cap\mathtt{PL}}(v)dv_*'dv'
\end{equation}
See also Figure \ref{fig:intersection}. Here $\overline{d\pi}_{\mathtt{S}_{v',v_*'}\cap\mathtt{PL}}$ is the normalized measure on the circle (or point) $\mathtt{S}_{v',v_*'}\cap\mathtt{PL}$. 
\begin{figure}[h]
    \centering
    \includegraphics[width=0.4\linewidth]{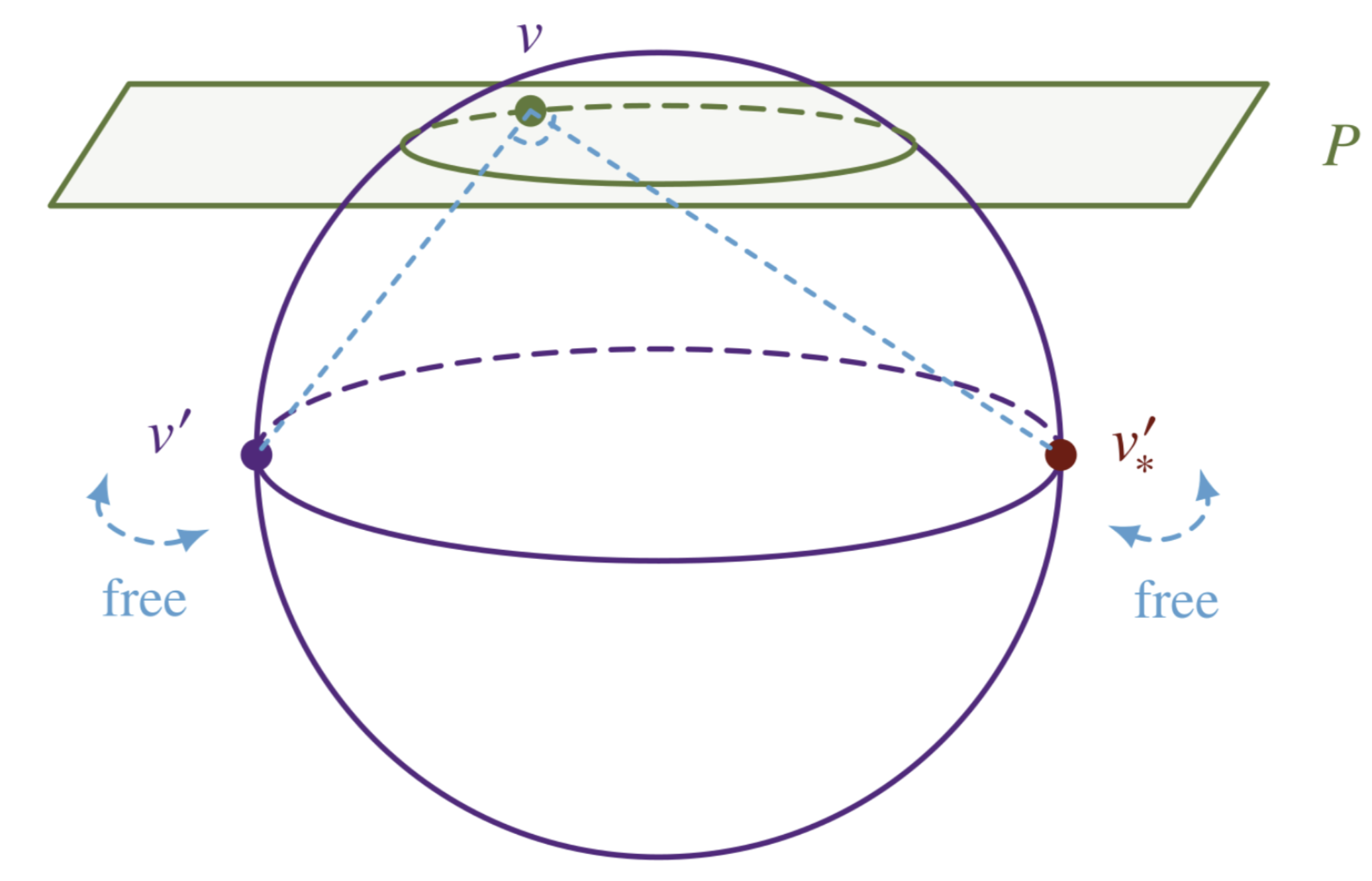}
    \caption{This figure is an illustration of Lemma \ref{lem:geometric_change_of_variable}. Originally we first integrate $v$ over a plane $\mathtt{PL}$, then integrate $v'$ freely in $\mathds{R}^3$, and finally integrate $v_*'$ on the orthogonal plane $v+(v'-v)^{\perp}$. This orthogonality implies $(v,v',v_*')$ lie on the same sphere, actually the collision sphere $\mathtt{S}_{v',v_*'}$ in \eqref{eq:collision_sphere}. Therefore we can change the order of integrations: we first integrate $v'$ and $v_*'$ freely in $\mathds{R}^3$, then integrate $v$ over the intersection $\mathtt{S}_{v',v_*'}\cap\mathtt{PL}$.}
    \label{fig:intersection}
\end{figure}

Apply the change of variables in Lemma \ref{lem:geometric_change_of_variable} to \eqref{eq:plane_toy_2}
\begin{equation}
\begin{split}
\textup{LHS of \eqref{eq:plane_homogenization_toy_example}}&=\mathtt{C}\int f^0(v')f^0(v_*')\mathds{1}_{\mathtt{S}_{v',v_*'}\cap \mathtt{PL}\not=\varnothing}\overline{d\pi}_{\mathtt{S}_{v',v_*'}\cap\mathtt{PL}}(v)dv_*'dv'\\
&=\mathtt{C}\int f^0(v')f^0(v_*')\mathds{1}_{\mathtt{S}_{v',v_*'}\cap\mathtt{PL}\not=\varnothing}dv_*'dv'\\
&\leq C_{\gamma}\int M_{\gamma(0)}(v')M_{\gamma(0)}(v_*')\mathds{1}_{\mathtt{S}_{v',v_*'}\cap\mathtt{PL}\not=\varnothing}dv_*'dv'=\textup{RHS of \eqref{eq:plane_homogenization_toy_example}}.
\end{split}
\end{equation}
This proves \eqref{eq:plane_homogenization_toy_example}, and therefore proves Proposition \ref{prop:toy_example}.
\end{proof}

Proposition \ref{prop:toy_example} is also heuristically explained by Figure \ref{fig:eight_figures}: the regularity of velocity distribution is improved by collision, and enough collision makes the distribution uniformly bounded, even if the initial data $f^0$ can be arbitrarily close to Dirac masses in velocity. In sections \ref{sec:combinatorial_objects}-\ref{sec:proof_of_main}, we develop a \textit{homogenization procedure} which allows us to establish inequalities like \eqref{eq:toy_replacement} for all decorated trees $\T$ and their Boltzmann integrals $\mathbf{C}_{f^0}[\mathcal{T}]$.

\paragraph{Constants.} Across the paper, $\gamma(0),\gamma_*,\delta_*,T_*,\ell$ are small enough constants determined by $\gamma$ and $\delta$, and $\mathtt{K}$ is a large enough constant determined by $(\gamma,\delta)$.

\section{Series Expansion and Homogenization Operations}\label{sec:combinatorial_objects}

Our goal is to establish a proper $L^{\infty}\ldelta$-estimate for the Boltzmann integrals $\mathbf{C}_{f^0}[\mathcal{T}]$ in Definition \ref{def:C_expansion}, such that the series expansion in \eqref{eq:formal_summation} absolutely converges in $\timeldelta$. Lemma \ref{lem:transport_loss} shows that the transport semi-group map $f^0\in \ldelta$ into space $\ldeltastar$ with degeneracies in parameters.

Across the paper, we frequently need the notion of small cube in the configuration space $\mathds{R}^3$ of velocities. Given any $\ell>0$ and any $\mathtt{p}\in \mathds{R}^3$, define the cube $Q_{\mathtt{p}}^{\ell}$ as
\begin{equation}\label{eq:cube_definition}
Q_{\mathtt{p}}^{\ell}:=\Big\{v\in\mathds{R}^3\Big|\ |v-\mathtt{p}|_{\infty}\leq \frac{\ell}{2}\Big\}.
\end{equation}

\begin{lemma}\label{lem:transport_loss}
Consider arbitrary $(\gamma>0,\delta>0)$ and any initial data $f^0\in \ldelta$. For any parameters $(T,\gamma_*,\delta_*)$ with 
\begin{equation}
T>0,\quad 0<\gamma_*<\gamma,\quad 0<\delta_*<\delta,
\end{equation}
we have the following inequality
\begin{equation}
\lVert \mathbf{S}_tf^0\rVert_{L^{\infty}\ldeltastar([0,T]\times \mathds{R}^3\times \mathds{R}^3)}\leq  C_{\gamma}\Big[1+\Big(\frac{T}{\delta-\delta_*}\Big)^3\Big(\frac{1}{\gamma-\gamma_*}\Big)^{\frac{3}{2}}\Big]\lVert f^0\rVert_{\ldelta}.
\end{equation}
\end{lemma}

\begin{proof}
For arbitrary $t\in [0,T]$, compute the norm $\lVert \mathbf{S}_tf^0\rVert_{\ldeltastar}$
\begin{equation}\label{eq:transport_norm_1}
\lVert \mathbf{S}_tf^0\rVert_{\ldeltastar}=\lVert f^0(x-tv,v)\rVert_{\ldeltastar}=\sup_{x\in\mathds{R}^3}\Big(\int_{\mathds{R}^3} \sup_{x':|x'-x|\leq \delta_*}\Big|f^0(x'-tv,v)\Big|e^{\gamma_*|v|^2}dv\Big)
\end{equation}
Choose parameter $\ell$ as $\ell=\frac{\delta-\delta_*}{2T}$. Decompose the integral in the RHS of \eqref{eq:transport_norm_1} into small cubes $Q_{\mathtt{p}}^{\ell}$
\begin{equation}\label{eq:transport_norm_2}
\int_{\mathds{R}^3} \sup_{x':|x'-x|\leq \delta_*}\Big|f^0(x'-tv,v)\Big|e^{\gamma_*|v|^2}dv=\sum_{\mathtt{p}\in \ell\mathds{Z}^3}\int_{Q_{\mathtt{p}}^{\ell}} \sup_{x':|x'-x|\leq \delta_*}\Big|f^0(x'-tv,v)\Big|e^{\gamma_*|v|^2}dv
\end{equation}
Since the velocity $v$ is only perturbed within a small cube $Q_{\mathtt{p}}^{\ell}$, the perturbation of initial position $x'-tv$ is also small enough
\begin{equation}
\sup_{x':|x'-x|\leq \delta_*}\sup_{v\in Q_{\mathtt{p}}^{\ell}}\Big|\Big(x'-tv\Big)-\Big(x-t\mathtt{p}\Big)\Big|\leq \sup_{x':|x'-x|\leq \delta_*}|x'-x|+2\sup_{v\in Q_{\mathtt{p}}^{\ell}}t|v-\mathtt{p}|_{\infty}\leq \delta_*+2T\ell= \delta
\end{equation}
Therefore the RHS of \eqref{eq:transport_norm_2} can be bounded from above by
\begin{equation}
\begin{split}
\sum_{\mathtt{p}\in \ell\mathds{Z}^3}\int_{Q_{\mathtt{p}}^{\ell}} \sup_{x':|x'-x|\leq \delta_*}\Big|f^0(x'-tv,v)\Big|e^{\gamma_*|v|^2}dv&\leq \sum_{\mathtt{p}\in \ell\mathds{Z}^3}\int_{Q_{\mathtt{p}}^{\ell}} [f^0]_{\delta}(x-t\mathtt{p},v)e^{\gamma_*|v|^2}dv\\
&=\sum_{\mathtt{p}\in \ell\mathds{Z}^3}\int_{Q_{\mathtt{p}}^{\ell}} [f^0]_{\delta}(x-t\mathtt{p},v)e^{\gamma|v|^2}e^{-(\gamma-\gamma_*)|v|^2}dv\\
&\leq C_{\gamma}\sum_{\mathtt{p}\in \ell\mathds{Z}^3}e^{-\frac{\gamma-\gamma_*}{4}|\mathtt{p}|^2}\int_{Q_{\mathtt{p}}^{\ell}}[f^0]_{\delta}(x-t\mathtt{p},v)e^{\gamma|v|^2}dv.
\end{split}
\end{equation}
The condition of $f^0\in\ldelta$ implies the following inequality uniform for parameters $(x,t,\mathtt{p})$
\begin{equation}
\int_{Q_{\mathtt{p}}^{\ell}}[f^0]_{\delta}(x-t\mathtt{p},v)e^{\gamma|v|^2}dv\leq \lVert f^0\rVert_{\ldelta}.
\end{equation}
Consequently
\begin{equation}
\begin{split}
\lVert \mathbf{S}_{t}f^0\rVert_{\timeldeltastar}\leq C_{\gamma}\lVert f^0\rVert_{\ldelta}\sum_{\mathtt{p}\in\ell\mathds{Z}^3}e^{-\frac{\gamma-\gamma_*}{4}|\mathtt{p}|^2}&\leq C_{\gamma}\lVert f^0\rVert_{\ldelta}\Big[1+\ell^{-3}\Big(\frac{1}{\gamma-\gamma_*}\Big)^{\frac{3}{2}}\Big]\\
&=C_{\gamma}\lVert f^0\rVert_{\ldelta}\Big[1+\Big(\frac{T}{\delta-\delta_*}\Big)^3\Big(\frac{1}{\gamma-\gamma_*}\Big)^{\frac{3}{2}}\Big].
\end{split}
\end{equation}
\end{proof}

Lemma \ref{lem:transport_loss} estimates the Boltzmann integral $\mathbf{C}_{f^0}[\circ]$ for the trivial tree $\T=\circ$.

\subsection{Carleman Representations}

In this subsection, we define several combinatorial objects necessary for the study of the Boltzmann integrals for general decorated tree $\T$. Then in the spirit of Carleman's change of variables (see Lemma \ref{lem:Carleman_change_of_variables}), we develop a representation of $\mathbf{C}_{h}[\mathcal{T}]$ that is convenient for the proof of the main results.

\begin{figure}[h]
    \centering
    \includegraphics[width=0.7\linewidth]{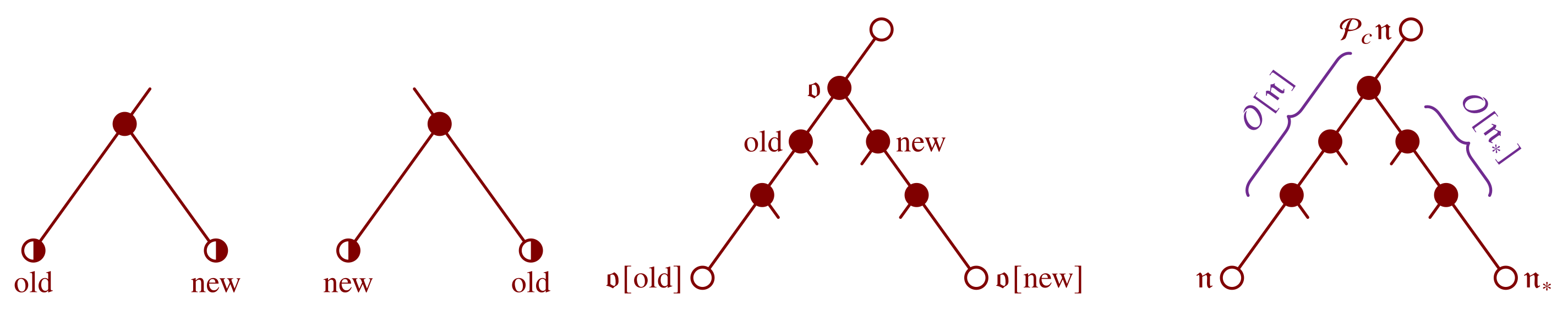}
    \caption{This is an illustration of several concepts in Definition \ref{def:deleting}. The left-most and middle-left parts classify the two child nodes of an overlap node $\o$ into 'old' and 'new'. The middle-right part illustrates the notion of $\oold$, $\onew$. The right-most part illustrates the notion of $\mathcal{O}[\n]$ and $\mathcal{P}_c\n$.}
    \label{fig:old_and_new}
\end{figure}

\begin{definition}\label{def:deleting}
(i) Given a decorated tree $\T$, for each overlap node $\o\in\mathcal{O}(\mathcal{T})$, the two child nodes of $\o$ are separately \underline{labeled as 'old' and 'new'}. The labeling rule is given in Figure \ref{fig:old_and_new}. If the 'old' child node of $\o$ is also an overlap node, then we further consider its 'old' child node. Iterate this process until we get a collision node or a leaf, and we denote this node as $\oold$. If the 'new' child node of $\o$ is also an overlap node, then we further consider its 'old' child node. Iterate this process until we get a collision node or a leaf, and we denote this node as $\onew$.

(ii) For an overlap node $\o$ within the decorated tree $\T$, the operation of \underline{'deleting overlap node $\o$'} is defined as the following (see also Figure \ref{fig:delete_overlap_nodes}): If we delete this node $\o$, then we get two decorated trees. The first decorated tree is denoted as $\mathscr{T}_{\o}$, which is the sub-tree with the 'new' child node of $\o$ as the root. The second decorated tree  is $\T\setminus \mathscr{T}_{\o}$ without node $\o$ (in this sense $\o$ is deleted). Sometimes we call $\mathscr{T}_{\o}$ as the sub-tree branched from $\o$. 

\begin{figure}[h]
    \centering
    \includegraphics[width=1\linewidth]{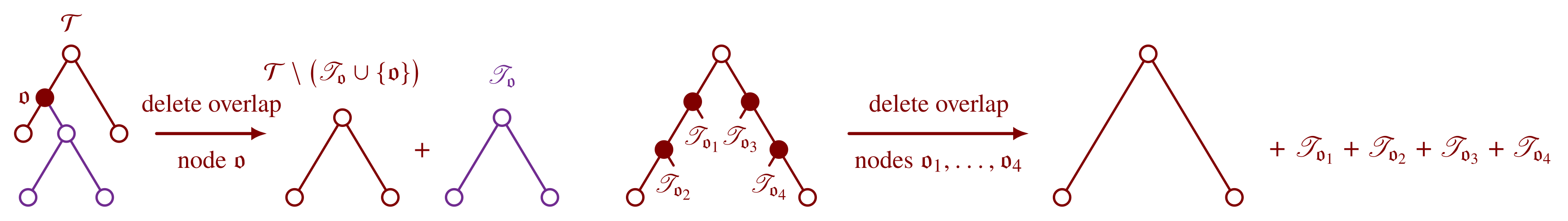}
    \caption{}
    \label{fig:delete_overlap_nodes}
\end{figure}

(iii) Given a decorated tree $\T$, delete all the overlap nodes and we get a unique (up to permutations of indexes) sequence $\{\T_i^+\}_{0\leq i\leq m}$ of collisional decorated trees (see Figure \ref{fig:GT}). For two nodes (either leaves or collision nodes) $\n$ and $\m$ within $\T$, if they belong to the same $\T_i^+$ and $\m$ is the parent node of $\n$ in $\T_i^+$, then we say $\m$ is the \underline{collision-parent} node of $\n$ in $\T$, and $\n$ is a collision-child node of $\m$. We use $\mathcal{P}_c\n$ to denote the collision-parent of $\mathfrak{n}$. Define $\mathcal{O}[\n]$ as (see also the right-most part of Figure \ref{fig:old_and_new})
\begin{equation}\label{eq:def_of_On}
\mathcal{O}[\mathfrak{n}]:=\Big\{\o\in \mathcal{O}(\T)\Big|\ \oold=\mathfrak{n}\Big\}.
\end{equation}
Notice that when $\mathfrak{n}$ has a collision-parent node $\mathcal{P}_c\n$, the set $\mathcal{O}[\n]$ is just the set of overlap nodes on the path from $\n$ to $\mathcal{P}_c\n$, and in this case we call $\mathcal{P}_c\n$ as the \underline{collision-parent} node of any $\o\in \mathcal{O}[\n]$.

\end{definition}

With the notion of $\mathcal{O}[\n]$, we can decompose the set $\mathcal{O}(\T)$ of overlap nodes
\begin{equation}
\mathcal{O}(\T)=\bigcup_{\n\in \mathcal{N}^+(\mathcal{T})}\mathcal{O}[\n],\quad \mathcal{O}[\n]\cap\mathcal{O}[\n']=\varnothing\textup{ if $\n\not=\n'$}.
\end{equation}
Given an arbitrary set of nodes $\{\n_1,...,\n_k\}\subset \mathcal{N}^+(\T)$, define $\mathcal{O}[\n_1,...,\n_k]$ as
\begin{equation}\label{eq:O_sequence_def}
\mathcal{O}[\n_1,...,\n_k]:=\bigcup_{1\leq i\leq k}\mathcal{O}[\n_i].
\end{equation}

\begin{definition}\label{def:tree_decompose}
Given a decorated tree $\T$, recall the unique sequence $\{\T_i^+\}_{0\leq i\leq m}$ of collisional decorated trees derived in Definition \ref{def:deleting} by deleting all overlap nodes of $\T$.  We \underline{define a directed graph $\mathcal{G}_{\T}$} related to the decorated tree $\T$:
\begin{itemize}
    \item The nodes of $\mathcal{G}_{\T}$ are $\{0,1,...,m\}$.
    \item For two different decorated trees $\T_{i_1}^+$ and $\T_{i_2}^+$, if in the original decorated tree $\T$ they are connected by an overlap node $\o$ with the root $\mathfrak{r}(\T_{i_2}^+)$ as $\mathfrak{r}(\T_{i_2}^+)=\onew$, then we draw a directed edge from $i_1$ to $i_2$ in graph $\mathcal{G_T}$.
\end{itemize}
The collisional decorated tree represented by the root of $\mathcal{G}_{\mathcal{T}}$ is called the primary sub-tree of $\mathcal{T}$, and denoted as $\tprime$. An illustration of this definition is given in Figure \ref{fig:GT}.
\end{definition}

\begin{figure}[h]
    \centering
    \includegraphics[width=0.8\linewidth]{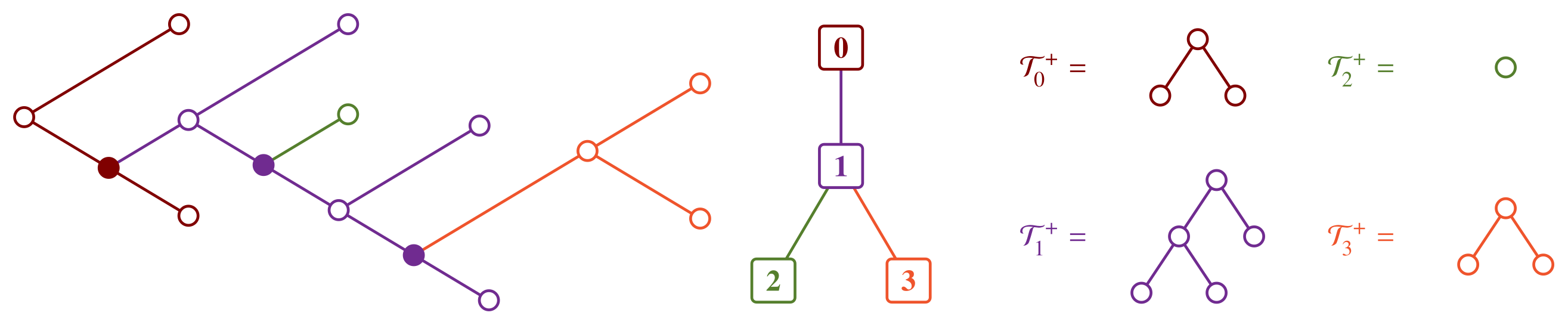}
    \caption{This is an example of the directed graph $\mathcal{G_T}$ constructed in Definition \ref{def:tree_decompose}. The left part is the original decorated tree $\T$, the middle part is the directed graph $\mathcal{G_T}$, and the right part shows the collisional decorated trees $\{\T_i^+\}_{0\leq i\leq 3}$.}
    \label{fig:GT}
\end{figure}

Now we explain several notions that will be used frequently. Given a variable $a$ and a set $\mathcal{S}$, we use $[a]_{\mathcal{S}}$ to refer to a sequence of variables indexed by $\mathcal{S}$
\begin{equation}\label{eq:bracket_notion}
[a]_{\mathcal{S}}:=\{a_{\mathtt{s}}\}_{\mathtt{s}\in \mathcal{S}}.
\end{equation}
Based on the notion in \eqref{eq:bracket_notion}, given an arbitrary decorated tree $\mathcal{T}$, we will define $d[v]_{\nnew^+(\T)}$, $d[v]_{\mathcal{N}^+(\T)}$, and $d[s]_{\mathcal{B}(\T)}$ as measures of variables $[v]_{\nnew^+(\T)}$, $[v]_{\mathcal{N}^+(\T)}$, and $[s]_{\mathcal{B}(\T)}$. Notice that we define the measures $d[v]_{\nnew^+(\T)}$ and $d[v]_{\mathcal{N}^+(\T)}$ because in the collision integral $\mathbf{C}_{h}[\mathcal{T}]$, the overlap node itself does not have a velocity variable associated with it. Only collision nodes $\mathcal{B}^+(\T)$ and leaves $\mathcal{L}(\T)$ have their own velocity variables. 

For a decorated tree $\T$, the measure $d[v]_{\mathcal{N}^+(\T)}$ is defined as
\begin{equation}
d[v]_{\mathcal{N}^+(\T)}:=d[v]_{\nnew^+(\T)}dv_{\mathfrak{r}(\T)},\quad \textup{ $v_{\mathfrak{r}(\T)}$ is a free variable, and $d[v]_{\nnew^+(\T)}$ is a measure determined by $\T$ and $v_{\mathfrak{r}(\T)}$.}
\end{equation}
Now given a decorated tree $\T$ and $v_{\mathfrak{r}(\T)}$, we define the measure $d[v]_{\nnew^+(\T)}$. Recall the sequence $\{\T_i^+\}_{0\leq i\leq m}$ of collisional decorated trees derived in Definition \ref{def:deleting} by deleting all overlap nodes $\mathcal{O}(\T)$, we will first define $d[v]_{\nnew^+(\T_i^+)}$, and then use them to construct $d[v]_{\nnew^+(\T)}$.

\begin{figure}[h]
    \centering
    \includegraphics[width=0.7\linewidth]{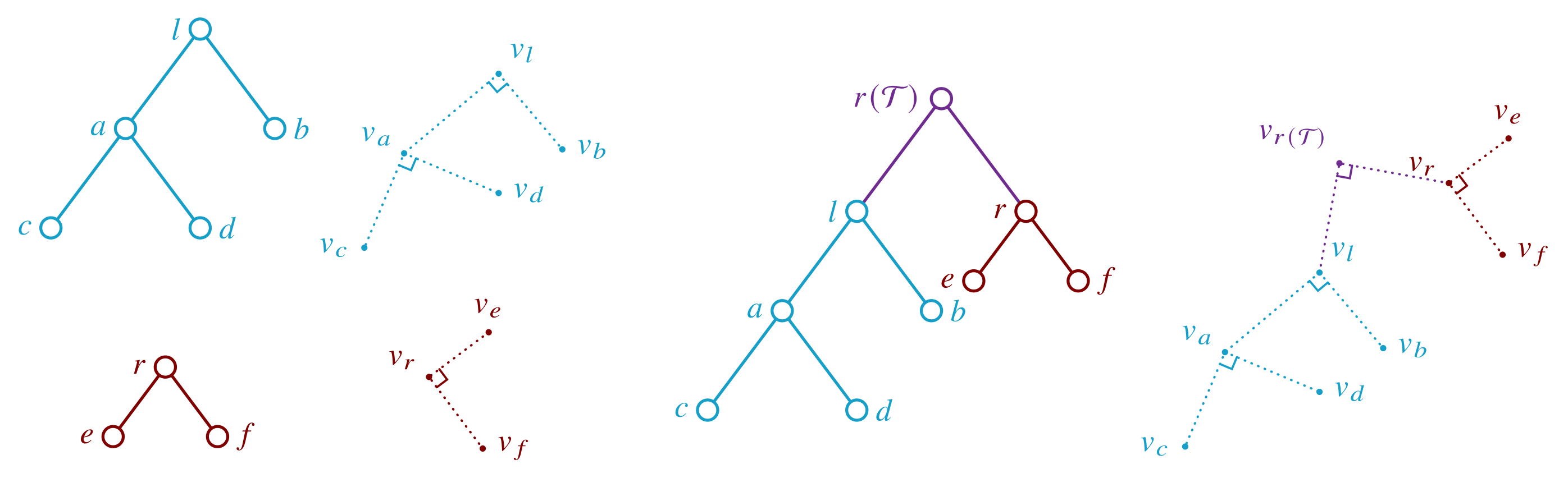}
    \caption{This is an illustration of the recursive construction of $d[v]_{\nnew^+(\T)}$ in Definition \ref{def:carleman_measure}. The trees with solid line are the decorated trees, and the trees with dotted line represent the geometry between different velocity variables. The decorated tree in the middle is $\T$, with two sub-trees $\T_l$ (the blue one) and $\T_r$ (the red one). For arbitrary $\n\in \mathcal{N}(\T)$ and its two collision-child nodes $\n_1$ and $\n_2$, we always have $(v_{\n}-v_{\n_1})\perp (v_{\n}-v_{\n_2})$.}
    \label{fig:velocity_tree}
\end{figure}

\begin{definition}\label{def:carleman_measure}
\textbf{\textup{[Carleman measure]}} Take a collisional decorated tree $\T$, and we define the measure $d[v]_{\nnew^+(\T)}$ recursively. If $\T$ is a single node $\T=\circ$, then $\nnew^+(\T)=\varnothing$ and there is nothing to integrate. If the root $\mathfrak{r}(\T)$ is a branching node, then denote $\T_l$ as the left sub-tree of $\mathfrak{r}(\T)$ and $\T_r$ as the right sub-tree of $\mathfrak{r}(\mathcal{T})$
\begin{equation}
d[v]_{\nnew^+(\T)}:=d[v]_{\nnew^+(\T_l)}d[v]_{\nnew^+(\T_r)}|v_l-v|^{-1}d\pi_{v+(v_l-v)^{\perp}}(v_r)dv_l,\quad v_l:=v_{\mathfrak{r}(\T_l)},\quad v_r:=v_{\mathfrak{r}(\T_r)}.
\end{equation}
Take an arbitrary decorated tree $\T$, with the sequence of collisional decorated trees $\{\T_{i}^+\}_{0\leq i\leq m}$ in Definition \ref{def:tree_decompose}. Define $d[v]_{\mathcal{N}_{\new}^+(\T)}$ as
\begin{equation}\label{eq:carleman_measure_of_velocities}
d[v]_{\mathcal{N}_{\new}^+(\T)}:=\Big(\prod_{\o\in \mathcal{O}(\T)}\pi|v_{\oold}-v_{\onew}|\Big)d[v]_{\mathcal{N}_{\new}(\T_{0}^+)}\prod_{1\leq i\leq m}d[v]_{\mathcal{N}(\T_i^+)}.
\end{equation}
Given any decorated tree $\T$, we define the measure $d[s]_{\mathcal{B}(\T)}$ recursively. If $\T$ is a single node $\T=\circ$, then $\mathcal{B}(\T)=\varnothing$ and there is nothing to integrate. If the root $\mathfrak{r}(\T)$ is a branching node, then denote $\T_l$ as the left sub-tree of $\mathfrak{r}(\mathcal{T})$ and $\T_r$ as the right sub-tree
\begin{equation}\label{eq:carleman_time_measure_def}
d[s]_{\mathcal{B}(\T)}:=\mathds{1}_{s_l\leq s_{\mathfrak{r}(\T)}}\mathds{1}_{s_r\leq s_{\mathfrak{r}(\T)}}d[s]_{\mathcal{B}(\T_l)}d[s]_{\mathcal{B}(\T_r)}ds_{\mathfrak{r}(\T)},\quad s_{l}:=s_{\mathfrak{r}(\T_l)},\quad s_r:=s_{\mathfrak{r}(\T_r)}.
\end{equation}
\end{definition}

\begin{definition}\label{def:overlap_parent}
The notion of collision-parent node has been defined in Definition \ref{def:deleting} (iii). We can also define the notion of \underline{overlap-parent node}: for a node $\n\in \mathcal{N}(\T)$ and an overlap node $\o\in \mathcal{O}(\T)$, we say $\o$ is the overlap-parent node of $\n$ if $\o[\textup{new}]=\n$. The notions of $\onew$ and $\oold$ are defined in Definition \ref{def:deleting} (i). Given two nodes $\n_1$ and $\n_2$, if $\n_1$ is the collision-parent node or the overlap-parent node of $\n_2$, we say $\n_1$ is the \underline{CO-parent node} of $\n_2$.

Given a node $\n$, if it has a collision-parent then we denote it as $\mathcal{P}_{c}\n$, if it has an overlap-parent node then we denote it as $\mathcal{P}_{o}\n$. Notice that a node can not have simultaneously a collision-parent node and an overlap-parent node.  Given a node $\n$, if it has a collision-parent or an overlap-parent then we denote it as $\mathcal{P}_{co}\n$.
\end{definition}

\begin{remark}\label{rmk:pseudo-trajectory_and_xfunction}
\textup{With the notions of combinatorial objects in Definition \ref{def:deleting}-\ref{def:overlap_parent}, we can give them an intuitional interpretation using the notion of \underline{pseudo-trajectory}. For each collision node $\n$, we use $\n[\textup{new}]$ and $\n[\textup{old}]$ to denote its two collision-child nodes (the order does not matter here). Given a decorated tree $\T$, a configuration $(t,x)$, velocities $[v]_{\mathcal{N}^+(\T)}$, and $[s]_{\mathcal{B}(\T)}$, we can construct the corresponding pseudo-trajectory and the \underline{position functions} $\{\mathbf{x}_{\n}\}_{\n\in \mathcal{N}(\T)}$. }

\begin{itemize}
    \item \textup{First we create a particle at time $t$ and position $x$ with velocity $v_{\mathfrak{r}(\T)}$.}
    \item \textup{If the root $\mathfrak{r}(\T)$ is not a branching node, then the particle travels backwards in time from $x$ to $\x_{\mathfrak{r}(\T)}:=x-tv_{\mathfrak{r}(\T)}$ in straight line with velocity $v_{\mathfrak{r}(\T)}$. If the root $\mathfrak{r}(\T)$ is a collision/overlap node then we label the particle as $\mathfrak{r}(\T)$/$\mathfrak{r}(\T)[\textup{old}]$, we run the next step iteratively.} 
    \item \textup{Let the set of particle(s) travel backwards in time until we arrive at another interaction time $s_{\n}$. If $\n$ is a collision node, then the velocity of particle with label $\n$ change from $v_{\n}$ to $v_{\n[\textup{old}]}$, and we relabel this particle by $\n[\textup{old}]$. We define $\x_{\n}$ as the current position of this node We also create a new particle at position $\x_{\n}$ with velocity $v_{\n[\textup{new}]}$, and label this particle by $\n[\textup{new}]$. If $\n$ is an overlap node, then we define $\x_{\n}$ as the current position of particle $\n[\textup{old}]$, and we create a new particle at position $\x_{\n}$ with velocity $v_{\n[\textup{new}]}$, and label it as $\n[\textup{new}]$.}
    \item \textup{At time $0$, for each $\l\in \mathcal{L}(\T)$, we define $\mathbf{x}_{\l}$ as the current position of the particle labelled by $\l$.}
\end{itemize}

\end{remark}
Under the notion of pseudo-trajectory, each branching node $\n\in \mathcal{B}(\T)$ represents an interaction with its time and space. Going backwards in time, in a collision the old particle changes its velocity, and we create a new particle; in an overlap the old particle does not change its velocity, and we create a new particle. This notion is helpful for understanding the proof in this paper, since many computations of Boltzmann integrals can be understood as modifying the corresponding pseudo-trajectories.

\begin{figure}[h]
    \centering
    \includegraphics[width=0.6\linewidth]{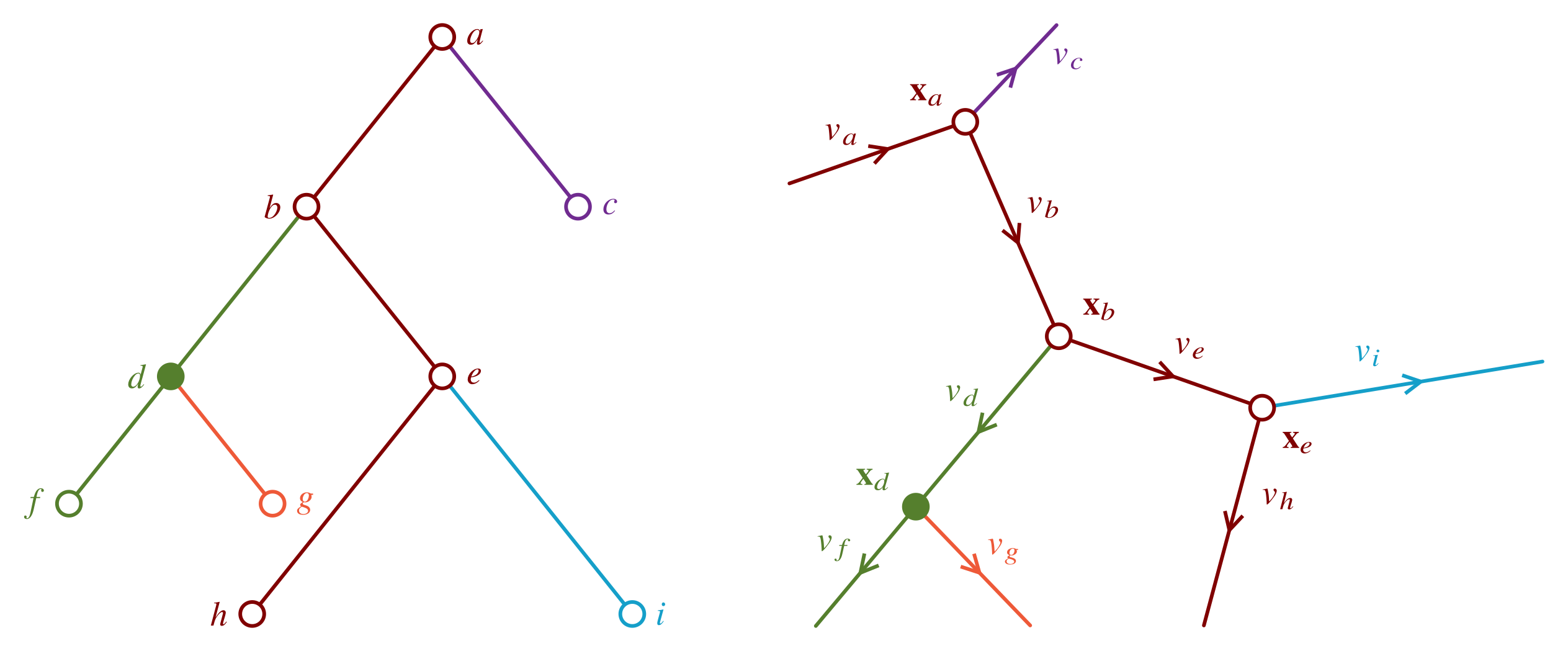}
    \caption{An example of decorated and its pseudo-trajectory.}
    \label{fig:pseudo-trajectory}
\end{figure}

Sometimes when we study the value of a Boltzmann integral at $(t,x)$, we use the following notions for convenience: Write $s_{\l}=0$ for $\l\in \mathcal{L}(\T)$. Define $\diamond$ as a virtual node, in the sense that if a node $\n$ does not have a (collision-)parent, then we write $\mathcal{P}\n=\diamond$ ($\mathcal{P}_c\n=\diamond$) with
\begin{equation*}
s_{\diamond}=t\quad \mathbf{x}_{\diamond}=x\  \textup{if $\n$ does not have a parent}.
\end{equation*}

\begin{lemma}\label{lem:carleman_representation}
\textbf{\textup{[Carleman representation]}} For any function $h(t,x,v)$, any decorated tree $\T$, and any $(t,x,v)\in [0,+\infty)\times \mathds{R}^3\times \mathds{R}^3$, we have
\begin{equation}\label{eq:carleman_representation}
\mathbf{C}_{h}[\T](t,x,v)=\int \Bigg(\prod_{\l\in \mathcal{L}(\T)}h\big(s_{\mathcal{P}\mathfrak{l}},\mathbf{x}_{\mathcal{P}\l},v_{\l}\big)\Bigg)\mathds{1}_{0\leq s_{\mathfrak{r}(\T)}\leq t}d[v]_{\mathcal{N}_{\new}^+(\T)}d[s]_{\mathcal{B}(\T)}.
\end{equation}
For any function $f^0(x,v)$, we have
\begin{equation}\label{eq:carleman_representation_xv}
\mathbf{C}_{f^0}[\T](t,x,v)= \int \Bigg(\prod_{\l\in \mathcal{L}(\T)}f^0\big(\mathbf{x}_{\l},v_{\l}\big)\Bigg)\mathds{1}_{0\leq s_{\mathfrak{r}(\T)}\leq t}d[v]_{\mathcal{N}_{\new}^+(\T)}d[s]_{\mathcal{B}(\T)}.
\end{equation}
\end{lemma}

\begin{proof}
\textit{Induction base:} We prove this Lemma by induction on size $|\mathcal{L}(\T)|$. When $\T$ is a single node $\T=\circ$, we have $\mathbf{C}_{h}[\T](t,x,v)=h(t,x,v)$. When $\T=\CollisionFork$, according to the Definition \ref{def:C_expansion} of $\mathbf{C}_{h}[\mathcal{T}]$
\begin{equation}
\begin{split}
\mathbf{C}_{h}[\T](t,x,v)=\int_0^t\mathbf{S}_{t-s}\mathbf{Q}^+\big(h(s),h(s)\big)ds&=\int_0^t \mathbf{Q}^+\big(h(s),h(s)\big)(x-(t-s)v,v)ds\\
&=\int_0^t \big((v_*-v)\cdot\omega\big)_+h\big(s,x-(t-s)v,v'\big)h\big(s,x-(t-s)v,v_*'\big)d\omega dv_*ds.
\end{split}
\end{equation}
Apply the change of variables in Lemma \ref{lem:Carleman_change_of_variables}
\begin{equation}
\begin{split}
\mathbf{C}_{h}[\T](t,x,v)&=\int_0^t h\big(s,x-(t-s)v,v'\big)h\big(s,x-(t-s)v,v_*'\big) |v'-v|^{-1}d\pi_{v+(v'-v)^{\perp}}(v_*')dv'ds.
\end{split}
\end{equation}
When $\T=\CollisionFork$, according to the Definition \ref{def:carleman_measure} of $d[v]_{\nnew^+(\T)}$ and $d[s]_{\mathcal{B}(\T)}$
\begin{equation}
d[v]_{\nnew^+(\T)}=|v'-v|^{-1}d\pi_{v+(v'-v)^{\perp}}(v_*')dv',\quad d[s]_{\mathcal{B}(\T)}=ds.
\end{equation}
Then
\begin{equation}
\mathbf{C}_{h}[\T](t,x,v):=\int \int \Bigg(\prod_{\l\in \mathcal{L}(\T)}h\big(s_{\mathcal{P}\mathfrak{l}},x_{\mathcal{P}\l},v_{\l}\big)\Bigg)\mathds{1}_{0\leq s_{\mathfrak{r}}\leq t}d[v]_{\mathcal{N}_{\new}^+(\T)}d[s]_{\mathcal{B}(\T)}.
\end{equation}
The same argument also applies to $\T=\OverlapFork$.

\textit{Induction hypothesis:} Assume for any decorated tree $\T$ with $|\mathcal{L}(\T)|=n$ and $n<N$, the $\mathbf{C}_h[\T]$ satisfies \eqref{eq:carleman_representation}.

\textit{Induction step:} Take a decorated tree $\T$ with $|\mathcal{L}(\T)|=N$, where the left sub-tree of the root is $\T_l$, the right sub-tree of the root is $\T_r$. We have $|\mathcal{L}(\T_l)|<N$ and $|\mathcal{L}(\T_r)|<N$. The induction hypothesis gives
\begin{equation}
\left\{
\begin{split}
& \mathbf{C}_{h}[\T_l](t,x,v):= \int \Bigg(\prod_{\l\in \mathcal{L}(\T_l)}h\big(s_{\mathcal{P}\mathfrak{l}},\mathbf{x}_{\mathcal{P}\mathfrak{l}},v_{\l}\big)\Bigg)\mathds{1}_{0\leq s_{\mathfrak{r}(\T_l)}\leq t}d[v]_{\mathcal{N}_{\new}^+(\T_l)}d[s]_{\mathcal{B}(\T_l)}\\
&\mathbf{C}_{h}[\T_r](t,x,v):= \int \Bigg(\prod_{\l\in \mathcal{L}(\T_r)}h\big(s_{\mathcal{P}\mathfrak{l}},\mathbf{x}_{\mathcal{P}\l},v_{\l}\big)\Bigg)\mathds{1}_{0\leq s_{\mathfrak{r}(\T_r)}\leq t}d[v]_{\mathcal{N}_{\new}^+(\T_r)}d[s]_{\mathcal{B}(\T_r)}
\end{split}
\right.
\end{equation}

Suppose the root $\mathfrak{r}(\T)$ is a collision node. Use the inductive Definition \ref{def:C_expansion} of $\mathbf{C}_h[\mathcal{T}]$, with the abbreviations $v_l:=v_{\mathfrak{r}(\T_l)},v_r:=v_{\mathfrak{r}(\T_r)},v:=v_{\mathfrak{r}(\T)},s:=s_{\mathfrak{r}(\T)}$
\begin{equation}\label{eq:carleman_representation_first_compute}
\begin{split}
\mathbf{C}_h[\T](t,x,v)&=\int \mathbf{S}_{t-s}\mathbf{Q}^+\Big(\mathbf{C}_{h}[\T_l](s),\mathbf{C}_{h}[\T_r](s)\Big)\mathds{1}_{0\leq s\leq t}ds\\
&=\int \mathbf{C}_h[\T_l]\Big(s,x-(t-s)v,v_l\Big)\mathbf{C}_{h}[\T_r]\Big(s,x-(t-s)v,v_r\Big)|v_l-v|^{-1}\mathds{1}_{0\leq s\leq t}d\pi_{v+(v_l-v)^{\perp}}(v_r)dv_lds
\end{split}
\end{equation}
Insert the inductive hypothesis into \eqref{eq:carleman_representation_first_compute}, and $\mathbf{C}_h[\T](t,x,v)$ becomes
\begin{equation}
\begin{split}
&\int \Bigg(\prod_{\l\in \mathcal{L}(\T_l)}h\big(s_{\mathcal{P}\mathfrak{l}},\mathbf{x}_{\mathcal{P}\mathfrak{l}},v_{\l}\big)\Bigg)\Bigg(\prod_{\l\in \mathcal{L}(\T_r)}h\big(s_{\mathcal{P}\mathfrak{l}},\mathbf{x}_{\mathcal{P}\mathfrak{l}},v_{\l}\big)\Bigg)\\
&\quad \times \underbrace{d[v]_{\mathcal{N}_{\new}^+(\T_l)}d[v]_{\mathcal{N}_{\new}^+(\T_r)}|v_l-v|^{-1}d\pi_{v+(v_l-v)^{\perp}}(v_r)dv_l}_{\textup{\blue{the recursive definition of $d[v]_{\nnew^+(\T)}$}}}\underbrace{\mathds{1}_{0\leq s_{\mathfrak{r}(\T_l)}\leq s}\mathds{1}_{0\leq s_{\mathfrak{r}(\T_r)}\leq s}d[s]_{\mathcal{B}(\T_l)}d[s]_{\mathcal{B}(\T_r)}\mathds{1}_{0\leq s\leq t}ds}_{\textup{\blue{the recursive definition of $d[s]_{\mathcal{B}(\T)}$}}}
\end{split}
\end{equation}
The definitions of $d[v]_{\nnew^+(\T)}$ and $d[s]_{\mathcal{B}(\T)}$ in Definition \ref{def:carleman_measure} and the fact that $\mathcal{L}(\T_l)\cup\mathcal{L}(\T_r)=\mathcal{L}(\T)$ conclude the proof. The case of $\mathfrak{r}(\T)$ as an overlap node is the same, and we omit the proof.

\end{proof}

\subsection{Homogenization Operations and Homogenization Procedure}

Lemma \ref{lem:carleman_representation} shows that the term $\mathbf{C}_{h}[\T](t)$ can be written as the multiplication of functions $h$ corresponding to all the leaves in $\mathcal{L}(\T)$, with certain couplings induced by the structure of $\T$
\begin{equation}\label{eq:product_of_initials}
\mathbf{C}_{h}[\T](t,x,v)=\int \Bigg(\prod_{\l\in\mathcal{L}(\T)}h\big(s_{\mathcal{P}\l},\x_{\mathcal{P}\mathfrak{l}},v_{\mathfrak{l}}\big)\Bigg)\mathds{1}_{0\leq s_{\mathfrak{r}(\T)}\leq t}d[v]_{\mathcal{N}_{\new}^+(\T)}d[s]_{\mathcal{B}(\T)}.
\end{equation}

A crucial idea of establishing the main results of the paper is the homogenization operations, which allow us to replace these highly singular functions $h$ in \eqref{eq:product_of_initials} by regular enough functions. To formalize this heuristic, we introduce the notion of Boltzmann integral with replacement.

\begin{definition}\label{def:C_expansion_replacement}
Consider any functions $h(t,x,v)$ and $h^{\dagger}(t,x,v)$. For any decorated tree $\T$, any set $\mathcal{L}^\dagger\subset \mathcal{L}(\T)$, the function $\mathbf{C}_{h}[\T,\mathcal{L}^{\dagger}=h^{\dagger}]$ is defined recursively. When $\T$ is a single node $\mathcal{T}=\circ$
\begin{equation}
\mathbf{C}_{h}[\T,\mathcal{L}^{\dagger}=h^{\dagger}](t,x,v):=
\left\{
\begin{array}{lr}
h(t,x,v)     &\textup{ if $\mathcal{L}^{\dagger}=\varnothing$}  \\
h^{\dagger}(t,x,v)     & \textup{ if $\mathcal{L}^{\dagger}=\mathcal{L}(\T)$}
\end{array}
\right.
\end{equation}
When $\T$ has at least one branching node, denote $\T_l$ as the left sub-tree and $\T_r$ as the right sub-tree of the root $\mathfrak{r}(\T)$
\begin{equation*}
\mathbf{C}_{h}[\T,\mathcal{L}^\dagger=h^{\dagger}](t):=
\left\{
\begin{array}{lr}
\int_0^t\mathbf{S}_{t-s}\mathbf{Q}^+\Big(\mathbf{C}_{h}[\T_l,\mathcal{L}^\dagger\cap \mathcal{L}(\T_l)=h^{\dagger}](s),\mathbf{C}_{h}[\T_r,\mathcal{L}^\dagger\cap\mathcal{L}(\T_r)=h^{\dagger}](s)\Big)ds     & \textup{ if $\mathfrak{r}(\T)$ is 'collision'} \\
 \int_0^t\mathbf{S}_{t-s}\mathbf{Q}^-\Big(\mathbf{C}_{h}[\T_l,\mathcal{L}^\dagger\cap \mathcal{L}(\T_l)=h^{\dagger}](s),\mathbf{C}_{h}[\T_r,\mathcal{L}^\dagger\cap\mathcal{L}(\T_r)=h^{\dagger}](s)\Big)ds        & \textup{ if $\mathfrak{r}(\T)$ is 'overlap'} 
\end{array}
\right.
\end{equation*}

\end{definition}

The same strategy of proving Lemma \ref{lem:carleman_representation} can derive the following lemma.

\begin{lemma}
Consider any functions $h(t,x,v)$ and $h^{\dagger}(t,x,v)$. For any decorated tree $\T$, any set $\mathcal{L}^\dagger\subset \mathcal{L}(\T)$, the function $\mathbf{C}_{h}[\T,\mathcal{L}^\dagger=h^{\dagger}]$ equals
\begin{equation}\label{eq:carleman_representation_replaced}
\mathbf{C}_{h}\Big[\T,\mathcal{L}^{\dagger}=h^{\dagger}\Big](t,x,v)=\int \Bigg(\prod_{\l\in \mathcal{L}(\T)\setminus\mathcal{L}^\dagger}h\big(s_{\mathcal{P}\mathfrak{l}},\x_{\mathcal{P}\l},v_{\l}\big)\Bigg)\Bigg(\prod_{\mathfrak{l}\in \mathcal{L}^{\dagger}}h^{\dagger}(s_{\mathcal{P}\mathfrak{l}},\x_{\mathcal{P}\mathfrak{l}},v_{\mathfrak{l}})\Bigg)\mathds{1}_{0\leq s_{\mathfrak{r}(\T)}\leq t}d[v]_{\mathcal{N}_{\new}^+(\T)}d[s]_{\mathcal{B}(\T)}.
\end{equation}
Consider any functions $f(x,v)$ and $f^{\dagger}(x,v)$
\begin{equation}\label{eq:carleman_representation_replaced_xv}
\mathbf{C}_{f}\Big[\T,\mathcal{L}^{\dagger}=f^{\dagger}\Big](t,x,v)=\int \Bigg(\prod_{\l\in \mathcal{L}(\T)\setminus\mathcal{L}^{\dagger}}f\big(\x_{\l},v_{\l}\big)\Bigg)\Bigg(\prod_{\mathfrak{l}\in \mathcal{L}^{\dagger}}f^{\dagger}(\x_{\mathfrak{l}},v_{\mathfrak{l}})\Bigg)\mathds{1}_{0\leq s_{\mathfrak{r}(\T)}\leq t}d[v]_{\mathcal{N}_{\new}^+(\T)}d[s]_{\mathcal{B}(\T)}.
\end{equation}
\end{lemma}

Suppose we have a singular initial data $f^0\in \ldelta$ and a decorated tree $\T$. If we want to estimate the Boltzmann integral $\mathbf{C}_{f^0}[\T]$ in Definition \ref{def:C_expansion}, a natural strategy is to replace the singular initial data $f^0$ in the original Boltzmann integral attached to different leaves (see Lemma \ref{lem:carleman_representation}) by Maxwellians step by step
\begin{equation*}
\mathbf{C}_{f^0}[\T]\mapsto \mathbf{C}_{f^0}\Big[\T,\{\l_1\}=M_{\gamma(0)}\Big]\mapsto \mathbf{C}_{f^0}\Big[\T,\{\l_1,\l_2\}=M_{\gamma(0)}\Big]\mapsto ......\mapsto \mathbf{C}_{f^0}\Big[\T,\mathcal{L}(\T)=M_{\gamma(0)}\Big]
\end{equation*}
In Section \ref{sec:homogenization_operation} we will prove a series of Lemmas estimating the errors of all the possible replacements. We call the different ways of replacing $f^0$ by $M_{\gamma(0)}$ as \textit{homogenization operations}.

\begin{definition}\label{def:homogenization_operation}
\textbf{\textup{[Homogenization Operation]}} Given any decorated tree $\T$ and any set $\mathcal{L}_M$ of Maxwellian leaves, a homogenization operation is a specific way of choosing one non-Maxwellian leaf $\l\in \mathcal{L}(\T)\setminus\mathcal{L}_M$ and update $\mathcal{L}_M$ by $\mathcal{L}_M\cup\{\l\}$, or a specific way of choosing two non-Maxwellian leaves $\{\l_l,\l_r\}\subset \mathcal{L}(\T)\setminus\mathcal{L}_M$ and updating $\mathcal{L}_M$ by $\mathcal{L}_M\cup\{\l_l,\l_r\}$.
\begin{itemize}
    \item If we homogenize two non-Maxwellian leaves $\{\l_l,\l_r\}\subset \mathcal{L}(\T)\setminus\mathcal{L}_M$, and they have the same collision-parent node, then this is called \textup{\textbf{double homogenization}}.
    \item If we homogenize one non-Maxwellian leaf $\l\in \mathcal{L}(\T)\setminus\mathcal{L}_M$, and it has a collision-parent node $\n$, and the other collision-child node of $\n$ is a branching node (thus not a leaf), then this is called \textbf{\textup{single homogenization}}.
    \item If we homogenize one non-Maxwellian leaf $\l\in \mathcal{L}(\T)\setminus\mathcal{L}_M$, and it does not have a collision-parent node, then this is called \textup{\textbf{individual homogenization}}.
\end{itemize}
\end{definition}

\begin{figure}[h]
    \centering
    \includegraphics[width=0.8\linewidth]{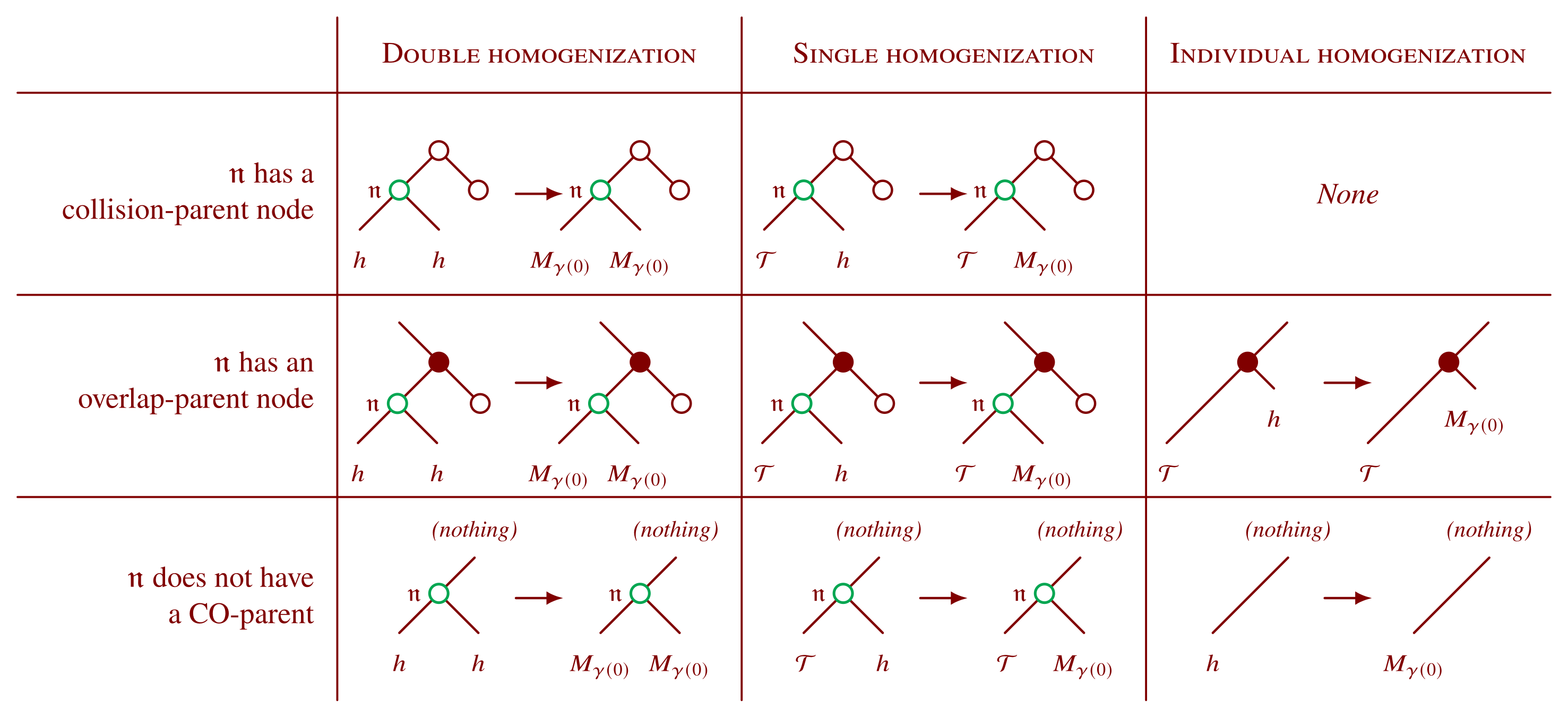}
    \caption{This is a more detailed classification of all the homogenization operations we use in this paper.}
    \label{fig:summary_of_operations}
\end{figure}

In Section \ref{sec:homogenization_operation}, we prove a series of Lemmas estimating the errors of each homogenization operation. Heuristically, the estimates are of the following forms
\begin{equation*}
\begin{split}
&\mathbf{C}_{f^0}\Big[\mathcal{T},\mathcal{L}_M=M_{\gamma(0)}\Big](t,x,v)\leq \textup{Error term}\times \mathbf{C}_{f^0}\Big[\mathcal{T},\mathcal{L}_M\cup\{\textup{newly homogenized leaves}\}=M_{\gamma(0)}\Big](t,x,v)\\
&\Big\lVert\mathbf{C}_{f^0}\Big[\mathcal{T},\mathcal{L}_M=M_{\gamma(0)}\Big](t)\Big\rVert_{\ldeltastar}\leq \textup{Error term}\times \Big\lVert\mathbf{C}_{f^0}\Big[\mathcal{T},\mathcal{L}_M\cup\{\textup{newly homogenized leaves}\}=M_{\gamma(0)}\Big](t)\Big\rVert_{\ldeltastar}
\end{split}
\end{equation*}
According to the list of homogenization operations in Definition \ref{def:homogenization_operation}, we can not arbitrarily choose two non-Maxwellian leaves and homogenize them. Therefore we need a specific way of putting different homogenization operations together, in order to homogenize all the leaves of $\T$, and this is called \textit{homogenization procedure}.

\begin{definition}\label{def:homogenization_procedure}
\textbf{\textup{[Homogenization Procedure]}}

Given a decorated tree $\T$, and the sequence $\{\T_i^+\}_{0\leq i\leq |\mathcal{G_T}|}$, we set $\mathcal{L}_M(\T_i^+)=\varnothing$ at the beginning for every $0\leq i\leq |\mathcal{G_T}|$.

\textup{Step 1:} In the graph $\mathcal{G}_{\mathcal{T}}$, choose a lowest node $i$ with $\mathcal{L}_M(\T_i^+)=\varnothing$.

\textup{Step 2:} In the collisional decorated sub-tree $\T_{i}^+$, choose a non-Maxwellian leaf $\l\in \mathcal{L}(\T_i^+)\setminus\mathcal{L}_M(\T_i^+)$ with the largest depth among the set $\mathcal{L}(\T_i^+)\setminus\mathcal{L}_M(\T_i^+)$ of non-Maxwellian leaves.

\textup{Step 3, Case 1:} If $\l$ has a parent node $\n$ in $\T_i^+$, and the other collision-child node of $\n$ is also a non-Maxwellian leaf denoted as $\mathfrak{l}_*$, then apply \textit{double homogenization} and update $\mathcal{L}_M(\T_i^+)$ by adding two new leaves $\{\mathfrak{l},\mathfrak{l}_*\}$ into it;\\
\textup{Step 3, Case 2:} If $\l$ has a parent node $\n$ in $\T_i^+$, and the other child node of $\n$ is not a leaf, then apply \textit{single homogenization} and update $\mathcal{L}_M(\T_i^+)$ by adding a new leaf $\{\mathfrak{l}\}$ into it\\
\textup{Step 3, Case 3:} If $\l$ has no parent node in $\T_i^+$, then apply \textit{individual homogenization} and update $\mathcal{L}_M(\T_i^+)$ by adding a new leaf $\{\mathfrak{l}\}$ into it.

\textup{Step 4:} Repeat step 2 until $\mathcal{L}_M(\T_i^+)=\mathcal{L}(\T_i^+)$.

\textup{Step 5:} Repeat step 1 until there is no node $i$ left in $\mathcal{G}_{\T}$ with $\mathcal{L}_M(\T_i^+)=\varnothing$.
\end{definition}

\begin{figure}
    \centering
    \includegraphics[width=0.7\linewidth]{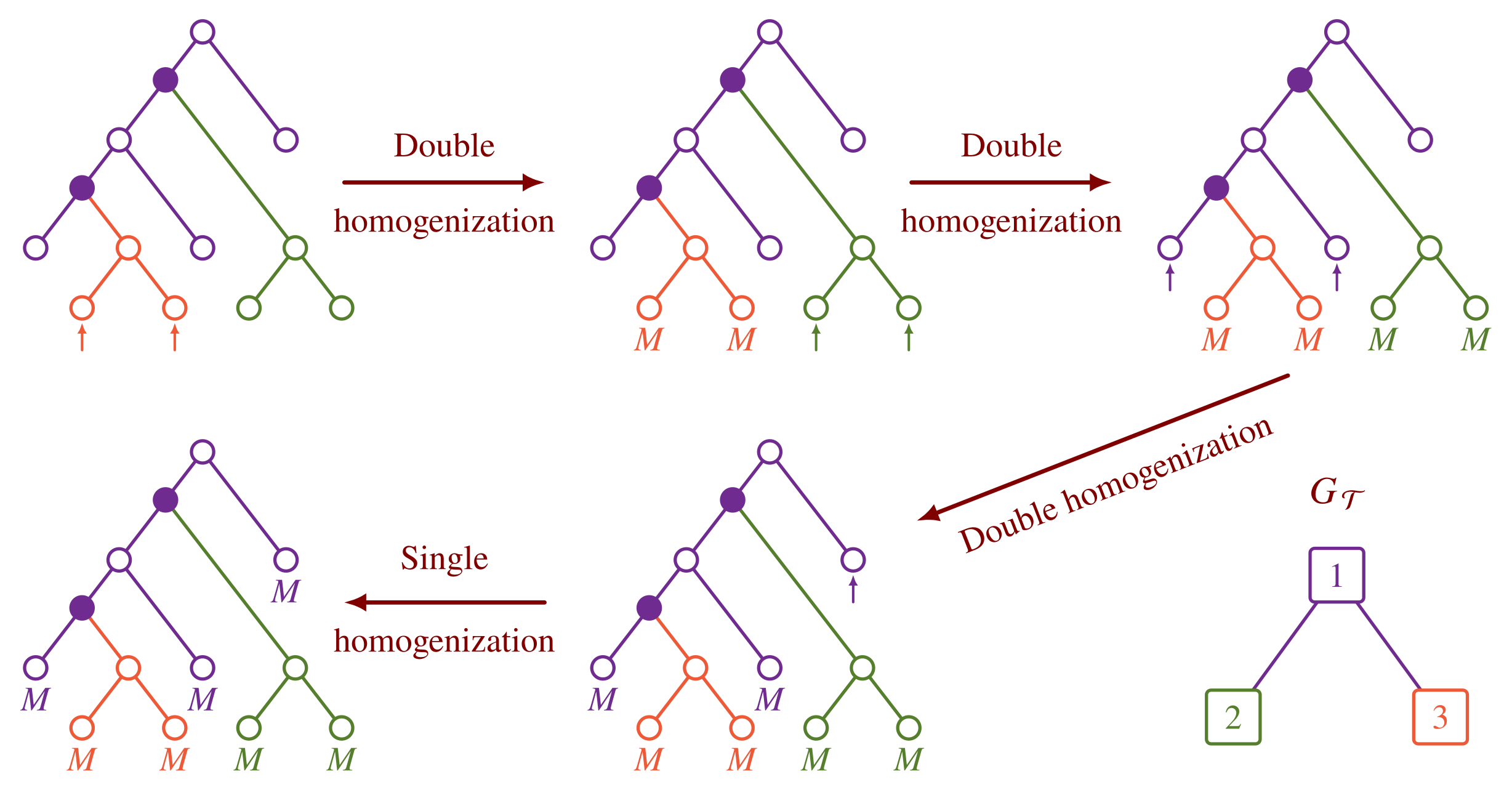}
    \caption{This is an example of the homogenization procedure in Definition \ref{def:homogenization_procedure}. When a leaf is noted as 'M', it means this leaf has been homogenized.}
    \label{fig:homogenization_operation}
\end{figure}

\section{Analysis of Overlap Functions and Degeneracies}\label{sec:overlap_and_degeneracy}

As we can see in Remark \ref{rmk:pseudo-trajectory_and_xfunction} on the pseudo-trajectory, the presence of overlaps does not change the geometry of trajectories. And as shown in Figure \ref{fig:intersection} and the proof of Proposition \ref{prop:toy_example}, we gain regularity from the structure of collisions. In this section, we prove several results filtering the effects of overlaps, which will later allow us to gain regularity from collisions (see Section \ref{sec:homogenization_operation}). 

\begin{definition}\label{def:overlap_function}
\textup{\textbf{[Overlap function] }}For any decorated tree $\mathcal{T}$, any $0\leq t\leq T_*$, and any $v\in\mathds{R}^3$, define $\mathbf{O}_{M_{\gamma(0)}}[\mathcal{T}](t,v)$ as the function representing the contribution of an overlap
\begin{equation}\label{eq:overlap_function}
\mathbf{O}_{M_{\gamma(0)}}[\mathcal{T}](t,v):=\pi\int_{\mathds{R}^3}|v-v_*|\mathbf{C}_{M_{\gamma(0)}}[\mathcal{T}](t,v_*)dv_*.
\end{equation}
\end{definition}

\subsection{Reduction of Overlaps}

\begin{lemma}\label{lem:full_to_prime}
Consider any decorated tree $\T$, any set $\mathcal{L}_M\subset\mathcal{L}(\T)$ of Maxwellian leaves, and any $h\in L^{\infty}\ldelta([0,T]\times \mathds{R}^3\times \mathds{R}^3)$. Suppose for any $\n\in\mathcal{N}^+(\tprime)$ and $\o\in \mathcal{O}[\n]$, all the leaves of $\mathscr{T}_{\o}$ are Maxwellian $\mathcal{L}(\mathscr{T}_{\o})\subset \mathcal{L}_M$. Then we have for any $(t,x,v)\in [0,T]\times \mathds{R}^3\times \mathds{R}^3$
\begin{equation}\label{eq:full_to_prime}
\begin{split}
&\mathbf{C}_{h}[\T,\mathcal{L}_M=M_{\gamma(0)}](t,x,v)=\int \Bigg(\prod_{\mathfrak{l}\in \mathcal{L}(\T_{0}^+)\setminus\mathcal{L}_M}h\big(s_{\mathcal{P}\l},\x_{\mathcal{P}\mathfrak{l}},v_{\mathfrak{l}}\big)\Bigg)\Bigg(\prod_{\mathfrak{l}\in \mathcal{L}(\T_{0}^+)\cap \mathcal{L}_M}M_{\gamma(0)}(v_{\l})\Bigg)\\
&\hspace{5cm}\Bigg(\prod_{\mathfrak{n}\in \mathcal{N}(\tprime)}\widetilde{\mathbf{O}}_{\n}\Big(s_{\n},[s]_{\mathcal{O}[\n]},s_{\mathcal{P}_c\n},v_{\n}\Big)\Bigg)\mathds{1}_{0\leq s_{\mathfrak{r}(\T)}\leq t}d[v]_{\nnew^+(\tprime)}d[s]_{\mathcal{B}(\T_0^+)},
\end{split}
\end{equation}
with
\begin{equation}\label{eq:O_tilde}
\widetilde{\mathbf{O}}_{\n}\Big(s_{\n},[s]_{\mathcal{O}[\n]},s_{\mathcal{P}_c\n},v_{\n}\Big):=\mathds{1}_{s_{\n}\leq s_{\mathcal{P}\n}}\prod_{\o\in \mathcal{O}[\n]}\Bigg(\mathds{1}_{s_{\o}\leq s_{\mathcal{P}\o}}\mathbf{O}_{M_{\gamma(0)}}[\mathscr{T}_{\o}](s_{\o},v_{\mathfrak{n}})ds_{\o}\Bigg)
\end{equation}

\end{lemma}

\begin{proof}
Let us expand the LHS of \eqref{eq:full_to_prime} and compare the LHS and RHS of \eqref{eq:full_to_prime}
\begin{equation}\label{eq:full_to_prime_expanded}
\begin{split}
&\int \Bigg(\prod_{\l\in \mathcal{L}(\T)\setminus\mathcal{L}_M}h\big(s_{\mathcal{P}\mathfrak{l}},\x_{\mathcal{P}\l},v_{\l}\big)\Bigg)\Bigg(\prod_{\mathfrak{l}\in \mathcal{L}_M}M_{\gamma(0)}(v_{\mathfrak{l}})\Bigg)\mathds{1}_{0\leq s_{\mathfrak{r}(\T)}\leq t}d[v]_{\mathcal{N}_{\new}^+(\T)}d[s]_{\mathcal{B}(\T)}\\
&=\int \Bigg(\prod_{\mathfrak{l}\in \mathcal{L}(\T_{0}^+)\setminus \mathcal{L}_M}h\big(s_{\mathcal{P}\l},\x_{\mathcal{P}\mathfrak{l}},v_{\mathfrak{l}}\big)\Bigg)\Bigg(\prod_{\l\in \mathcal{L}(\tprime)\cap \mathcal{L}_M}M_{\gamma(0)}(v_\l)\Bigg)\Bigg(\prod_{\mathfrak{n}\in \mathcal{N}(\tprime)}\widetilde{\mathbf{O}}_{\n}\Big(s_{\n},[s]_{\mathcal{O}[\n]},s_{\mathcal{P}_c\n},v_{\n}\Big)\Bigg)\mathds{1}_{0\leq s_{\mathfrak{r}(\T)}\leq t}d[v]_{\nnew^+(\tprime)}d[s]_{\mathcal{B}(\T_0^+)},
\end{split}
\end{equation}
We first decompose the integral measure $d[s]_{\mathcal{B}(\T)}$ of time in \eqref{eq:carleman_time_measure_def}. We integrate over the time of collisions in $\tprime$ by $d[s]_{\mathcal{B}(\tprime)}$. For each collision node or leaf $\n\in \mathcal{N}^+(\tprime)$, there may be overlap nodes $\o\in \mathcal{O}[\mathfrak{n}]$ between $\n$ and its collision-parent node $\mathcal{P}_c\n$, see the middle part of Figure \ref{fig:full_to_prime}. For each overlap node $\o\in \mathcal{O}[\n]$, the time $s_{\o}$ of this overlap must be between the previous collision time and the next collision time, thus $\mathds{1}_{s_{\n}\leq s_{\o}\leq s_{\mathcal{P}_c\n}}$.
\begin{equation}\label{eq:prime_decompositioni_carleman_time}
\mathds{1}_{0\leq s_{\mathfrak{
r}(\T)}\leq t}d[s]_{\mathcal{B}(\T)}=\prod_{\n\in \mathcal{N}^+(\T_0^+)}\mathds{1}_{s_{\n}\leq s_{\mathcal{P}\n}}\Bigg(\prod_{\o\in \mathcal{O}[\mathfrak{n}]}\Big(\underbrace{\red{\mathds{1}_{0\leq s_{\mathfrak{r}(\mathscr{T}_{\o})}\leq s_{\o}}d[s]_{\mathcal{B}(\mathscr{T}_{\o})}}}_{\textup{\blue{integral of times for $\mathscr{T}_{\o}$}}}\underbrace{\mathds{1}_{s_{\o}\leq s_{\mathcal{P}\o}}ds_{\o}}_{\textup{\blue{time of overlap $\o$}}}\Big)\Bigg)d[s]_{\mathcal{B}(\tprime)},
\end{equation}
where for each $\o\in \mathcal{O}(\T)$ we also integrate over the time $[s]_{\mathcal{B}(\mathscr{T}_{\o})}$ of the branched sub-tree $\mathscr{T}_{\o}$ of $\o$, for example see the right part of Figure \ref{fig:full_to_prime}.
\begin{figure}[h]
    \centering
    \includegraphics[width=0.7\linewidth]{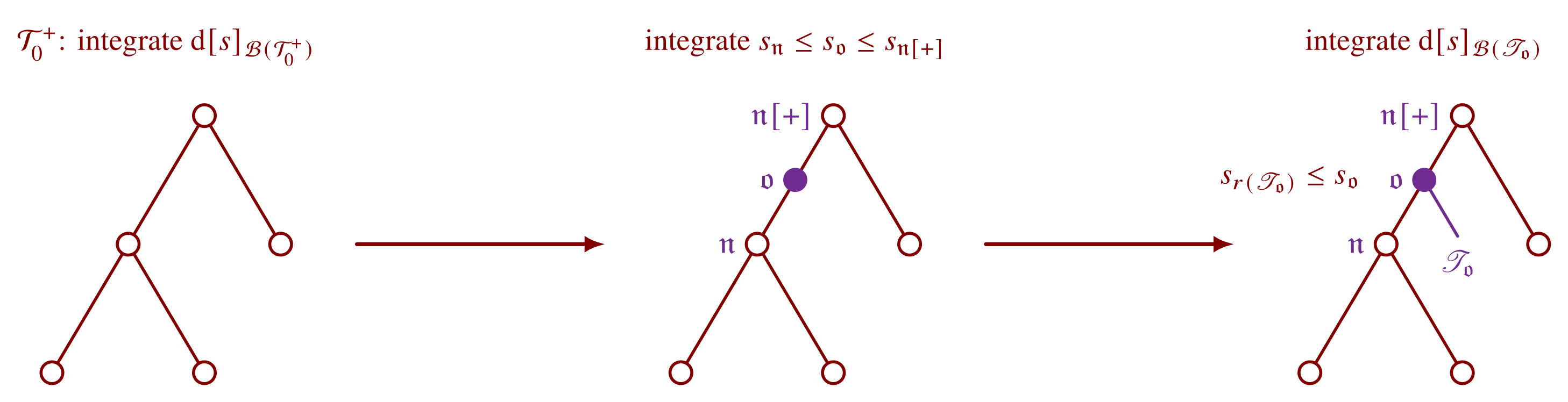}
    \caption{}
    \label{fig:full_to_prime}
\end{figure}

Next we decompose the integral measure $d[v]_{\nnew^+(\T)}$ of velocities in \eqref{eq:carleman_measure_of_velocities}. First we integrate over the velocities of collisions in $\tprime$ by $d[v]_{\nnew^+(\tprime)}$. For each collision node or leaf $\n\in \mathcal{N}^+(\tprime)$, there may be overlap nodes $\o\in \mathcal{O}[\mathfrak{n}]$ between $\n$ and its collision-parent node $\mathcal{P}_c\n$. 
\begin{equation}\label{eq:prime_decompositioni_carleman_velocity}
d[v]_{\nnew^+(\T)}=\prod_{\n\in \mathcal{N}^+(\tprime)}\Bigg(\prod_{\o\in \mathcal{O}[\mathfrak{n}]}\red{\Big(\pi|v_{\oold}-v_{\onew}|\ d[v]_{\mathcal{N}^+(\mathscr{T}_{\o})}\Big)}\Bigg)d[v]_{\nnew^+(\tprime)}
\end{equation}
Then we decompose the set $\mathcal{L}(\T)$ of leaves
\begin{equation}\label{eq:initial_data_decomposition}
\begin{split}
&\Bigg(\prod_{\l\in \mathcal{L}(\T)\setminus\mathcal{L}_M}h\big(s_{\mathcal{P}\mathfrak{l}},\x_{\mathcal{P}\l},v_{\l}\big)\Bigg)\Bigg(\prod_{\mathfrak{l}\in \mathcal{L}_M}M_{\gamma(0)}(v_{\mathfrak{l}})\Bigg)\\
&=\Bigg(\prod_{\n\in \mathcal{N}^+(\tprime)}\prod_{\o\in \mathcal{O}[\n]}\red{\Big(\prod_{\l\in \mathcal{L}(\mathscr{T}_{\o})}M_{\gamma(0)}(v_{\l})\Big)}\Bigg)\Bigg(\prod_{\l\in \mathcal{L}(\tprime)\setminus \mathcal{L}_M}h\big(s_{\mathcal{P}\l},\x_{\mathcal{P}\l},v_{\l}\big)\Bigg)\Bigg(\prod_{\l\in \mathcal{L}(\tprime)\cap \mathcal{L}_M}M_{\gamma(0)}(v_{\l})\Bigg)
\end{split}
\end{equation}
The structures of the three decompositions \eqref{eq:prime_decompositioni_carleman_time}-\eqref{eq:initial_data_decomposition} are very similar: given each $\n\in \mathcal{N}^+(\tprime)$ and $\o\in \mathcal{O}[\n]$, we have terms related to the branched sub-tree $\mathscr{T}_{\o}$. Notice
\begin{equation}
\begin{split}
&\int \Big(\prod_{\mathfrak{l}\in \mathcal{L}(\mathscr{T}_{\o})}M_{\gamma(0)}(s_{\mathcal{P}\l},\x_{\mathcal{P}\l},v_{\l})\Big)\Big(\pi|v_{\oold}-v_{\onew}|\ d[v]_{\mathcal{N}^+(\mathscr{T}_{\o})}\Big)\mathds{1}_{s_{\mathfrak{r}(\mathscr{T}_{\o})}\leq s_{\o}}d[s]_{\mathcal{B}(\mathscr{T}_{\o})}\\
&=\int  |v_{\o}-v| \mathbf{C}_{M_{\gamma(0)}}[\mathscr{T}_{\o}](s_{\o},v_{\o})dv_{\o}=\mathbf{O}_{M_{\gamma(0)}}[\mathscr{T}_{\o}](s_{\o},v).
\end{split}
\end{equation}
Combining the decompositions \eqref{eq:prime_decompositioni_carleman_time}-\eqref{eq:initial_data_decomposition}, the LHS of \eqref{eq:full_to_prime} equals
\begin{equation}
\begin{split}
&\int \Bigg(\prod_{\l\in \mathcal{L}(\tprime)\setminus \mathcal{L}_M}h\big(s_{\mathcal{P}\l},\x_{\mathcal{P}\l},v_{\l}\big)\Bigg)\Bigg(\prod_{\l\in \mathcal{L}(\tprime)\cap \mathcal{L}_M}M_{\gamma(0)}(v_{\l})\Bigg)\\
&\hspace{3cm}\times \Bigg(\prod_{\n\in \mathcal{N}^+(\tprime)}\mathds{1}_{s_{\n}\leq s_{\mathcal{P}\n}}\prod_{\o\in\mathcal{O}[\n]}\mathbf{O}_{M_{\gamma(0)}}[\mathscr{T}_{\o}](s_{\o},v_{\n})\mathds{1}_{s_{\o}\leq s_{\mathcal{P}\o}}ds_{\o}\Bigg)\mathds{1}_{0\leq s_{\mathfrak{r}(\T)}\leq t}d[v]_{\nnew^+(\tprime)}d[s]_{\mathcal{B}(\tprime)},
\end{split}
\end{equation}
where the time indicator functions in the third bracket recovers the time constrain in $\widetilde{O}_{\n}$ (see \eqref{eq:O_tilde}).
This concludes the proof of Lemma \ref{lem:full_to_prime}.
\end{proof}

\subsection{Log-Lipschitz Regularity of Overlap Functions}

\begin{lemma}\label{lem:non_decreasing_form}
For any decorated tree $\T$, any $t\in[0,+\infty)$, and any $\gamma(0)\in (0,+\infty)$, we have
\begin{equation}\label{eq:non_decreasing_form}
\mathbf{C}_{M_{\gamma(0)}}[\T](t,v)=c_{\T}(t)M_{\gamma(0)}(v)\Lambda_{\gamma(0),\T}(v).
\end{equation}
Here $c_{\T}:[0,+\infty)\rightarrow [0,+\infty)$ is a function of $t$, and $\Lambda_{\gamma(0),\T}:\mathds{R}^3\rightarrow [0,+\infty)$ is a radially symmetric function non-decreasing in $|v|$. The function $c_{\T}(t)$ is only determined by $\T$, and $\Lambda_{\gamma(0),\T}(v)$ is only determined by $\gamma(0)$ and $\T$.
\end{lemma}

\begin{proof}

This lemma will be proved by induction on the size $|\mathcal{L}(\T)|$ of $\mathcal{T}$. For $\mathcal{T}$ having a single node $\mathcal{T}=\circ$
\begin{equation}
\mathbf{C}_{M_{\gamma(0)}}[\circ](t,v)=M_{\gamma(0)}(v),\quad \textup{ therefore $c_{\circ}(t)\equiv 1$ and $\Lambda_{\gamma(0),\circ}\equiv 1$.}
\end{equation}
Suppose for a positive integer $M$, any decorated tree $\T'$ with $|\mathcal{L}(\T')|\leq M$ satisfies the condition \eqref{eq:non_decreasing_form}. For a decorated tree $\T$ with $|\mathcal{L}(\T)|=M+1$, denote its left sub-tree as $\T_l$ and right sub-tree as $\T_r$. Consequently we have $|\mathcal{L}(\T_l)|\leq M$ and $|\mathcal{L}(\T_r)|\leq M$. Therefore the induction hypothesis implies
\begin{equation}\label{eq:induction_hypothesis_Lambda}
\mathbf{C}_{M_{\gamma(0)}}[\T_l](t,v)=c_{\T_l}(t)M_{\gamma(0)}(v)\Lambda_{\gamma(0),\mathcal{T}_l}(v),\quad \mathbf{C}_{M_{\gamma(0)}}[\T_r](t,v)=c_{\T_r}(t)M_{\gamma(0)}(v)\Lambda_{\gamma(0),\mathcal{T}_r}(v).
\end{equation}
During the proof we use $\Lambda_l$ and $\Lambda_r$ as the abbreviations of $\Lambda_{\gamma(0),\mathcal{T}_l}$ and $\Lambda_{\gamma(0),\mathcal{T}_r}$.

\textit{If the root of $\T$ is an overlap node:}

By the recursive definition of the Boltzmann integral in Definition \ref{def:C_expansion} and the spatial independence of the Maxwellian $M_{\gamma(0)}$
\begin{equation}
\mathbf C_{M_{\gamma(0)}}[\T](t,v)
=
\int_0^t
\mathbf Q^-\Big(
        \mathbf C_{M_{\gamma(0)}}[\T_l](s),
        \mathbf C_{M_{\gamma(0)}}[\T_r](s)
\Big)(v)\,ds .
\end{equation}
Using the induction hypothesis \eqref{eq:induction_hypothesis_Lambda}
\begin{equation}\label{eq:LEM4.1_overlap}
\begin{split}
\mathbf C_{M_{\gamma(0)}}[\T](t,v)&=\int_0^t
c_{\T_l}(s)c_{\T_r}(s)\mathbf Q^-\big(M_{\gamma(0)}\Lambda_l,M_{\gamma(0)}\Lambda_r\big)(v)\,ds=c_{\mathcal{T}}(t)\mathbf Q^-\big(M_{\gamma(0)}\Lambda_l,M_{\gamma(0)}\Lambda_r\big)(v),\ c_{\T}(t):=\int_0^t c_{\T_l}(s)c_{\T_r}(s)ds
\end{split}
\end{equation}
It remains to prove that $\mathbf{Q}^-\big(M_{\gamma(0)}\Lambda_l,M_{\gamma(0)}\Lambda_r\big)$ can be expressed as $M_{\gamma(0)}(v)\Lambda(v)$, where $\Lambda$ is a nonnegative radial function that is non-decreasing in $|v|$. By the definition of $\mathbf{Q}^-$ in \eqref{eq:gain_and_loss_decomposition}
\begin{equation}
\mathbf Q^-(M_{\gamma(0)}\Lambda_l,M_{\gamma(0)}\Lambda_r)(v)
=
M_{\gamma(0)}(v)\Lambda_l(v)\mathcal{O}_r(v),\quad \mathcal{O}_r(v):=\pi\int_{\mathbb R^3}
|v-v_*|M_{\gamma(0)}(v_*)\Lambda_r(v_*)\,dv_*.
\end{equation}
Therefore it is enough to show that $\mathcal O_r$ is radial,
nonnegative, and nondecreasing in $|v|$. Nonnegativity of $\mathcal{O}_r$ is immediate. Since $M_{\gamma(0)}\Lambda_r$ is a radially symmetric function, the function $\mathcal{O}_r(v)$ is also radially symmetric. It remains to prove the monotonicity of $\mathcal{O}_r$ in $|v|$. Since $\mathcal{O}_r$ is radially symmetric, it is enough to prove the monotonicity in one direction, for example $\mathtt{e}_1=(1,0,0)$. Re-write the function $\mathcal{O}_r$ using the reflection symmetry of $M_{\gamma(0)}\Lambda_{r}$
\begin{equation}\label{eq:symmetrized_O2}
\mathcal{O}_r(r\mathtt{e}_1)=\pi\int_{\mathds{R}^2\times \mathds{R}_{\geq 0}}\Big(|r\mathtt{e}_1-v_*|+|r\mathtt{e}_1+v_*|\Big)M_{\gamma(0)}(v_*)\Lambda_r(v_*)dv_*.
\end{equation}
We first notice this function $\mathcal{O}_r(r\mathtt{e}_1)$ is convex for $r\in\mathds{R}$, because the function $|r\mathtt{e}_1-v_*|+|r\mathtt{e}_1+v_*|$ is convex by the triangle inequality
\begin{equation}
\big|\big(\lambda r_{a}+(1-\lambda)r_b\big)\mathtt{e}_1\pm v_*\big|=\big|\lambda \big(r_a\mathtt{e}_1\pm v_*\big)+(1-\lambda)\big(r_b\mathtt{e}_1\pm v_*\big) \big|\leq \lambda \big|r_a\mathtt{e}_1\pm v_*\big|+(1-\lambda)\big|r_b\mathtt{e}_1\pm v_*\big|.
\end{equation}
The function $\mathcal{O}_r$ is even in the sense of $\mathcal{O}_r(r\mathtt{e}_1)=\mathcal{O}_r(-r\mathtt{e}_1)$, by its form in \eqref{eq:symmetrized_O2}. An even and convex function is non-decreasing in $r\geq 0$.

\textit{If the root of $\T$ is a collision node:}

Expand the Boltzmann integral $\mathbf{C}_{M_{\gamma(0)}}[\T](t,v)$ using the induction hypothesis \eqref{eq:induction_hypothesis_Lambda}, similarly to \eqref{eq:LEM4.1_overlap}
\begin{equation}
\mathbf C_{M_{\gamma(0)}}[\T](t,v)=c_{\mathcal{T}}(t)\mathbf Q^+\big(M_{\gamma(0)}\Lambda_l,M_{\gamma(0)}\Lambda_r\big)(v),\quad c_{\T}(t):=\int_0^tc_{\T_l}(s)c_{\T_r}(s)ds.
\end{equation}
The energy conservation of collision gives $M_{\gamma(0)}(v')M_{\gamma(0)}(v_*')=M_{\gamma(0)}(v)M_{\gamma(0)}(v_*)$
\begin{equation}\label{eq:I_omega_def}
\begin{split}
\mathbf{Q}^+\big(M_{\gamma(0)}\Lambda_l,M_{\gamma(0)}\Lambda_r\big)(v)&=\int \big((v_*-v)\cdot\omega \big)_+M_{\gamma(0)}(v')M_{\gamma(0)}(v_*')\Lambda_l(v')\Lambda_r(v_*')dv_*d\omega\\
&=\int \big((v_*-v)\cdot\omega \big)_+M_{\gamma(0)}(v)M_{\gamma(0)}(v_*)\Lambda_l(v')\Lambda_r(v_*')dv_*d\omega\\
&=M_{\gamma(0)}(v)\int_{\mathds{S}^2}\mathcal{I}_{\omega}(v)d\omega,\quad \mathcal{I}_{\omega}(v):=\int \big((v_*-v)\cdot\omega\big)_+M_{\gamma(0)}(v_*)\Lambda_l(v')\Lambda_r(v_*')dv_*.
\end{split}
\end{equation}
Consequently 
\begin{equation}
\mathbf{C}_{M_{\gamma(0)}}[\T](t,v)=c_{\T}(t)M_{\gamma(0)}(v)\Lambda_{\gamma(0),\T}(v),\quad \Lambda_{\gamma(0),\T}(v):=\int_{\mathds{S}^2}\mathcal{I}_{\omega}(v)d\omega.
\end{equation}
The non-negativity and radial symmetry of $\Lambda_{\gamma(0),\T}$ are direct. Thus it suffices to prove $\Lambda_{\gamma(0),\mathcal{T}}(v)$ is non-decreasing in $|v|$. Given any $\omega\in\mathds{S}^2$, we can decompose the variables $(v,v_*)$ according to the directions along or perpendicular to $\omega$, with $v^{\parallel}$ and $v_*^{\parallel}$ parallel to $\omega$, and $v^{\perp}$ and $v_*^{\perp}$ orthogonal to $\omega$. In addition, $(v',v_*')$ can be decomposed similarly according to \eqref{eq:pre_collision}
\begin{equation}\label{eq:paralell_perp_decomposition}
v=v^{\perp}+v^{\parallel},\quad v_*=v_*^{\perp}+v_*^{\parallel},\qquad v'=v^{\perp}+v_*^{\parallel},\quad v_*'=v_*^{\perp}+v^{\parallel}.
\end{equation}
Apply the decomposition \eqref{eq:paralell_perp_decomposition} to the definition of $\mathcal{I}_{\omega}(v)$ in \eqref{eq:I_omega_def}
\begin{equation}\label{eq:first_compute_Iomega}
\begin{split}
&\mathcal{I}_{\omega}(v)=\int \big(v_{*}^{\parallel}\cdot\omega -v^{\parallel}\cdot\omega \big)_+\exp\Big(-\gamma(0)|v_*^{\perp}|^2-\gamma(0)|v_*^{\parallel}|^2\Big)\Lambda_l\Big(\big(|v^{\perp}|^2+|v_*^{\parallel}|^2\big)^{\frac{1}{2}}\Big)\Lambda_r\Big(\big(|v_*^{\perp}|^2+|v^{\parallel}|^2\big)^{\frac{1}{2}}\Big)d\pi_{P[\omega^{\perp}]}(v_*^{\perp})d\pi_{P[\omega]}(v_*^{\parallel})
\end{split}
\end{equation}
Here $d\pi_{P[\omega]}$ is the uniform measure on the straight line spanned by $\omega$, and $d\pi_{P[\omega^{\perp}]}$ is the uniform measure on the plane orthogonal to $\omega$.

The integral in equation \eqref{eq:first_compute_Iomega} can be written as the multiplication of two integrals. One of them is the integral over the parallel component $v_*^{\parallel}$, and the other one is the integral over the orthogonal component $v_*^{\perp}$
\begin{equation}\label{eq:parallel_orthogonal_decomposition}
\mathcal{I}_{\omega}(v)=\underbrace{\Bigg[\int\big((v_*^{\parallel}-v^{\parallel})\cdot\omega \big)_+e^{-\gamma(0)|v_*^{\parallel}|^2}\Lambda_l\Big(\big(|v^{\perp}|^2+|v_*^{\parallel}|^2\big)^{\frac{1}{2}}\Big)d\pi_{P[\omega]}(v_*^{\parallel})\Bigg]}_{\textup{\blue{integral over parallel component}}}\ \underbrace{\Bigg[\int e^{-\gamma(0)|v_*^{\perp}|^2}\Lambda_r\Big(\big(|v_*^{\perp}|^2+|v^{\parallel}|^2\big)^{\frac{1}{2}}\Big)d\pi_{P[\omega^{\perp}]}(v_*^{\perp})\Bigg]}_{\textup{\blue{integral over orthogonal component}}}
\end{equation}
Since $\Lambda_r$ is radially symmetric and non-decreasing, we can see the integral over the orthogonal component in \eqref{eq:first_compute_Iomega} is non-decreasing in $|v|$. For the integral over parallel component, the cross-section $\big((v_*^{\parallel}-v^{\parallel})\cdot\omega\big)_+$ is not non-decreasing in $|v|$. To handle this, we notice that the radial symmetry of $\Lambda_{\gamma(0),\mathcal{T}}(v)$ makes it sufficient to prove the monotonicity of $\Lambda_{\gamma(0),\mathcal{T}}(v)+\Lambda_{\gamma(0),\mathcal{T}}(-v)$
\begin{equation}
\begin{split}
\Lambda_{\gamma(0),\mathcal{T}}(v)+\Lambda_{\gamma(0),\mathcal{T}}(-v)&=\int _{\mathds{S}^2}\underbrace{\Bigg[\int e^{-\gamma(0)|v_*^{\perp}|^2} \Lambda_r\Big(\big(|v_*^{\perp}|^2+|v^{\parallel}|^2\big)^{\frac{1}{2}}\Big)d\pi_{P[\omega^{\perp}]}(v_*^{\perp})\Bigg]}_{\textup{\blue{integral over orthogonal component}}}\\
&\times \underbrace{\Bigg[\int \Big(\big(v_*^{\parallel}\cdot\omega -v^{\parallel}\cdot\omega\big)_++\big(v_*^{\parallel}\cdot\omega +v^{\parallel}\cdot\omega\big)_+\Big)e^{-\gamma(0) |v_*^{\parallel}|^2}\Lambda_l\Big(\big(|v^{\perp}|^2+|v_*^{\parallel}|^2\big)^{\frac{1}{2}}\Big)d\pi_{P[\omega]}(v_*^{\parallel})\Bigg]}_{\textup{\blue{integral over parallel component}}}d\omega
\end{split}
\end{equation}
The sum of cross-sections is monotone in $|v^{\parallel}|$ because
\begin{equation}
\big(v_*^{\parallel}\cdot\omega -v^{\parallel}\cdot\omega\big)_++\big(v_*^{\parallel}\cdot\omega +v^{\parallel}\cdot\omega\big)_+=\max\{2v_*^{\parallel}\cdot\omega,v^{\parallel}_*\cdot\omega+|v^{\parallel}|,0\}
\end{equation}
Consequently
\begin{equation}\label{eq:I-omega}
\begin{split}
&\Lambda_{\gamma(0),\mathcal{T}}(v)+\Lambda_{\gamma(0),\mathcal{T}}(-v)=\int_{\mathds{S}^2} \underbrace{\Bigg[\int \max\{2v_*^{\parallel}\cdot\omega,v_*^{\parallel}\cdot\omega+|v^{\parallel}|,0\} e^{-\gamma(0)|v_*^{\parallel}|^2}\Lambda_l\Big(\big(|v^{\perp}|^2+|v_*^{\parallel}|^2\big)^{\frac{1}{2}}\Big)d\pi_{P[\omega]}(v_*^{\parallel})\Bigg]}_{\textup{\blue{integral over parallel component}}}\\
&\hspace{8cm}\times \underbrace{\Bigg[\int e^{-\gamma(0)|v_*^{\perp}|^2}\Lambda_r\Big(\big(|v_*^{\perp}|^2+|v^{\parallel}|^2\big)^{\frac{1}{2}}\Big)d\pi_{P[\omega^{\perp}]}(v_*^{\perp})\Bigg]}_{\textup{\blue{integral over orthogonal component}}}d\omega
\end{split}
\end{equation}
Both integrals are non-decreasing in $|v|$, and this concludes the proof.
\end{proof}

\begin{lemma}\label{lem:overlap_regularity}
For any parameters $(\gamma(0)>0,\delta>0)$, there exists a constant $\overline{\ell}$ such that the following statement is true. For every decorated tree $\mathcal T$,
every $t\in[0,+\infty)$, every $v\in\mathbb R^3$, and every $\boldsymbol{\eta}\in \mathds{R}^3$ with $|\boldsymbol{\eta}|\leq \overline{\ell}$
\begin{equation}
\mathbf{O}_{M_{\gamma(0)}}[\mathcal T](t,v+\boldsymbol{\eta})
\leq 
3
\mathbf{O}_{M_{\gamma(0)}}[\mathcal T](t,v).
\end{equation}
\end{lemma}

\begin{proof}

Recall the definition of $\mathbf{O}_{M_{\gamma(0)}}[\T]$
\begin{equation}
\mathbf{O}_{M_{\gamma(0)}}[\T](t,v+\boldsymbol{\eta})=\int_{\mathds{R}^3}|v+\boldsymbol{\eta}-u|\mathbf{C}_{M_{\gamma(0)}}[\T](t,u)du.
\end{equation}
In the domain of $|u-v|\geq 10\overline{\ell}$, the perturbation $\boldsymbol{\eta}$ is just a small perturbation to the cross-section because the original cross-section $|u-v|$ is large enough compared to $\boldsymbol{\eta}$
\begin{equation}
|v+\boldsymbol{\eta}-u|\leq \frac{6}{5}|v-u|,
\end{equation}
So we can decompose the overlap function 
\begin{equation}\label{eq:decompose_perturbed_overlap}
\mathbf{O}_{M_{\gamma(0)}}[\mathcal{T}](t,v+\boldsymbol{\eta})=\underbrace{\int_{|u-v|\geq 10\overline{\ell}}|v+\boldsymbol{\eta}-u|\mathbf{C}_{M_{\gamma(0)}}[\T](t,u)du}_{\textup{\blue{away from $v$}}}+\underbrace{\int_{|u-v|< 10\overline{\ell}}|v+\boldsymbol{\eta}-u|\mathbf{C}_{M_{\gamma(0)}}[\T](t,u)du}_{\textup{\blue{close to $v$}}}.
\end{equation}
The first term on the RHS of \eqref{eq:decompose_perturbed_overlap} satisfies
\begin{equation}
\int_{|u-v|\geq 10\overline{\ell}}|v+\boldsymbol{\eta}-u|\mathbf{C}_{M_{\gamma(0)}}[\T](t,u)du\leq \frac{6}{5}\int_{|u-v|\geq 10\overline{\ell}}|v-u|\mathbf{C}_{M_{\gamma(0)}}[\T](t,u)du\leq \frac{6}{5}\mathbf{O}_{M_{\gamma(0)}}[\mathcal{T}](t,v).
\end{equation}
In the region $|u-v|<10\overline{\ell}$, we can not directly compare the cross-sections before and after the perturbation $\boldsymbol{\eta}$, because the original cross-section $|u-v|$ can be arbitrarily small. Instead, we control the contribution of this domain by the total mass
\begin{equation}
\int_{|u-v|< 10\overline{\ell}}|v+\boldsymbol{\eta}-u|\mathbf{C}_{M_{\gamma(0)}}[\T](t,u)du\leq 11\overline{\ell}\underbrace{\int_{|u-v|\leq 10\overline{\ell}}\mathbf{C}_{M_{\gamma(0)}}[\mathcal{T}](t,u)du}_{\textup{\blue{total mass within a small disk}}}.
\end{equation}
When $\overline{\ell}$ is chosen small enough determined by $\gamma(0)$, we have
\begin{equation}
\int_{|u-v|\leq 10\overline{\ell}}\mathbf{C}_{M_{\gamma(0)}}[\mathcal{T}](t,u)du\leq \int_{|u-v|> 10\overline{\ell}}\mathbf{C}_{M_{\gamma(0)}}[\mathcal{T}](t,u)du
\end{equation}
and consequently
\begin{equation}
11\overline{\ell}\int_{|u-v|< 10\overline{\ell}}\mathbf{C}_{M_{\gamma(0)}}[\mathcal{T}](t,u)du\leq \frac{6}{5}\int_{|u-v|>10\overline{\ell}}|v-u|\mathbf{C}_{M_{\gamma(0)}}[\mathcal{T}](t,u)du\leq \frac{6}{5}\mathbf{O}_{M_{\gamma(0)}}[\mathcal{T}](t,v).
\end{equation}
Now the two components of the RHS of \eqref{eq:decompose_perturbed_overlap} have all been bounded from above, and we conclude the proof of Lemma \ref{lem:overlap_regularity}.

\end{proof}

\section{Homogenization Operations and Homogenization Procedure}\label{sec:homogenization_operation}

In this section, we prove several estimates on homogenization operations given in Definition \ref{def:homogenization_operation}.

\begin{condition}{DUO}\label{cond:DUO}
Given a decorated tree $\T$ and a subset $\mathcal{L}_M\subset \mathcal{L}(\T)$ of Maxwellian leaves, we say a collision node $\n$ satisfies the condition \ref{cond:DUO} if
\begin{itemize}
    \vspace{-0.2cm}\item the two collision-child nodes of $\n$ are two non-Maxwellian leaves $\{\l_l,\l_r\}\subset \mathcal{L}(\T)\setminus\mathcal{L}_M$
    \item for any overlap node $\mathfrak{o}\in \mathcal{O}[\l_l,\l_r,\n]$, all of the leaves of the branched sub-tree $\mathscr{T}_{\o}$ are Maxwellian $\mathcal{L}(\mathscr{T}_{\o})\subset \mathcal{L}_M$
\end{itemize}
\end{condition}

\begin{condition}{SIN}\label{cond:SIN}
Consider any decorated tree $\T$, any set $\mathcal{L}_M$ of Maxwellian leaves. We say a collision node $\n\in \mathcal{B}^+(\T)$ satisfies condition \ref{cond:SIN} if
\begin{itemize}
    \vspace{-0.2cm}\item  among the two collision-child nodes of $\n$, one of them is a non-Maxwellian leaf $\mathfrak{l}\in \mathcal{L}(\T)\setminus\mathcal{L}_{M}$. The other one is a branching collision node $\mathfrak{c}\in\mathcal{B}^+(\T)$, and all leaves of the sub-tree $\mathcal{T}_{\mathfrak{c}}$ are Maxwellian $\mathcal{L}(\mathcal{T}_{\mathfrak{c}})\subset \mathcal{L}_M$,
    \item for any overlap node $\mathfrak{o}\in \mathcal{O}[\l,\mathfrak{c},\n]$, all of the leaves of the branched sub-tree $\mathscr{T}_{\o}$ are Maxwellian $\mathcal{L}(\mathscr{T}_{\o})\subset \mathcal{L}_M$.
\end{itemize}
\end{condition}

\begin{condition}{IND}\label{cond:IND}
Consider any decorated tree $\T$ and a subset $\mathcal{L}_M\subset \mathcal{L}(\T)$ of Maxwellian leaves, we say a non-Maxwellian leaf $\l\in \mathcal{L}(\T)\setminus \mathcal{L}_M$ satisfies condition \ref{cond:IND} if for any overlap node $\o\in \mathcal{O}[\l]$, all of the leaves of the branched sub-tree $\mathscr{T}_{\o}$ are Maxwellian $\mathcal{L}(\mathscr{T}_{\o})\subset \mathcal{L}_M$.
\end{condition}

In this section we frequently involve overlap functions introduced in Definition \ref{def:overlap_function}. Their regularity are estimated by Lemma \ref{lem:overlap_regularity}, contributing the error terms of the form $10^{|\mathcal{O}[...]|}$ in this section's Lemmas.

\subsection{Double Homogenization}\label{sec:double_homo}

\begin{lemma}\label{lem:double_collision}
Consider any constants $(\gamma>0,\delta>0,T>0)$, and any non-negative function $h(t,x,v)$ with
\begin{equation}
h\in L^{\infty}\ldelta\big([0,T]\times \mathds{R}^3\times \mathds{R}^3\big).
\end{equation}
There exists a constant $0<\gamma(0)<\gamma$ such that the following statement is true. Consider any decorated tree $\T$, any set $\mathcal{L}_M\subset \mathcal{L}(\T)$ of Maxwellian leaves. Assume a collision node $\n$ satisfies condition \ref{cond:DUO} and has a collision-parent node $\m$. Then for any $(t,x,v)\in [0,T]\times \mathds{R}^3\times \mathds{R}^3$
\begin{equation}\label{eq:double_collision}
\begin{split}
&\mathbf{C}_{h}\Big[\T,\mathcal{L}_M=M_{\gamma(0)}\Big](t,x,v)\leq 10^{|\mathcal{O}[\l_l,\l_r,\n]|}\Big(C_{\gamma,\delta}\lVert h\rVert_{\timeldelta}\Big)^2\mathbf{C}_{h}\Big[\T,\mathcal{L}_M\cup\{\l_l,\l_r\}=M_{\gamma(0)}\Big](t,x,v)
\end{split}
\end{equation}

\end{lemma}

\begin{figure}[h]
    \centering
    \includegraphics[width=0.85\linewidth]{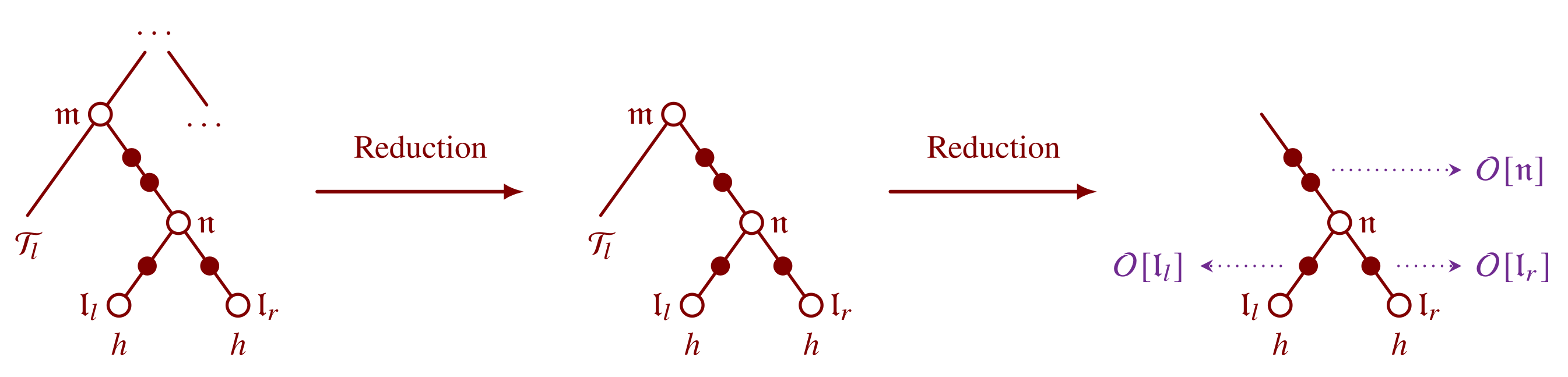}
    \caption{ The inductive definition of Boltzmann integrals in Definition \ref{def:C_expansion_replacement} allows us to reduce the estimate of $\mathbf{C}_{h}[\mathcal{T},\mathcal{L}_M=M_{\gamma(0)}]$ for $\mathcal{T}$ (left), to the estimate of $\mathbf{C}_{h}[\mathcal{T}_r,\mathcal{L}_M=M_{\gamma(0)}]$ (right). Here $\mathcal{T}_r$ is a sub-tree containing the collision node $\mathfrak{n}$.}
    \label{fig:double_collision}
\end{figure}

\begin{proof}

\underline{\textit{Step 1: Reduction of estimate.}}

According to the inductive definition of $\mathbf{C}_{h}[\mathcal{T},\mathcal{L}_M=M_{\gamma(0)}]$ in Definition \ref{def:C_expansion_replacement}, the point-wise estimate of $\mathbf{C}_h[\mathcal{T},\mathcal{L}_M=M_{\gamma(0)}]$ can be reduced to the point-wise estimate of $\mathbf{C}_h[\T_{\mathfrak{m}},\mathcal{L}_M=M_{\gamma(0)}]$. Here $\mathcal{T}_{\m}$ is the sub-tree of $\T$ with root $\m$. Consequently it is sufficient to prove the following inequality for any $(t,x,v)\in [0,T]\times \mathds{R}_x^3\times\mathds{R}_v^3$.
\begin{equation}\label{eq:Tm_before_expansion}
\mathbf{C}_{h}\big[\mathcal{T}_{\mathfrak{m}},\blue{\mathcal{L}_M=M_{\gamma(0)}}\big](t,x,v)\leq 10^{|\mathcal{O}[\mathfrak{l}_r,\l_l,\mathfrak{n}]|}\Big(C_{\gamma,\delta}\lVert h\rVert_{\timeldelta}\Big)^2\mathbf{C}_{h}\big[\mathcal{T}_{\mathfrak{m}},\red{\mathcal{L}_M\cup\{\mathfrak{l}_r,\mathfrak{l}_l\}=M_{\gamma(0)}}\big](t,x,v).
\end{equation}
Equation \eqref{eq:Tm_before_expansion} reduces the estimate of $\T$ into the estimate of a sub-tree $\mathcal{T}_{\mathfrak{m}}$, which corresponds to the first reduction in Figure \ref{fig:double_collision}. Denote $\T_l$ as the left sub-tree of $\T_{\mathfrak{m}}$, and $\T_r$ as the right sub-tree of $\T_{\mathfrak{m}}$. It is convenient to assume $\mathfrak{n}\in \mathcal{T}_{r}$, and therefore $\{\mathfrak{l}_l,\mathfrak{l}_r\}\subset \mathcal{L}(\mathcal{T}_r)$. Expand the LHS of \eqref{eq:Tm_before_expansion}
\begin{equation*}
\begin{split}
&\mathbf{C}_{h}[\T_{\m},\blue{\mathcal{L}_M=M_{\gamma(0)}}](t,x,v)=\int_{0}^{t}\mathbf{S}_{t-s}\mathbf{Q}^+\Big(\underbrace{\mathbf{C}_{h}\Big[\T_l,\mathcal{L}_M=M_{\gamma(0)}\Big](s)}_{\textup{\blue{left sub-tree}}},\underbrace{\mathbf{C}_{h}\Big[\T_r,\blue{\mathcal{L}_M=M_{\gamma(0)}}\Big](s)}_{\textup{\blue{right sub-tree}}}\Big)ds,
\end{split}
\end{equation*}
and the RHS of \eqref{eq:Tm_before_expansion}
\begin{equation}
\begin{split}
&\mathbf{C}_{h}[\T_{\m},\red{\mathcal{L}_M\cup\{\mathfrak{l}_r,\mathfrak{l}_l\}=M_{\gamma(0)}}](t,x,v)=\int_{0}^{t}\mathbf{S}_{t-s}\mathbf{Q}^+\Big(\underbrace{\mathbf{C}_{h}\Big[\T_l,\mathcal{L}_M=M_{\gamma(0)}\Big](s)}_{\textup{\blue{left sub-tree}}},\underbrace{\mathbf{C}_{h}\Big[\T_r,\red{\mathcal{L}_M\cup\{\mathfrak{l}_r,\mathfrak{l}_l\}=M_{\gamma(0)}}\Big](s)}_{\textup{\blue{right sub-tree}}}\Big)ds
\end{split}
\end{equation}
Consequently to prove the estimate \eqref{eq:Tm_before_expansion} of $\mathcal{T}_{\m}$, it suffices to prove for any $(t,x,v)\in [0,T]\times \mathds{R}^3\times\mathds{R}^3$
\begin{equation}\label{eq:Q+_before_expansion}
\begin{split}
&\mathbf{Q}^+\Big(\mathbf{C}_{h}[\T_l,\mathcal{L}_M=M_{\gamma(0)}](t),\mathbf{C}_{h}[\T_r,\blue{\mathcal{L}_M=M_{\gamma(0)}}](t)\Big)(x,v)\\
&\leq 10^{|\mathcal{O}[\mathfrak{l}_r,\l_l,\mathfrak{n}]|}\Big(C_{\gamma,\delta}\lVert h\rVert_{\timeldelta}\Big)^2 \mathbf{Q}^+\Big(\mathbf{C}_{h}[\T_l,\mathcal{L}_M=M_{\gamma(0)}](t),\mathbf{C}_{h}[\T_r,\red{\mathcal{L}_M\cup\{\mathfrak{l}_r,\l_l\}=M_{\gamma(0)}}](t)\Big)(x,v)
\end{split}
\end{equation}
Expand the operator $\mathbf{Q}^+$ in the LHS of \eqref{eq:Q+_before_expansion}
\begin{equation*}
\begin{split}
&\mathbf{Q}^+\Big(\mathbf{C}_{h}[\T_l,\mathcal{L}_M=M_{\gamma(0)}](t),\mathbf{C}_{h}[\T_r,\blue{\mathcal{L}_M=M_{\gamma(0)}}](t)\Big)(x,v)\\
=&\int \underbrace{\mathbf{C}_{h}\Big[\T_l,\mathcal{L}_M=M_{\gamma(0)}\Big](t,x,v_{1})}_{\textup{\blue{left sub-tree: free integral}}}\underbrace{\mathbf{C}_{h}\Big[\T_r,\blue{\mathcal{L}_M=M_{\gamma(0)}}\Big](t,x,v_{2})}_{\textup{\blue{right sub-tree: integral on a plane}}}|v_1-v|^{-1}d\pi_{v+(v_{1}-v)^{\perp}}(v_{2})dv_{1},
\end{split}
\end{equation*}
and the operator $\mathbf{Q}^+$ in the RHS of \eqref{eq:Q+_before_expansion}
\begin{equation*}
\begin{split}
&\mathbf{Q}^+\Big(\mathbf{C}_{h}[\T_l,\mathcal{L}_M=M_{\gamma(0)}](t),\mathbf{C}_{h}[\T_r,\red{\mathcal{L}_M\cup\{\mathfrak{l}_r,\mathfrak{l}_l\}=M_{\gamma(0)}}](t)\Big)(x,v)\\
=&\int \underbrace{\mathbf{C}_{h}\Big[\T_l,\mathcal{L}_M=M_{\gamma(0)}\Big](t,x,v_{1})}_{\textup{\blue{left sub-tree: free integral}}}\underbrace{\mathbf{C}_{h}\Big[\T_r,\red{\mathcal{L}_M\cup\{\mathfrak{l}_r,\mathfrak{l}_l\}=M_{\gamma(0)}}\Big](t,x,v_{2})}_{\textup{\blue{right sub-tree: integral on a plane}}}|v_1-v|^{-1}d\pi_{v+(v_{1}-v)^{\perp}}(v_{2})dv_{1},
\end{split}
\end{equation*}
Notice that the singular function $|v_1-v|^{-1}$ does not involve $v_2$. We can fix the variables $(v,v_1)$ and estimate the integral against $v_2$. Consequently to prove \eqref{eq:Tm_before_expansion}, it suffices to prove for any $(t,x)\in [0,T]\times \mathds{R}^3$ and any plane $\mathtt{PL}\subset \mathds{R}^3$
\begin{equation}\label{eq:PL_integration}
\int \mathbf{C}_{h}\Big[\T_r,\blue{\mathcal{L}_{M}=M_{\gamma(0)}}\Big](t,x,v)d\pi_{\mathtt{PL}}(v)\leq \int 10^{|\mathcal{O}[\mathfrak{l}_r,\mathfrak{l}_l,\mathfrak{n}]|}\Big(C_{\gamma,\delta}\lVert h\rVert_{\timeldelta}\Big)^2\mathbf{C}_{h}\Big[\T_r,\red{\mathcal{L}_{M}\cup\{\mathfrak{l}_r,\mathfrak{l}_l\}=M_{\gamma(0)}}\Big](t,x,v)d\pi_{\mathtt{PL}}(v)
\end{equation}
Equation \eqref{eq:PL_integration} corresponds to the second reduction in Figure \ref{fig:double_collision}.

\underline{\textit{Step 2: Geometric change of variables.}}

Here in Section \ref{sec:double_homo}, we denote $\mathbf{OV}(v_{l},v_r,v_{\n})$ as
\begin{equation}
\mathbf{OV}(v_{l},v_r,v_{\n}):=\prod_{\o\in\mathcal{O}[\l_l]}\mathbf{O}_{M_{\gamma(0)}}[\mathscr{T}_{\o}](s_{\o},v_{l})\prod_{\o\in\mathcal{O}[\l_r]}\mathbf{O}_{M_{\gamma(0)}}[\mathscr{T}_{\o}](s_{\o},v_{r})\prod_{\o\in\mathcal{O}[\n]}\mathbf{O}_{M_{\gamma(0)}}[\mathscr{T}_{\o}](s_{\o},v_{\n})
\end{equation}
The regularity of $\mathbf{OV}(v_l,v_r,v_\n)$ can be estimated by Lemma \ref{lem:overlap_regularity}. Using Lemma \ref{lem:full_to_prime}, we expand the LHS of \eqref{eq:PL_integration} as the following, to filter the effects of overlap functions. \blue{We omit the integral of time variables because the estimates are uniform for all time variables, and the time variables can be taken as any fixed value }
\begin{equation}\label{eq:first_expansion}
\begin{split}
\int &\mathbf{OV}(v_l,v_r,v_\n)h(s_{\mathcal{P}l},\x_{\mathcal{P}l},v_{l})h(s_{\mathcal{P}r},\x_{\mathcal{P}r},v_{r})|v_l-v_{\n}|^{-1}d\pi_{v_\n+(v_{l}-v_\n)^{\perp}}(v_{r})dv_{l}d\pi_{\mathtt{PL}}(v_\n).
\end{split}
\end{equation}
where $\mathbf{O}_{M_{\gamma(0)}}[\mathscr{T}_{\o}](s,v)$ is the overlap function given by Definition \ref{def:overlap_function}. In \eqref{eq:first_expansion}, we first integrate $v_{\n}$ over the plane $\mathtt{PL}$, then integrate $v_l$ freely, and eventually integrate $v_{r}$ over the orthogonal plane $v_{\n}+(v_l-v_{\n})^{\perp}$. Perform the change of variables to $\{v_{r},v_{l},v_{\n}\}$ according to Lemma \ref{lem:geometric_change_of_variable}, where $v_l$ and $v_r$ become free variables
\begin{equation}\label{eq:sp_velocity}
|v_l-v_{\n}|^{-1}d\pi_{v_\n+(v_{l}-v_\n)^{\perp}}(v_{r})dv_{l}d\pi_{\mathtt{PL}}(v_\n)=\pi\mathds{1}_{\mathtt{S}_{v_l,v_r}\cap\mathtt{PL}\not=\varnothing}\overline{d\pi}_{\mathtt{S}_{v_l,v_r}\cap \mathtt{PL}}(v_\n)dv_{r}dv_{l}.
\end{equation}
Therefore now we need to estimate
\begin{equation*}
\begin{split}
\int &\mathbf{OV}(v_l,v_r,v_{\n})h(s_{\mathcal{P}l},\x_{\mathcal{P}l},v_l)h(s_{\mathcal{P}r},\x_{\mathcal{P}r},v_r)\mathds{1}_{\mathtt{S}_{v_l,v_r}\cap \mathtt{PL}\not=\varnothing}\overline{d\pi}_{\mathtt{S}_{v_l,v_r}\cap\mathtt{PL}}(v_{\n})dv_rdv_l
\end{split}
\end{equation*}
To estimate this integral, we decompose the space $\mathds{R}^3$ of velocities into small cubes $\mathds{R}^3=\bigcup_{\mathtt{p}\in \ell \mathds{Z}^3}Q_{\mathtt{p}}^{\ell}$, with $Q_{\mathtt{p}}^{\ell}$ defined in \eqref{eq:cube_definition}. We want to estimate for each $(\mathtt{p}_l,\mathtt{p}_r)\in \ell \mathds{Z}^3\times \ell\mathds{Z}^3$
\begin{equation}\label{eq:second_expansion}
\begin{split}
&\int_{\plq\times \prq}\mathbf{OV}(v_l,v_r,v_{\n})h(s_{\mathcal{P}l},\x_{\mathcal{P}l},v_l)h(s_{\mathcal{P}r},\x_{\mathcal{P}r},v_r)\mathds{1}_{\mathtt{S}_{v_l,v_r}\cap \mathtt{PL}\not=\varnothing}(v_l,v_r)\overline{d\pi}_{\mathtt{S}_{v_l,v_r}\cap\mathtt{PL}}(v_{\n})dv_rdv_l.
\end{split}
\end{equation}

\underline{\textit{Step 3: Functions with variables $v_{l}$ and $v_{r}$.}}

Denote $\mathtt{e}_1$ and $\mathtt{e}_2$ as two orthogonal unit vectors spanning the plane $\mathtt{PL}$. When $\mathtt{S}_{v_l,v_r}\cap \mathtt{PL}\not=\varnothing$, we define $\uprho_{\mathtt{PL}}(v_{l},v_{r})$ as the radius of $\mathtt{S}_{v_{l},v_{r}}\cap \mathtt{PL}$, and $\mathtt{c}_{\mathtt{PL}}(v_{l},v_{r})$ as the center of $\mathtt{S}_{v_{l},v_{r}}\cap \mathtt{PL}$. Here $\uprho_{\mathtt{PL}}(v_{l},v_{r})$ and $\mathtt{c}_{\mathtt{PL}}(v_{l},v_{r})$ are all functions of $v_{l}$ and $v_{r}$. Notice that when $\mathtt{S}_{v_l,v_r}$ is tangential to $\mathtt{PL}$, the intersection $\mathtt{S}_{v_l,v_r}\cap \mathtt{PL}$ is a single point, with $\uprho_{\mathtt{PL}}(v_l,v_r)=0$ and $\mathtt{c}_{\mathtt{PL}}(v_l,v_r)=\mathtt{S}_{v_l,v_r}\cap\mathtt{PL}$. Using these notations, the integration over the measure $\overline{d\pi}_{\mathtt{S}\cap \mathtt{PL}}$ can be parametrized by $\uptheta\in [0,2\pi]$
\begin{equation}\label{eq:parametrize_the_circle}
\int H(v_{\n})\overline{d\pi}_{\mathtt{S}_{v_l,v_r}\cap \mathtt{PL}}(v_{\n})=\frac{1}{2\pi}\int_0^{2\pi}H\big(v_{\n}(\uptheta,v_{l},v_{r})\big)d\uptheta,\quad v_{\n}(\uptheta,v_{l},v_{r}):=\mathtt{c}_{\mathtt{PL}}(v_{l},v_{r})+\uprho_{\mathtt{PL}}(v_{l},v_{r})\Big(\mathtt{e}_1\cos\uptheta+\mathtt{e}_2\sin\uptheta\Big)
\end{equation}
with $v_{\n}(\uptheta,v_{l},v_{r})$ as a function of $(\uptheta,v_{l},v_{r})$. For the integral in \eqref{eq:second_expansion}, when we fix $(v_l,v_r)$, the integrand 
\begin{equation}\label{eq:integrand_function_of_vn}
\mathds{1}_{\mathtt{S}_{v_l,v_r}\cap \mathtt{PL}\not=\varnothing}(v_l,v_r)\mathbf{OV}\Big(v_l,v_r,v_{\n}(\uptheta,v_l,v_r)\Big)h(s_{\mathcal{P}l},\x_{\mathcal{P}l},v_l)h(s_{\mathcal{P}r},\x_{\mathcal{P}r},v_r)
\end{equation}
can be seen as a function of $v_{\n}$. Specifically according to Remark \ref{rmk:pseudo-trajectory_and_xfunction}, the position functions $\x_{\mathcal{P}l}$ and $\x_{\mathcal{P}r}$ satisfy
\begin{equation*}
\x_{\mathcal{P}l}:=x-(t-s_{\n})v_{\n}-(s_{\n}-s_{\mathcal{P}l})v_l,\quad \x_{\mathcal{P}r}:=x-(t-s_{\n})v_{\n}-(s_{\n}-s_{\mathcal{P}r})v_r.
\end{equation*}
Therefore we can apply \eqref{eq:parametrize_the_circle} to \eqref{eq:integrand_function_of_vn} and parametrize the integrand in \eqref{eq:integrand_function_of_vn} by $\uptheta$. Define set $E_{\mathtt{PL}}$ as
\begin{equation}\label{eq:E_PL}
E_{\mathtt{PL}}:=\{(v_l,v_r)\in\mathbb R^3\times\mathbb R^3:\mathtt{S}_{v_l,v_r}\cap \mathtt{PL}\neq\varnothing\}.
\end{equation}
In this way, the integral in \eqref{eq:second_expansion} equals
\begin{equation}\label{eq:parametrized_integral_double_collision}
\frac{1}{2\pi}\int_0^{2\pi}\Bigg[\int_{(\plq\times \prq)\cap E_{\mathtt{PL}}} \mathbf{OV}\Big(v_l,v_r,v_{\n}(\uptheta,v_l,v_r)\Big)h\Big(s_{\mathcal{P}l},\x_{\mathcal{P}l}(\uptheta,v_l,v_r),v_l\Big)h\Big(s_{\mathcal{P}r},\x_{\mathcal{P}r}(\uptheta,v_l,v_r),v_r\Big)dv_{r}dv_{l}\Bigg]d\uptheta.
\end{equation}
with position functions as
\begin{equation}
\x_{\mathcal{P}l}(\uptheta,v_l,v_r)=x-(t-s_{\n})v_{\n}(\uptheta,v_l,v_r)-(s_{\n}-s_{\mathcal{P}l})v_{l},\quad \x_{\mathcal{P}r}(\uptheta,v_l,v_r)=x-(t-s_{\n})v_{\n}(\uptheta,v_l,v_r)-(s_{\n}-s_{\mathcal{P}r})v_{r}.
\end{equation}

\textit{Step 4: Estimate of perturbations for positions.}

When the velocities $(v_l,v_r)$ are perturbed within the cube pair $\plq\times \prq$, we can estimate the perturbation of $v_{\n}$
\begin{equation}\label{eq:perturbation_position_double_collision}
\sup_{(v_l,v_r),(v_l',v_r')\in \big(\plq\times \prq\big)\cap E_{\mathtt{PL}}}\Big|v_{\n}(\uptheta,v_l,v_r)-v_{\n}(\uptheta,v_l',v_r')\Big|\leq C\ell^{\frac{1}{2}}\Delta, \quad \Delta:=\max\{1,|\mathtt{p}_l|,|\mathtt{p}_r|\}, 
\end{equation}
and estimate the corresponding perturbations of the position $\x_{\mathcal{P}l}(\uptheta,v_l,v_r)$ (with $\ell$ chosen as small enough determined by $\gamma$ and $\delta$)
\begin{equation}
\sup_{(v_l,v_r),(v_l',v_r')\in \big(\plq\times \prq\big)\cap E_{\mathtt{PL}}}\Big|\x_{\mathcal{P}l}(\uptheta,v_l',v_r')-\x_{\mathcal{P}l}(\uptheta,v_l,v_r)\Big|\leq \delta\Delta.
\end{equation}
and the same estimate is true for $\x_{\mathcal{P}r}(\uptheta,v_l,v_r)$. If $(Q_{\mathtt{p}_l}^{\ell}\times Q_{\mathtt{p}_r}^{\ell})\cap E_{\mathtt{PL}}=\varnothing$, then the integral in \eqref{eq:parametrized_integral_double_collision} is $0$. If $(Q_{\mathtt{p}_l}^{\ell}\times Q_{\mathtt{p}_r}^{\ell})\cap E_{\mathtt{PL}}\not=\varnothing$, then we denote the point $(u_l,u_r)\in (Q_{\mathtt{p}_l}^{\ell}\times Q_{\mathtt{p}_r}^{\ell})\cap E_{\mathtt{PL}}$ as the point such that
\begin{equation}
\mathbf{OV}(u_l,u_r,v_{\n}(\uptheta,u_l,u_r))=\sup_{(v_l,v_r)\in (Q_{\mathtt{p}_l}^{\ell}\times Q_{\mathtt{p}_r}^{\ell})\cap E_{\mathtt{PL}}}\mathbf{OV}(v_l,v_r,v_{\n}(\uptheta,v_l,v_r)).
\end{equation}
The \eqref{eq:perturbation_position_double_collision} shows as we integrate $v_l$ and $v_r$, the perturbation of position functions are controlled. Thus we can replace $h$ by its local supremum $[h]_{\delta\Delta}$ around positions $\xp_l:=\mathbf{x}_{\mathcal{P}l}(\uptheta,u_l,u_r)$ and $\xp_r:=\mathbf{x}_{\mathcal{P}r}(\uptheta,u_l,u_r)$, and replace $\mathbf{OV}$ by its supremum at $(u_l,u_r,v_\n(\uptheta,u_l,u_r))$. To simplify notations, here in the proof we adopt the following abbreviations:
\begin{equation}
h_l(v_l):=[h(s_{\mathcal{P}l})]_{\delta \Delta}(\xp_l,v_l),\quad h_r(v_r):=[h(s_{\mathcal{P}r})]_{\delta \Delta}(\xp_r,v_r).
\end{equation}
Then we have
\begin{equation}\label{eq:double_collision_final}
\begin{split}
&\int_{(\plq\times \prq)\cap E_{\mathtt{PL}}} \mathbf{OV}(v_l,v_r,v_{\n}(\uptheta,v_l,v_r))h\Big(s_{\mathcal{P}l},\x_{\mathcal{P}l}(\uptheta,v_l,v_r),v_l\Big)h\Big(s_{\mathcal{P}r},\x_{\mathcal{P}r}(\uptheta,v_l,v_r),v_r\Big)dv_{r}dv_{l}\\
&\leq \int_{(\plq\times \prq)\cap E_{\mathtt{PL}}} \mathbf{OV}(v_l,v_r,v_{\n}(\uptheta,v_l,v_r))h_l(v_l)h_r(v_r)dv_{r}dv_{l}\leq \mathbf{OV}(u_l,u_r,v_{\n}(\uptheta,u_l,u_r))\Big(\int_{Q_{\mathtt{p}_l}^{\ell}}h_l(v_l)dv_l\Big)\Big(\int_{Q_{\mathtt{p}_r}^{\ell}} h_r(v_r)dv_r\Big)
\end{split}
\end{equation}

\textit{Step 5: Construction of a not-so-small set within $\big(Q_{\pl}^{2\ell}\times Q_{\pr}^{2\ell}\big)\cap E_{\mathtt{PL}}$.}

The second inequality in \eqref{eq:double_collision_final} is actually tight in the following sense: the condition of $h\in \timeldelta$ implies only $h_l,h_r\in L_{\gamma}^1$, and the functions $h_l$ and $h_r$ can be arbitrarily concentrated within $Q_{\p_l}^{\ell}$ and $Q_{\p_r}^{\ell}$. Thus even if $\big(Q_{\pl}^{\ell}\times Q_{\pr}^{\ell}\big)\cap E_{\mathtt{PL}}$ can be very small, we can not gain anything from here. 

Given any $v_r$, define $E_{\mathtt{PL},v_r}$ as
\begin{equation}
E_{\mathtt{PL},v_r}:=\Big\{v_l\in \mathds{R}^3\Big| \mathtt{S}_{v_l,v_r}\cap \mathtt{PL}\not=\varnothing\Big\}.
\end{equation}
The set $E_{\mathtt{PL},v_r}$ can be equivalently characterized by
\begin{equation}
\textup{dist}\big(\frac{v_l+v_r}{2},\mathtt{PL}\big)\leq |\frac{v_l-v_r}{2}|\quad \textup{or}\quad \dist\Big(v_l,2\mathtt{PL}-v_r\Big)\leq |v_l-v_r|.
\end{equation}
Given $v_r$, the boundary of $E_{\mathtt{PL},v_r}$ is a paraboloid where each point has equal distance to $2\mathtt{PL}-v_r$ and $v_r$. Thus the complement $E_{\mathtt{PL},v_r}^c$ is a convex set.

Now we construct a set $B$ such that
\begin{equation}
B:=E_{\mathtt{PL}}\cap\Bigg(Q_{u_l}^{\ell\Delta^{-1}}\times Q_{u_r}^{\ell\Delta^{-1}}\Bigg).
\end{equation}
The convexity of $E_{\mathtt{PL},v_r}^c$ for any $v_r$ implies $|B|\geq C\Big(\frac{\ell}{\Delta}\Big)^6$. Compute the perturbations of velocities
\begin{equation}
\sup_{(v_l,v_r)\in B}\Big(|v_l-u_l|+|v_r-u_r|+|v_{\n}(v_l,v_r)-v_{\n}(u_l,u_r)|\Big)\leq C\ell^{\frac{1}{2}}.
\end{equation}
The regularity of overlap functions in Lemma \ref{lem:overlap_regularity} implies
\begin{equation}
\inf_{(v_l,v_r)\in B}\mathbf{OV}\big(v_l,v_r,v_{\n}(\uptheta,v_l,v_r)\big)\geq 3^{-|\mathcal{O}[\l_l,\l_r,\n]|}\mathbf{OV}\big(u_l,u_r,v_\n(\uptheta,u_l,u_r)\big).
\end{equation}
Notice
\begin{equation}
\lVert h_l\rVert_{L_{\gamma}^1}\leq C\Delta^3\lVert h\rVert_{\timeldelta},\quad \lVert h_r\rVert_{L_{\gamma}^1}\leq C\Delta^3\lVert h\rVert_{\timeldelta}.
\end{equation}
The smallness of $|B|$ and the factor $\Delta^3$ can be absorbed by the Gaussian weight
\begin{equation}\label{eq:absorb_by_gaussian_weight}
\begin{split}
\Big(\int_{Q_{\mathtt{p}_l}^{\ell}}h_l(v_l)dv_l\Big)\Big(\int_{Q_{\mathtt{p}_r}^{\ell}}h_r(v_r)dv_r\Big)&\leq \Bigg[\Big(\int_{Q_{\mathtt{p}_l}^{\ell}}h_l(v_l)e^{\gamma|v_l|^2}dv_l\Big)\sup_{v_l\in Q_{\mathtt{p}_l}^{\ell}}e^{-\gamma|v_l|^2}\Bigg]\Bigg[\Big(\int_{Q_{\mathtt{p}_r}^{\ell}}h_r(v_r)e^{\gamma|v_r|^2}dv_r\Big)\sup_{v_r\in Q_{\mathtt{p}_r}^{\ell}}e^{-\gamma|v_r|^2}\Bigg]\\
&\leq \Big(C_{\gamma,\delta}\lVert h\rVert_{L^{\infty}\ldelta}\Big)^2\int_{B}M_{\gamma(0)}(v_l)M_{\gamma(0)}(v_r)dv_ldv_r
\end{split}
\end{equation}
Thus
\begin{equation}\label{eq:upper_bound_for_each_pair}
\begin{split}
&\int_{\big(\plq\times \prq\big)\cap E_{\mathtt{PL}}} \mathbf{OV}\big(v_l,v_r,v_{\n}(\uptheta,v_l,v_r)\big)h\Big(s_{\mathcal{P}l},\x_{\mathcal{P}l}(\uptheta,v_l,v_r),v_l\Big)h\Big(s_{\mathcal{P}r},\x_{\mathcal{P}r}(\uptheta,v_l,v_r),v_r\Big)dv_{r}dv_{l}\\
&\leq 10^{|\mathcal{O}[\l_l,\l_r,\n]|}\Big(C_{\gamma,\delta}\lVert h\rVert_{\timeldelta}\Big)^2\int_{(Q_{\mathtt{p}_l}^{2\ell}\times Q_{\mathtt{p}_r}^{2\ell})\cap E_{\mathtt{PL}}} \mathbf{OV}\big(v_l,v_r,v_\n(\uptheta,v_l,v_r)\big)M_{\gamma(0)}(v_l)M_{\gamma(0)}(v_r)dv_{r}dv_{l}
\end{split}
\end{equation}
Sum \eqref{eq:upper_bound_for_each_pair} over all $(\pl,\pr)\in \ell \mathds{Z}^3\times \ell \mathds{Z}^3$, and we prove the Lemma \ref{lem:double_collision}.

\end{proof}

\begin{lemma}\label{lem:double_overlap}
Consider any constants $(\gamma>0,\delta>0,T>0)$, and any non-negative function $h(t,x,v)$ within the space $L^{\infty}\ldelta\big([0,T]\times \mathds{R}^3\times \mathds{R}^3\big)$. There exists a constant $0<\gamma(0)<\gamma$ such that the following statement is true. For any decorated tree $\T$, any set $\mathcal{L}_M\subset \mathcal{L}(\T)$ of Maxwellian leaves, and any collision node $\n$, suppose $\n$ satisfies condition \ref{cond:DUO} and \underline{has an overlap-parent node} $\o_*\in \mathcal{O}(\T)$. Then for any $(t,x,v)\in [0,T]\times \mathds{R}^3\times \mathds{R}^3$
\begin{equation}\label{eq:double_overlap}
\begin{split}
&\mathbf{C}_{h}\Big[\T,\mathcal{L}_M=M_{\gamma(0)}\Big](t,x,v)\leq 10^{|\mathcal{O}[\l_l,\l_r,\n]|}\Big(C_{\gamma,\delta}\lVert h\rVert_{\timeldelta}\Big)^2\mathbf{C}_{h}\Big[\T,\mathcal{L}_M\cup\{\l_l,\l_r\}=M_{\gamma(0)}\Big](t,x,v)
\end{split}
\end{equation}
\end{lemma}

\begin{figure}[h]
    \centering
    \includegraphics[width=0.85\linewidth]{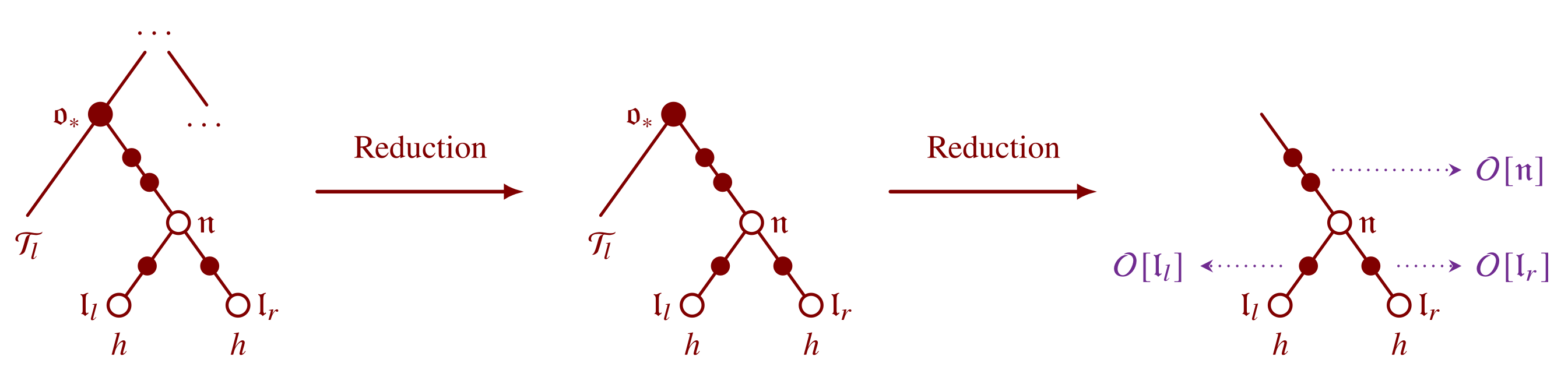}
    \caption{The estimate of $\mathbf{C}_h[\T,\mathcal{L}_M=M_{\gamma(0)}]$ (see left) can be reduced to the estimate of $\mathbf{C}_h[\T_*,\mathcal{L}_M=M_{\gamma(0)}]$, where $\T_*$ is a sub-tree containing $\n$ (see right).}
    \label{fig:double_overlap}
\end{figure}
\begin{proof}

\textit{Step 1: Reduction of estimate.}

In the \textit{step 1} of the proof of Lemma \ref{lem:double_collision}, the estimate of $\mathbf{C}_h[\T,\mathcal{L}_M=M_{\gamma(0)}]$ can be reduced to the Boltzmann integral associated with a sub-tree containing $\n$. More precisely, to prove \eqref{eq:double_collision}, it suffices to prove \eqref{eq:PL_integration}. This reduction has been illustrated by Figure \ref{fig:double_collision}.

It is the same for Lemma \ref{lem:double_overlap}. By the assumption of $\o_*$ as the overlap-parent of $\n$ in the statement of Lemma \ref{lem:double_overlap}, we know the node $\n$ is in the sub-tree $\mathscr{T}_{\o_*}$ branched from $\o_*$. For simplicity here in the proof we denote $\mathscr{T}_{\o_*}$ as $\mathcal{T}_*$. Therefore to prove Lemma \ref{lem:double_overlap}, it suffices to prove for any $(t,x)\in [0,T]\times \mathds{R}^3$ and any $\mathbf{c}\in\mathds{R}^3$
\begin{equation}\label{eq:R3c_integration}
\begin{split}
&\int |\mathbf{c}-v|\mathbf{C}_{h}\Big[\T_*,\blue{\mathcal{L}_{M}=M_{\gamma(0)}}\Big](t,x,v)dv\\
&\leq 10^{|\mathcal{O}[\mathfrak{l}_r,\mathfrak{l}_l,\mathfrak{n}]|}\Big(C_{\gamma,\delta}\lVert h\rVert_{\timeldelta}\Big)^2 \int |v-\mathbf{c}|\mathbf{C}_{h}\Big[\T_*,\red{\mathcal{L}_{M}\cup\{\mathfrak{l}_r,\mathfrak{l}_l\}=M_{\gamma(0)}}\Big](t,x,v)dv
\end{split}
\end{equation}
Using Lemma \ref{lem:full_to_prime}, we expand the LHS of \eqref{eq:R3c_integration} as the following, to filter the effects of overlaps, \blue{and again we freeze all collision times and overlap times}
\begin{equation}\label{eq:first_expansion_sigmac}
\begin{split}
\int &|v_{\n}-\mathbf{c}|\mathbf{OV}(v_l,v_r,v_\n)\underbrace{h(s_{\mathcal{P}l},\x_{\mathcal{P}l},v_{l})h(s_{\mathcal{P}r},\x_{\mathcal{P}r},v_{r})|v_l-v_{\n}|^{-1}d\pi_{v_{\n}+(v_l-v_{\n})^{\perp}}(v_r)dv_{l}dv_{\n}}_{\textup{\blue{we need to decouple the variables $v_r$ and $v_l$ in some sense}}}.
\end{split}
\end{equation}
where $\mathbf{O}_{M_{\gamma(0)}}[\mathscr{T}_{\o}](s,v)$ is the overlap function given by Definition \ref{def:overlap_function}. 

\textit{Step 2: Geometric change of variables.}

Perform the change of variables to $\{v_{r},v_{l},v_{\n}\}$ according to Lemma \ref{lem:geometric_change_of_variable_2}, where $v_l$ and $v_r$ become free variables
\begin{equation}
|v_l-v_{\n}|^{-1}d\pi_{v_\n+(v_{l}-v_\n)^{\perp}}(v_{r})dv_{l}dv_{\n}=\pi\overline{d\pi}_{\mathtt{S}_{v_l,v_r}}(v_\n)|v_l-v_r|dv_{r}dv_{l}.
\end{equation}
Therefore now we need to estimate
\begin{equation*}
\begin{split}
\int &|v_l-v_r|\ |v_{\mathfrak{n}}-\mathbf{c}|\mathbf{OV}(v_l,v_r,v_\n)h(s_{\mathcal{P}l},\x_{\mathcal{P}l},v_l)h(s_{\mathcal{P}r},\x_{\mathcal{P}r},v_r)\overline{d\pi}_{\mathtt{S}_{v_l,v_r}}(v_{\n})dv_rdv_l
\end{split}
\end{equation*}
To estimate this integral, we decompose the space $\mathds{R}^3$ of velocities into small cubes as in \eqref{eq:cube_definition}. We want to estimate for each $(\mathtt{p}_l,\mathtt{p}_r)\in \ell \mathds{Z}^3\times \ell\mathds{Z}^3$
\begin{equation}\label{eq:second_expansion_sigmac}
\begin{split}
&\int_{\plq\times \prq}|v_l-v_r|\ |v_{\n}-\mathbf{c}|\mathbf{OV}(v_l,v_r,v_\n)h(s_{\mathcal{P}l},\x_{\mathcal{P}l},v_l)h(s_{\mathcal{P}r},\x_{\mathcal{P}r},v_r)\overline{d\pi}_{\mathtt{S}_{v_l,v_r}}(v_{\n})dv_rdv_l.
\end{split}
\end{equation}
Define $\uprho_{\mathtt{S}}(v_{l},v_{r})$ as the radius of $\mathtt{S}_{v_{l},v_{r}}$, and $\mathtt{c}_{\mathtt{S}}(v_{l},v_{r})$ as the center of $\mathtt{S}_{v_{l},v_{r}}$. Using these notations, the integration over the measure $\overline{d\pi}_{\mathtt{S}}$ can be parametrized by $\upsigma\in \mathds{S}^2$
\begin{equation}\label{eq:v_n_asfunction_sigma}
\int H(v_{\n})\overline{d\pi}_{\mathtt{S}_{v_l,v_r}}(v_{\n})=\frac{1}{4\pi}\int_{\mathds{S}^2} H\Big(v_{\n}\big(\upsigma,v_l,v_r\big)\Big)d\upsigma,\quad v_{\n}\big(\upsigma,v_l,v_r\big):=\mathtt{c}_{\mathtt{S}}(v_l,v_r)+\upsigma \uprho_{\mathtt{S}}(v_l,v_r).
\end{equation}
When we fix $(v_l,v_r)$, the integrand in \eqref{eq:second_expansion_sigmac} 
\begin{equation}\label{eq:integrand_function_of_vn_sigmac}
|v_{\n}-\mathbf{c}|\ |v_l-v_r|\mathbf{OV}(v_l,v_r,v_\n)h(s_{\mathcal{P}l},\x_{\mathcal{P}l},v_l)h(s_{\mathcal{P}r},\x_{\mathcal{P}r},v_r)
\end{equation}
can be seen as a function of $v_{\n}$. Specifically according to Remark \ref{rmk:pseudo-trajectory_and_xfunction}, the position functions $\x_{\mathcal{P}l}$ and $\x_{\mathcal{P}r}$ are
\begin{equation*}
\x_{\mathcal{P}l}:=x-(t-s_{\n})v_{\n}-(s_{\n}-s_{\mathcal{P}l})v_{l},\quad \x_{\mathcal{P}r}:=x-(t-s_{\n})v_{\n}-(s_{\n}-s_{\mathcal{P}r})v_{r}. 
\end{equation*}
Therefore we can apply the parametrization \eqref{eq:v_n_asfunction_sigma} to \eqref{eq:integrand_function_of_vn_sigmac} and parametrize the integrand in \eqref{eq:integrand_function_of_vn_sigmac} by $\upsigma$. Define $\Phi_{\upsigma,\mathbf{c}}(v_l,v_r)$ as the following function 
\begin{equation}\label{eq:Phi_sigma_c}
\begin{split}
&\Phi_{\upsigma,\mathbf{c}}(v_{l},v_{r}):=|v_{\n}(\upsigma,v_l,v_r)-\mathbf{c}|\ |v_l-v_r| \mathbf{OV}\Big(v_l,v_r,v_{\n}(\upsigma,v_l,v_r)\Big)
\end{split}
\end{equation}
In this way, the integral in \eqref{eq:second_expansion_sigmac} equals
\begin{equation*}
\frac{1}{4\pi}\int_{\mathds{S}^2}\Bigg[\int_{\plq\times \prq} \Phi_{\upsigma,\mathbf{c}}(v_{l},v_{r})h\Big(s_{\mathcal{P}l},\x_{\mathcal{P}l}(\upsigma,v_l,v_r),v_l\Big)h\Big(s_{\mathcal{P}r},\x_{\mathcal{P}r}(\upsigma,v_l,v_r),v_r\Big)dv_{r}dv_{l}\Bigg]d\upsigma.
\end{equation*}
with
\begin{equation}
\x_{\mathcal{P}l}(\upsigma,v_l,v_r)=x-(t-s_{\n})v_{\n}(\upsigma,v_l,v_r)-(s_{\n}-s_{\mathcal{P}l})v_l,\quad \x_{\mathcal{P}r}(\upsigma,v_l,v_r)=x-(t-s_{\n})v_{\n}(\upsigma,v_l,v_r)-(s_{\n}-s_{\mathcal{P}r})v_r.
\end{equation}

\textit{Step 3: Estimate of perturbations for positions.}

When the velocities $(v_l,v_r)$ are perturbed within the cube pair $\plq\times \prq$, we can estimate the corresponding perturbations of the position $\x_{\mathcal{P}l}(\upsigma,v_l,v_r)$ (and the same for $\x_{\mathcal{P}r}$)
\begin{equation}
\sup_{(v_l,v_r)\in \plq\times \prq}\Big|\x_{\mathcal{P}l}(\upsigma,\pl,\pr)-\x_{\mathcal{P}l}(\upsigma,v_l,v_r)\Big|\leq C\ell^{\frac{1}{2}}\leq \delta.
\end{equation}
Therefore the definition of $\timeldelta$ in \eqref{eq:Ldeltasup} and the assumption of $h\in \timeldelta$ imply
\begin{equation}\label{eq:double_overlap_local_sup}
\begin{split}
&\int_{\plq\times \prq} \Phi_{\upsigma,\mathbf{c}}(v_{l},v_{r})h\Big(s_{\mathcal{P}l},\x_{\mathcal{P}l}(\upsigma,v_l,v_r),v_l\Big)h\Big(s_{\mathcal{P}r},\x_{\mathcal{P}r}(\upsigma,v_l,v_r),v_r\Big)dv_{r}dv_{l}\leq \int_{\plq\times \prq} \Phi_{\upsigma,\mathbf{c}}(v_{l},v_{r})h_l(v_l)h_r(v_r)dv_{r}dv_{l}.
\end{split}
\end{equation}
with 
\begin{equation}
h_l(v_l):=[h(s_{\mathcal{P}l})]_{\delta}(\xp_l,v_l),\quad h_r(v_r):=[h(s_{\mathcal{P}r})]_{\delta}(\xp_r,v_r),\quad \xp_l:=\mathbf{x}_{\mathcal{P}l}(\upsigma,\mathtt{p}_l,\mathtt{p}_r),\ \xp_r:=\mathbf{x}_{\mathcal{P}r}(\upsigma,\mathtt{p}_l,\mathtt{p}_r).
\end{equation}
Next we prove for any fixed $v_r\in\mathds{R}^3$ we have
\begin{equation}\label{eq:double_overlap_left_side}
\int_{\plq}\Phi_{\upsigma,\mathbf{c}}(v_l,v_r)h_l(v_l)dv_l\leq 3^{|\mathcal{O}[\l_l,\n]|}C_{\gamma,\delta}\lVert h\rVert_{\timeldelta}\int_{\kplq}\Phi_{\upsigma,\mathbf{c}}(v_l,v_r)M_{\gamma(0)}(v_l)dv_l,
\end{equation}
and for any fixed $v_l\in\mathds{R}^3$ we have
\begin{equation}\label{eq:double_overlap_right_side}
\int_{\prq}\Phi_{\upsigma,\mathbf{c}}(v_l,v_r)h_r(v_r)dv_r\leq 3^{|\mathcal{O}[\l_r,\n]|}C_{\gamma,\delta}\lVert h\rVert_{\timeldelta}\int_{\kprq}\Phi_{\upsigma,\mathbf{c}}(v_l,v_r)M_{\gamma(0)}(v_r)dv_r.
\end{equation}
Apply the homogenization \eqref{eq:double_overlap_left_side} of $h_l$ and the homogenization of $h_r$ \eqref{eq:double_overlap_right_side} to \eqref{eq:double_overlap_local_sup}
\begin{equation}\label{eq:double_overlap_enlarged_cube}
\begin{split}
&\int_{\plq\times \prq} \Phi_{\upsigma,\mathbf{c}}(v_{l},v_{r})h\Big(s_{\mathcal{P}l},\x_{\mathcal{P}l}(\upsigma,v_l,v_r),v_l\Big)h\Big(s_{\mathcal{P}r},\x_{\mathcal{P}r}(\upsigma,v_l,v_r),v_r\Big)dv_{r}dv_{l}\\
&\leq 10^{|\mathcal{O}[\l_l,\l_r,\n]|}\Big(C_{\gamma,\delta}\lVert h\rVert_{\ldelta}\Big)^2\int_{\kplq\times \kprq} \Phi_{\upsigma,\mathbf{c}}(v_{l},v_{r})M_{\gamma(0)}(v_l)M_{\gamma(0)}(v_r)dv_ldv_r.
\end{split}
\end{equation}
Sum \eqref{eq:double_overlap_enlarged_cube} over $(\mathtt{p}_l,\mathtt{p}_r)\in \ell\mathds{Z}^3\times \ell\mathds{Z}^3$, we conclude the proof of Lemma \ref{lem:double_overlap}.
\begin{figure}[h]
    \centering
    \includegraphics[width=0.7\linewidth]{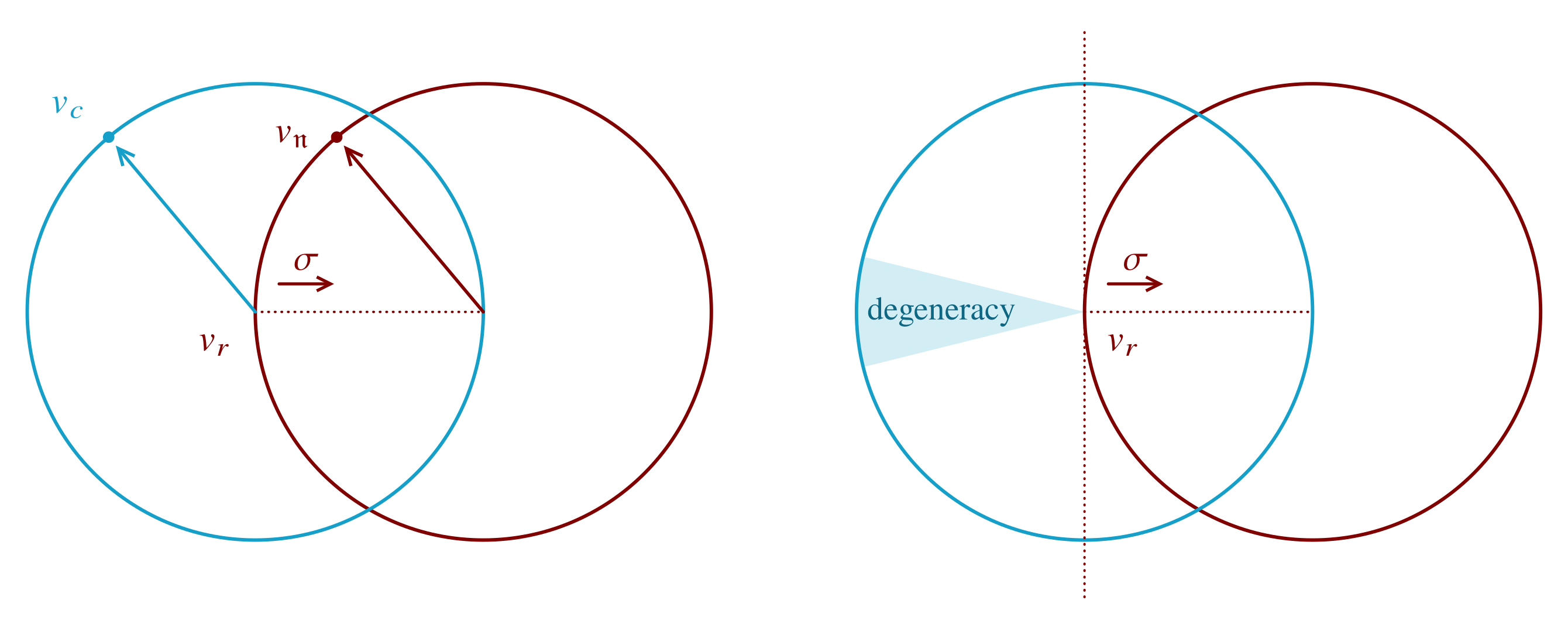}
    \caption{We use the notation $v_c=\frac{v_l+v_r}{2}$.}
    \label{fig:degeneracy}
\end{figure}

\textit{Step 4: Proof of \eqref{eq:double_overlap_left_side} and \eqref{eq:double_overlap_right_side}.}

By the symmetry between variables $v_l$ and $v_r$, we only need to prove \eqref{eq:double_overlap_left_side}. For any fixed $v_r$, we construct a set $B\subset \mathds{R}^3$ such that
\begin{equation}\label{eq:double_overlap_B_condition}
B\subset Q_{\pl}^{\mathtt{K}\ell},\quad \inf_{v_l\in B}\Phi_{\upsigma,\mathbf{c}}(v_l,v_r)\geq C_{\gamma,\delta}3^{-|\mathcal{O}[\l_l,\n]|}\sup_{v_l\in \plq}\Phi_{\upsigma,\mathbf{c}}(v_l,v_r),\quad |B|\geq C_{\gamma,\delta}\ell^3.
\end{equation}
We first analyze some geometric properties of $v_{\n}(\upsigma,\cdot,v_r)$ in \eqref{eq:v_n_asfunction_sigma} as a map to build an intuition. First for any $R\geq 0$
\begin{equation}
v_{\n}(\upsigma,\cdot,v_r):\Big\{v_l| |v_l-v_r|=R\Big\}\rightarrow \Big\{v_\n| |v_\n-(v_r+\upsigma R)|=R\Big\}.
\end{equation}
Then for any $v_l\in\mathds{R}^3$, we always have (see Figure \ref{fig:degeneracy})
\begin{equation}
\Big(v_\n(\upsigma,v_l,\upsigma)-v_r\Big)\cdot\upsigma\geq 0,
\end{equation}
which means the image of $v_{\n}(\upsigma,\cdot,v_l)$ lies at the half space divided by $\sigma^{\perp}+v_r$. Finally, when the angle between $v_l-v_r$ and $-\upsigma$ is small enough (see Figure \ref{fig:degeneracy}), the map is degenerate in the sense of $\nabla_{v_l}v_{\n}(\upsigma,\cdot,v_r)=\frac{1}{2}\textup{Id}+\upsigma \otimes \frac{v_l-v_r}{|v_l-v_r|}$ having eigenvalues and eigenvectors as
\begin{equation}\label{eq:v_n_eigen}
\lambda_1=1,\quad u_1=\frac{v_l-v_r}{|v_l-v_r|}\times \upsigma; \quad \lambda_2=1,\quad u_2=\upsigma -\Big(\upsigma\cdot\frac{v_l-v_r}{|v_l-v_r|}\Big)\frac{v_l-v_r}{|v_l-v_r|};\quad \lambda_3=1+\upsigma\cdot\frac{v_l-v_r}{|v_l-v_r|},\quad u_3=\upsigma.
\end{equation}
To satisfy \eqref{eq:double_overlap_B_condition}, we need to construct $B$ such that
\begin{equation}
\inf_{v_l\in B}|v_{\n}(\upsigma,v_l,v_r)-\mathbf{c}|\geq C_{\gamma}\sup_{v_l\in Q_{\mathtt{p}_l}^{\ell}}|v_{\n}(\upsigma,v_l,v_r)-\mathbf{c}|.
\end{equation}

Denote $\mathtt{PL}_{\upsigma}$ as the plane in $\mathds{R}^3$ containing $v_r$ and spanned by $\mathtt{p}_l-v_r$ and $\upsigma$. If $|\mathbf{c}-v_{\n}(\upsigma,\mathtt{p}_l,v_r)|\geq 10\ell$, then define $B$ as
\begin{equation}
B:=Q_{\mathtt{p}_l}^{\ell}+5\ell\frac{\mathtt{p}_l-v_r}{|\mathtt{p}_l-v_r|}.
\end{equation}
Since we only move $v_l$ at distance $5\ell$, the cross-section $|v_{\n}-\mathbf{c}|$ stay proportional before and after the perturbation of $v_l$.

If $|\mathbf{c}-v_{\n}(\upsigma,\mathtt{p}_l,v_r)|< 10\ell$, then we first move $Q_{\mathtt{p}_l}^{\ell}$ in the direction orthogonal to $\mathtt{PL}$ with distance $10\ell$ to $Q_{\mathtt{p}_l'}^{\ell}$, and then construct $B$ as (see Figure \ref{fig:double_overlap})
\begin{equation}
B:=Q_{\mathtt{p}_l'}^{\ell}+5\ell\frac{\mathtt{p}_l'-v_r}{|\mathtt{p}_l'-v_r|}.
\end{equation}
We can verify the set $B$ satisfies \eqref{eq:double_overlap_B_condition}. Then in the same way as \eqref{eq:absorb_by_gaussian_weight}, we prove \eqref{eq:double_overlap_left_side} and \eqref{eq:double_overlap_right_side}.
\end{proof}

\begin{figure}[h]
    \centering
    \includegraphics[width=1\linewidth]{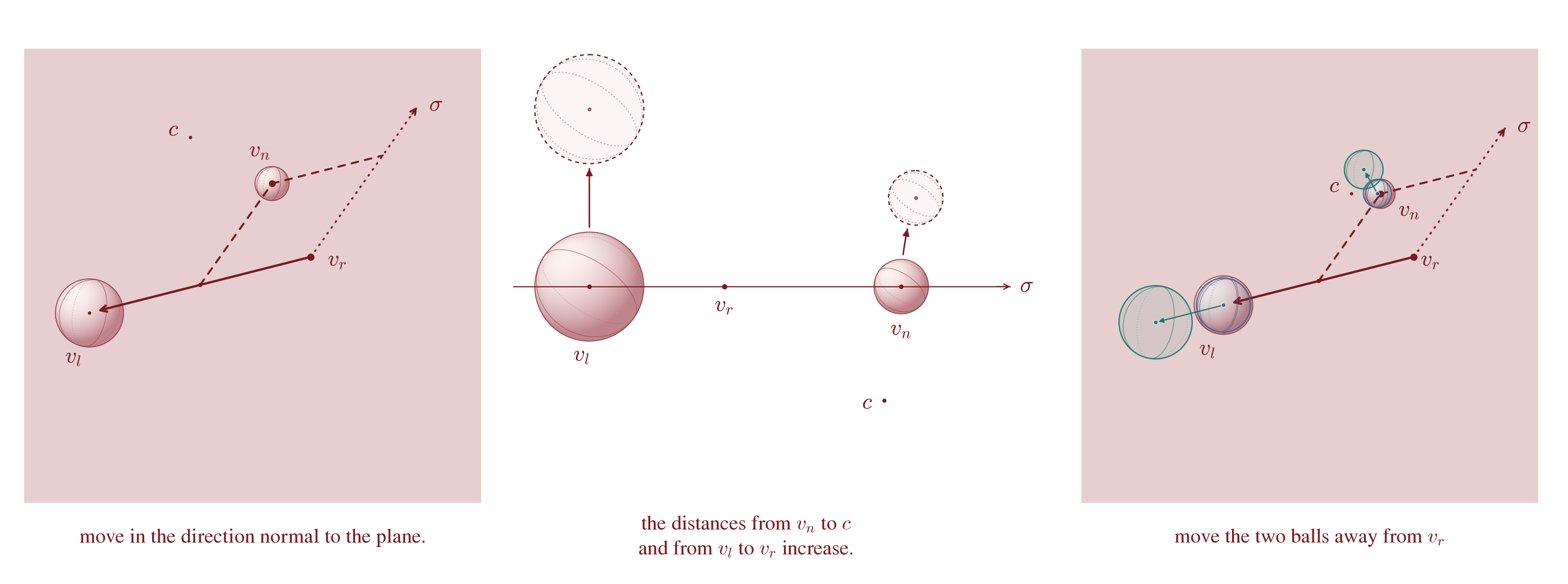}
    \caption{Construction of $B$ in the proof of Lemma \ref{lem:double_overlap}.}
    \label{fig:double_overlap_B_set}
\end{figure}

\begin{lemma}\label{lem:double_null}
Consider any constants $(\gamma>0,\delta>0,T>0)$, and any non-negative function $h(t,x,v)$ within the space $L^{\infty}\ldelta\big([0,T]\times \mathds{R}^3\times \mathds{R}^3\big)$. There exist constants $0<\gamma_*<\gamma(0)<\gamma$ and $\delta_*>0$ such that the following statement is true. Take any decorated tree $\T$, any set $\mathcal{L}_M\subset \mathcal{L}(\T)$ of Maxwellian leaves. Suppose a collision node $\n\in\mathcal{B}^+(\T)$ satisfies condition \ref{cond:DUO} and has no collision-parent node or overlap-parent node. Then 
\begin{equation}\label{eq:double_null_norm}
\begin{split}
&\Big\lVert \mathbf{C}_{h}\Big[\T,\mathcal{L}_M=M_{\gamma(0)}\Big](t,\cdot,\cdot)\Big\rVert_{\ldeltastar}\leq \underbrace{10^{|\mathcal{O}[\l_l,\l_r,\n]|}\Big(C_{\gamma,\delta}\lVert h\rVert_{\timeldelta}\Big)^2}_{\textup{\blue{error from overlaps and initial data}}}\Big\lVert \mathbf{C}_{M_{\gamma(0)}}[\T](t,\cdot,\cdot)\Big\rVert_{\ldeltastar}.
\end{split}
\end{equation}
\end{lemma}

\begin{proof}

For convenience, in the proof we use $\mathbf{C}_h^{\mathcal{T}}$ to denote $\mathbf{C}_h[\mathcal{T},\mathcal{L}_M=M_{\gamma(0)}]$. First use Lemma \ref{lem:full_to_prime} to expand $\mathbf{C}_h^{\mathcal{T}}$ and filter the effects of overlaps
\begin{equation}\label{eq:double_null_norm_pointwise}
\begin{split}
&\int \mathbf{OV}(v_l,v_r,v_{\n})h(s_{\mathcal{P}l},\x_{\mathcal{P}l},v_{l}) h(s_{\mathcal{P}r},\x_{\mathcal{P}r},v_{r})|v_l-v_{\n}|^{-1}d\pi_{v_\n+(v_{l}-v_\n)^{\perp}}(v_{r})dv_{l},
\end{split}
\end{equation}
with the position function $\x$ defined in Remark \ref{rmk:pseudo-trajectory_and_xfunction}. Therefore when we shift the initial position $x$ to $x'$ with $|x'-x|\leq \delta_*$, the only change in \eqref{eq:double_null_norm_pointwise} is the position function $\x$. Consequently
\begin{equation}
\begin{split}
&\Big[\mathbf{C}_h^{\mathcal{T}}(t)\Big]_{\delta_*}(x,v)\leq \int \mathbf{OV}(v_l,v_r,v_\n)[h(s_{\mathcal{P}l})]_{2\delta_*}(\x_{\mathcal{P}l},v_{l}) [h(s_{\mathcal{P}r})]_{2\delta_*}(\x_{\mathcal{P}r},v_{r})|v_l-v_{\n}|^{-1}d\pi_{v_\n+(v_{l}-v_\n)^{\perp}}(v_{r})dv_{l},
\end{split}
\end{equation}
Now we can calculate the norm $\lVert [\mathbf{C}_h^{\mathcal{T}}(t)]_{\delta_*}(x)\rVert_{L_{\gamma_*}^1}$
\begin{equation*}
\begin{split}
\lVert [\mathbf{C}_h^{\mathcal{T}}(t)]_{\delta_*}(x)\rVert_{L_{\gamma_*}^1}\leq &\int e^{\gamma_*|v_{\n}|^2}\mathbf{OV}(v_l,v_r,v_\n)[h(s_{\mathcal{P}l})]_{2\delta_*}(\x_{\mathcal{P}l},v_{l}) [h(s_{\mathcal{P}r})]_{2\delta_*}(\x_{\mathcal{P}r},v_{r})|v_l-v_{\n}|^{-1}d\pi_{v_\n+(v_{l}-v_\n)^{\perp}}(v_{r})dv_{l}dv_{\n},
\end{split}
\end{equation*}
First perform the change of variables to $\{v_{r},v_{l},v_{\n}\}$ using Lemma \ref{lem:geometric_change_of_variable_2}
\begin{equation*}
|v_l-v_{\n}|^{-1}d\pi_{v_\n+(v_{l}-v_\n)^{\perp}}(v_{r})dv_{l}dv_{\n}=\pi|v_l-v_r|\overline{d\pi}_{\mathtt{S}}(v_\n)dv_{l}dv_{r}
\end{equation*}
Now the goal is to homogenize the following integral
\begin{equation}\label{eq:second_expansion_sigma}
\begin{split}
\int &e^{\gamma_*|v_{\n}|^2}\mathbf{OV}(v_l,v_r,v_\n)[h(s_{\mathcal{P}l})]_{2\delta_*}(\x_{\mathcal{P}l},v_{l}) [h(s_{\mathcal{P}r})]_{2\delta_*}(\x_{\mathcal{P}r},v_{r})|v_{l}-v_{r}|\ \overline{d\pi}_{\mathtt{S}}(v_{\n})dv_{r}dv_{l}.
\end{split}
\end{equation}

When we fix $(v_l,v_r)$, the integrand in \eqref{eq:second_expansion_sigma} 
\begin{equation}\label{eq:integrand_function_of_vn_sigma}
e^{\gamma_*|v_{\n}|^2}|v_l-v_r|\mathbf{OV}(v_l,v_r,v_\n)[h(s_{\mathcal{P}l})]_{2\delta_*}(\x_{\mathcal{P}l},v_{l}) [h(s_{\mathcal{P}r})]_{2\delta_*}(\x_{\mathcal{P}r},v_{r})
\end{equation}
can be seen as a function of $v_{\n}$. Therefore we can apply \eqref{eq:v_n_asfunction_sigma} to \eqref{eq:integrand_function_of_vn_sigma} and parametrize the integrand in \eqref{eq:integrand_function_of_vn_sigma} by $\upsigma$. Define $\Phi_{\upsigma,\gamma_*}(v_l,v_r)$ as the following function 
\begin{equation}\label{eq:Phi_sigma_gamma}
\begin{split}
&\Phi_{\upsigma,\gamma_*}(v_{l},v_{r}):=\exp\Big(\gamma_*|v_{\n}(\upsigma,v_l,v_r)|^2\Big)|v_l-v_r| \mathbf{OV}\Big(v_l,v_r,v_{\n}(\upsigma,v_l,v_r)\Big)
\end{split}
\end{equation}
In this way, parametrize and  decompose the integral in \eqref{eq:second_expansion_sigma} into integrals over cubes $Q_{\mathtt{p}_l}\times Q_{\mathtt{p}_r}$
\begin{equation*}
\frac{1}{4\pi}\int_{\mathds{S}^2}\Bigg[\int_{\plq\times \prq} \Phi_{\upsigma,\gamma_*}(v_{l},v_{r})[h(s_{\mathcal{P}l})]_{2\delta_*}\Big(\x_{\mathcal{P}l}(\upsigma,v_l,v_r),v_l\Big)[h(s_{\mathcal{P}r})]_{2\delta_*}\Big(\x_{\mathcal{P}r}(\upsigma,v_l,v_r),v_r\Big)dv_{r}dv_{l}\Bigg]d\upsigma.
\end{equation*}
with the position functions $\x_{\mathcal{P}l}$ and $\mathbf{x}_{\mathcal{P}r}$ determined by $v_l$, $v_r$, and $v_{\n}(\upsigma,v_l,v_r)$. When the velocities $(v_l,v_r)$ are perturbed within the cube pair $\plq\times \prq$, the perturbations of position functions are also small enough. Therefore the definition of $\ldelta$ in \eqref{eq:Ldeltasup} and the assumption of $h\in \timeldelta$ imply
\begin{equation}
\begin{split}
&\int_{\plq\times \prq} \Phi_{\upsigma,\gamma_*}(v_{l},v_{r})[h(s_{\mathcal{P}l})]_{2\delta_*}\Big(\x_{\mathcal{P}l}(\upsigma,v_l,v_r),v_l\Big)[h(s_{\mathcal{P}r})]_{2\delta_*}\Big(\x_{\mathcal{P}r}(\upsigma,v_l,v_r),v_r\Big)dv_{r}dv_{l}\\
&\leq \int_{\plq\times \prq} \Phi_{\upsigma,\gamma_*}(v_{l},v_{r})h_l(v_l)h_r(v_r)dv_{r}dv_{l}
\end{split}
\end{equation}
with
\begin{equation}
h_l(v_l):=[h(s_{\mathcal{P}l})]_{\delta}(\xp_l,v_l),\quad h_r(v_r):=[h(s_{\mathcal{P}r})]_{\delta}(\xp_r,v_r),\quad \xp_l:=\mathbf{x}_{\mathcal{P}l}(\upsigma,\mathtt{p}_l,\mathtt{p}_r),\ \xp_r:=\mathbf{x}_{\mathcal{P}r}(\upsigma,\mathtt{p}_l,\mathtt{p}_r).
\end{equation}
For any fixed $v_r\in\mathds{R}^3$
\begin{equation}
\int_{\plq}\Phi_{\upsigma,\gamma_*}(v_l,v_r)h_l(v_l)dv_l\leq \sup_{v_l\in \plq}\Phi_{\upsigma,\gamma_*}\Big(\int_{\plq}h_l(v_l)dv_l\Big)\leq \Big(\sup_{v_l\in \plq}\Phi_{\upsigma,\gamma_*}(v_l,v_r)\Big)\lVert h\rVert_{\timeldelta}\sup_{v_l\in\plq}e^{-\gamma|v_l|^2}.
\end{equation}
Next we prove for any fixed $v_r\in\mathds{R}^3$ we have
\begin{equation}\label{eq:double_null_left_side}
\int_{\plq}\Phi_{\upsigma,\gamma_*}(v_l,v_r)h_l(v_l)dv_l\leq 3^{|\mathcal{O}[\l_l,\n]|}C_{\gamma,\delta}\lVert h\rVert_{\timeldelta}\int_{\kplq}\Phi_{\upsigma,\gamma_*}(v_l,v_r)M_{\gamma(0)}(v_l)dv_l,
\end{equation}
and for any fixed $v_l\in\mathds{R}^3$ we have
\begin{equation}\label{eq:double_null_right_side}
\int_{\prq}\Phi_{\upsigma,\gamma_*}(v_l,v_r)h_r(v_r)dv_r\leq C_{\gamma,\delta}3^{|\mathcal{O}[\l_r,\n]|}C_{\gamma,\delta}\lVert h\rVert_{\timeldelta}\int_{\kprq}\Phi_{\upsigma,\gamma_*}(v_l,v_r)M_{\gamma(0)}(v_r)dv_r.
\end{equation}
By symmetry it is sufficient to prove \eqref{eq:double_null_left_side}. We want to construct a set $B\subset Q_{\mathtt{p}_l}^{\mathtt{K}\ell}$ such that (notice $C_{\gamma}$ is a constant determined by $\gamma$, not $\gamma_*$)
\begin{equation}\label{eq:double_null_B_condition}
|B|\geq \ell^{3}\Delta^{-3},\quad \Delta:=\max\{1,|\mathtt{p}_l|,|v_r|\}; \inf_{v_l\in B}\Phi_{\upsigma,\gamma_*}(v_l,v_r)\geq 3^{-|\mathcal{O}[\l_l,\n]|}\exp\Big(-C_{\gamma}\gamma_*|v_l|^2\Big).
\end{equation}
Then in the same way as \eqref{eq:absorb_by_gaussian_weight}, we can prove \eqref{eq:double_null_left_side} and \eqref{eq:double_null_right_side}, and therefore conclude the proof of Lemma \ref{lem:double_null}.

Denote $\mathtt{p}_{l}^*$ as the point in $Q_{\mathtt{p}_l}^{\ell}$ such that 
\begin{equation}
\Phi_{\upsigma,\gamma_*}(v_l^*,v_r)=\max_{v_l\in Q_{\mathtt{p}_l}^{\ell}}\Phi_{\upsigma,\gamma_*}(v_l,v_r)
\end{equation}

It is convenient to assume $v_r\not=\mathtt{p}_l$. If $|\mathtt{p}_l|\geq \frac{1}{10}|v_r|$ or $|v_r-\mathtt{p}_l|\leq 10$, then we construct $B$ as the following. 
\begin{equation}
B:=Q_{\mathtt{p}_l^*}^{\ell \Delta^{-1}}+3\ell\frac{\mathtt{p}_*-v_r}{|\mathtt{p}_l^*-v_r|}.
\end{equation}

If $|\mathtt{p}_l|<\frac{1}{10}|v_r|$ and $|\mathtt{p}_l-v_r|>10$, we write $A_{\upsigma,v_r}$ as the inverse map of $v_{\n}(\upsigma,\cdot,v_r)$. Denote $B_{\textup{pre}}$ as 
\begin{equation}
B_{\textup{pre}}:=\Big\{v\in\mathds{R}^3\Big|\ (v-v_r)\cdot\upsigma>0\Big\}\cap \Big\{v\in\mathds{R}^3\Big|\ |v|\geq |v_{\n}(\upsigma,v_l^*,v_r)|\Big\}.
\end{equation}
Set $B$ is defined as $A_{\upsigma,v_r}(B_{\textup{pre}})\cap Q_{\mathtt{p}_l^*}^{\mathtt{K}\ell}$. This set $B$ can be verified as satisfying \eqref{eq:double_null_B_condition}.

\end{proof}

\subsection{Single Homogenization}

\begin{lemma}\label{lem:single_collsion}
Consider any constants $(\gamma>0,\delta>0,T>0)$, and any non-negative function $h(t,x,v)$ within the space $L^{\infty}\ldelta\big([0,T]\times \mathds{R}^3\times \mathds{R}^3\big)$. There exists a constant $0<\gamma(0)<\gamma$ such that the following statement is true. Consider any decorated tree $\T$, any set $\mathcal{L}_M$ of Maxwellian leaves. Assume a collision node $\n\in \mathcal{B}^+(\T)$ satisfies condition \ref{cond:SIN} and has a collision-parent node $\m$. Then for any $(t,x,v)\in [0,T]\times \mathds{R}^3\times \mathds{R}^3$
\begin{equation}\label{eq:single_collision}
\begin{split}
&\mathbf{C}_{h}\Big[\T,\mathcal{L}_M=M_{\gamma(0)}\Big](t,x,v)\leq 10^{|\mathcal{O}[\l,\n]|}\Big(C_{\gamma,\delta}\lVert h\rVert_{\timeldelta}\Big)\mathbf{C}_{h}\Big[\T,\mathcal{L}_M\cup\{\l\}=M_{\gamma(0)}\Big](t,x,v)
\end{split}
\end{equation}
\end{lemma}

\begin{figure}[h]
    \centering
    \includegraphics[width=0.85\linewidth]{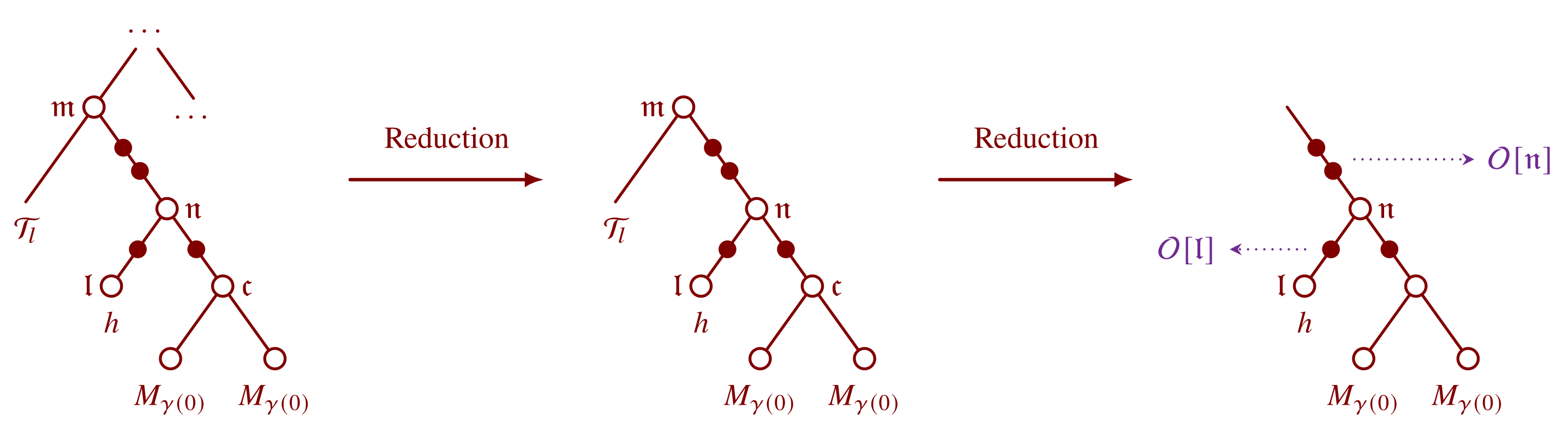}
    \caption{The estimate of the entire decorated tree can be reduced to the estimate of a sub-tree relevant to the node $\n$.}
    \label{fig:single_collision}
\end{figure}

\begin{proof}

\textit{Step 1: Integrate the Gaussian branch.}

In the proof of Lemma \ref{lem:double_collision} on the double homogenization, using the inductive definition of $\mathbf{C}_{h}[\mathcal{T},\mathcal{L}_M=M_{\gamma(0)}]$ in Definition \ref{def:C_expansion_replacement}, we have shown the homogenization of $\mathbf{C}_{h}[\mathcal{T},\mathcal{L}_M=M_{\gamma(0)}]$ can be reduced to the homogenization of $\mathbf{C}_{h}[\mathcal{T}_*,\mathcal{L}_M=M_{\gamma(0)}]$, where $\mathcal{T}_*$ is a sub-tree of $\mathcal{T}$. 

Precisely, denote $\mathfrak{m}$ as the collision-parent of $\mathfrak{n}$, and the left sub-tree of $\mathfrak{m}$ as $\mathcal{T}_l$ and the right sub-tree of $\mathfrak{m}$ as $\mathcal{T}_r$. We denote the sub-tree among $\{\mathcal{T}_l,\mathcal{T}_r\}$ containing $\mathfrak{n}$ as $\mathcal{T}_*$. To prove the estimate \eqref{eq:single_collision} of single homogenization, it suffices to prove for any $(t,x)\in [0,T]\times \mathds{R}^3$ and any plane $\mathtt{PL}\subset \mathds{R}^3$
\begin{equation}\label{eq:PL_integration_single_collision}
\int \mathbf{C}_{h}\Big[\T_*,\blue{\mathcal{L}_{M}=M_{\gamma(0)}}\Big](t,x,v)d\pi_{\mathtt{PL}}(v)\leq \int 10^{|\mathcal{O}[\mathfrak{l},\mathfrak{n}]|}\Big(C_{\gamma,\delta}\lVert h\rVert_{\timeldelta}\Big)\mathbf{C}_{h}\Big[\T_*,\red{\mathcal{L}_{M}\cup\{\mathfrak{l}\}=M_{\gamma(0)}}\Big](t,x,v)d\pi_{\mathtt{PL}}(v)
\end{equation}
Using Lemma \ref{lem:full_to_prime}, we expand the LHS of \eqref{eq:PL_integration_single_collision} as the following, with $\x_{\mathcal{P}\l}$ as the position function defined in Remark \ref{rmk:pseudo-trajectory_and_xfunction} and $\mathcal{T}_{\mathfrak{c},*}$ as the sub-tree of $\n$ containing $\mathfrak{c}$
\begin{equation}\label{eq:first_expansion_single_collision}
\begin{split}
&\int \Bigg(\prod_{\o\in\mathcal{O}[\l]}\mathbf{O}_{M_{\gamma(0)}}[\mathscr{T}_{\o}](s_{\o},v_{\mathfrak{l}})\prod_{\o\in\mathcal{O}[\n]}\mathbf{O}_{M_{\gamma(0)}}[\mathscr{T}_{\o}](s_{\o},v_{\n})\Bigg)\\
&\hspace{5cm}h(s_{\mathcal{P}\l},\x_{\mathcal{P}\l},v_{\mathfrak{l}})\underbrace{\Bigg(\int \mathbf{C}_{M_{\gamma(0)}}[\mathcal{T}_{\mathfrak{c},*}](s_{\n},v_{\mathfrak{c}})d\pi_{v_{\n}+(v_{\l}-v_{\n})^{\perp}}(v_{\mathfrak{c}})\Bigg)}_{\textup{\blue{we need to analyze its regularity}}}|v_{\mathfrak{l}}-v_{\n}|^{-1}dv_{\mathfrak{l}}d\pi_{\mathtt{PL}}(v_\n)
\end{split}
\end{equation}
Since the sub-tree $\mathcal{T}_{\mathfrak{c},*}$ has only Maxwellian leaves, its Boltzmann integral also only involves Maxwellians as in \eqref{eq:first_expansion_single_collision}. This allows us to integrate over $v_{\mathfrak{c}}$ and analyze the regularity of this integral with respect to $v_\n$ and $v_\l$.

To this end, using Lemma \ref{lem:non_decreasing_form}: there exist function $c(t)$, non-negative radial and non-decreasing-in-$|v|$ function $\Lambda(v)$ such that
\begin{equation}
\mathbf{C}_{M_{\gamma(0)}}[\mathcal{T}_{\mathfrak{c},*}](s_{\n},v)=c(s_{\n})e^{-\gamma(0)|v|^2}\Lambda(v)
\end{equation}
Consequently
\begin{equation}\label{eq:reduction_of_TC_integral}
\int \mathbf{C}_{M_{\gamma(0)}}[\mathcal{T}_{\mathfrak{c},*}](s_{\n},v)d\pi_{v_{\n}+(v_{\l}-v_{\n})^{\perp}}(v)=c(s_{\n})\int e^{-\gamma(0)|v|^2}\Lambda(v) d\pi_{v_{\n}+(v_{\l}-v_{\n})^{\perp}}(v)
\end{equation}
Given any $v_{\n}$ and $v_{\l}$, define $\mathtt{d}(v_{\n},v_{\l})$ as the distance from $0$ to the orthogonal plane $v_{\n}+(v_{\l}-v_{\n})^{\perp}$
\begin{equation}\label{eq:orthogonal_distance}
\mathtt{d}(v_{\n},v_{\l}):=\textup{dist}\Big(0,v_{\n}+(v_{\l}-v_{\n})^{\perp}\Big).
\end{equation}
We use the radial symmetry of $\Lambda(v)$ in the RHS of \eqref{eq:reduction_of_TC_integral} to simplify the expression of the RHS integral, together with some changes of variables 
\begin{equation}
\begin{split}
&\int \mathbf{C}_{M_{\gamma(0)}}[\mathcal{T}_{\mathfrak{c},*}](s_{\n},v)d\pi_{v_{\n}+(v_{\l}-v_{\n})^{\perp}}(v)=c(s_{\n})\exp\Big(-\gamma(0)|\mathtt{d}(v_{\n},v_{\l})|^2\Big)\int_0^{+\infty}e^{-\gamma(0)|r|^2}\Lambda\big(\sqrt{\mathtt{d}^2(v_\n,v_\l)+r^2}\big)2\pi rdr \\
&\textup{\blue{\small (change of variable from $r$ to $u=r^2$)}}\quad =C\ c(s_{\n})\exp\Big(-\gamma(0)|\mathtt{d}(v_{\n},v_{\l})|^2\Big)\int_0^{+\infty}e^{-\gamma(0)u}\Lambda\big(\sqrt{\mathtt{d}^2(v_\n,v_\l)+u}\big) du\\
&\textup{\blue{\small (change of variable $u\mapsto \sqrt{\mathtt{d}^2(v_{\n},v_{\l})+u}$ )}}=C\ c(s_{\n})\int_{\mathtt{d}(v_{\n},v_{\l})}^{+\infty}e^{-\gamma(0)u^2}\Lambda\big(u\big) udu.
\end{split}
\end{equation}
\begin{figure}[h]
    \centering
    \includegraphics[width=0.6\linewidth]{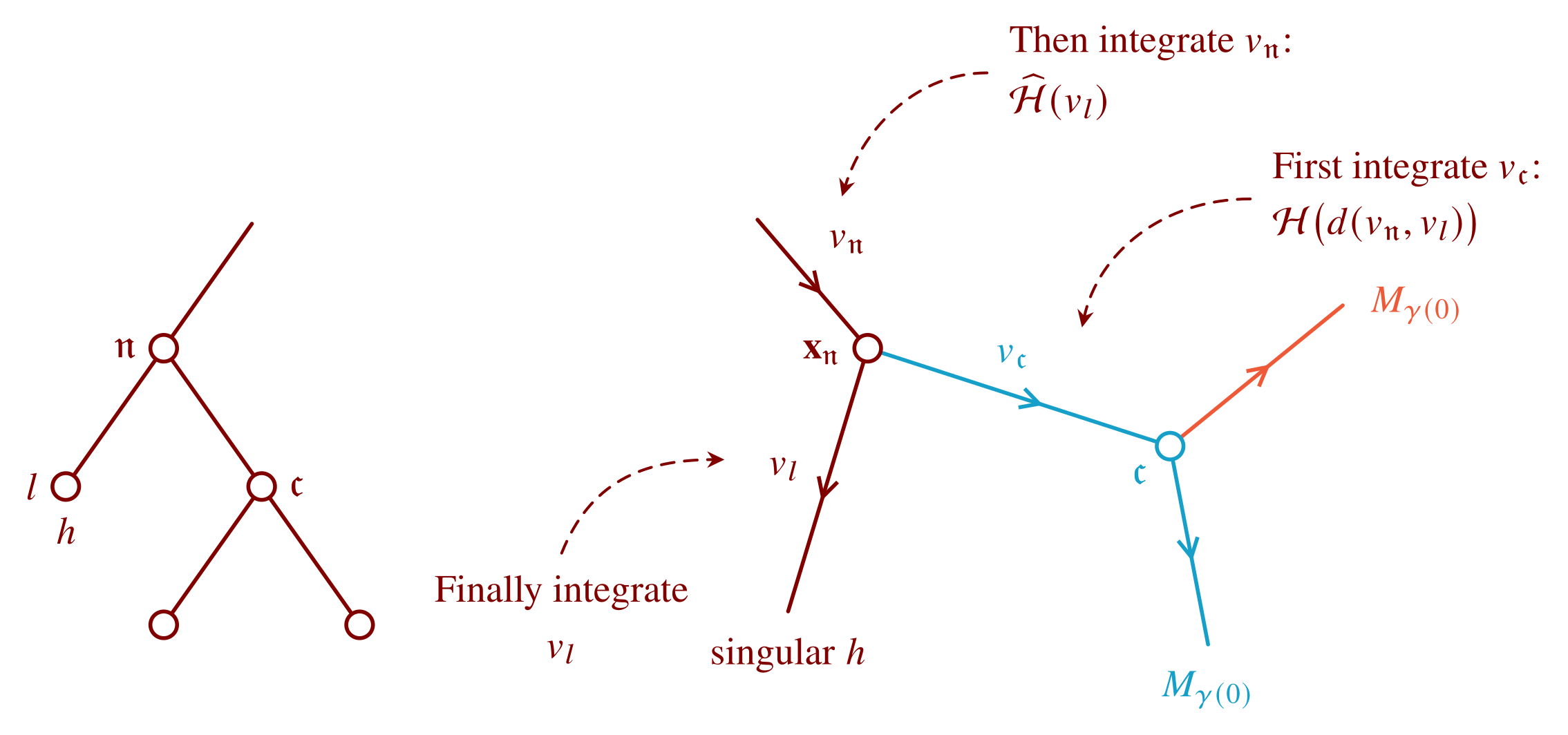}
    \caption{Caption}
    \label{fig:single_col_homo_integration}
\end{figure}
Denote $\mathcal{H}(R)$ as
\begin{equation}
\mathcal{H}(R):=\int_{R}^{+\infty}e^{-\gamma(0)u^2}\Lambda\big(u\big) udu.
\end{equation}
Since $\Lambda$ is non-decreasing as stated in Lemma \ref{lem:non_decreasing_form}, we can 
\begin{equation}\label{eq:H_regularity}
\mathcal{H}(R_*)\geq \exp\Big(-\gamma(0)\big(R_*^2-R^2\big)_+\Big)\mathcal{H}(R).
\end{equation}
To prove Lemma \ref{lem:single_collsion}, it is sufficient to prove
\begin{equation}\label{eq:single_collision_reduced_fullspace}
\begin{split}
&\int \mathbf{OV}(v_{\n},v_{\l})h(s_{\mathcal{P}l},\x_{\mathcal{P}l},v_{\mathfrak{l}})\mathcal{H}\Big(\mathtt{d}(v_{\n},v_{\l})\Big)|v_{\mathfrak{l}}-v_{\n}|^{-1}dv_\l d\pi_{\mathtt{PL}}(v_\n)\\
&\leq 10^{|\mathcal{O}[\l,\n]|}\Big(C_{\gamma,\delta}\lVert h\rVert_{\timeldelta}\Big)\int \mathbf{OV}(v_{\n},v_{\l})M_{\gamma(0)}(v_{\l})\mathcal{H}\Big(\mathtt{d}(v_{\n},v_{\l})\Big)|v_{\mathfrak{l}}-v_{\n}|^{-1}dv_\l d\pi_{\mathtt{PL}}(v_\n),
\end{split}
\end{equation}
with $\mathbf{OV}(v_{\n},v_{\l})$ defined as the abbreviation of the multiplication of overlap functions
\begin{equation}
\mathbf{OV}(v_{\n},v_{\l}):=\prod_{\o\in\mathcal{O}[\l]}\mathbf{O}_{M_{\gamma(0)}}[\mathscr{T}_{\o}](s_{\o},v_{\mathfrak{l}})\prod_{\o\in\mathcal{O}[\n]}\mathbf{O}_{M_{\gamma(0)}}[\mathscr{T}_{\o}](s_{\o},v_{\n}).
\end{equation}
According to Lemma \ref{lem:overlap_regularity}, for any $\boldsymbol{\eta}$ with $|\boldsymbol{\eta}|$ small enough determined by $(\gamma,\delta)$, we have
\begin{equation}\label{eq:overlap_bundle_regularity_single_collision}
\mathbf{OV}(v_{\n}+\boldsymbol{\eta},v_{\l})\leq 3^{|\mathcal{O}[\n]|}\mathbf{OV}(v_{\n},v_{\l}),\quad \mathbf{OV}(v_{\n},v_{\l}+\boldsymbol{\eta})\leq 3^{|\mathcal{O}[\l]|}\mathbf{OV}(v_{\n},v_{\l}).
\end{equation}

\textit{Step 2: Decompose the integration domain.}

Given plane $\mathtt{PL}$, denote $\mathtt{O}$ as the point on $\mathtt{PL}$ being the closest to $0$. Choose to orthogonal unit vectors $\mathtt{e}_1$ and $\mathtt{e}_2$ spanning $\mathtt{PL}$. Given any $\ell>0$, define $\ell\mathtt{Z^2}$ as the scaled lattice on $\mathtt{PL}$
\begin{equation}
\ell\mathtt{Z^2}:=\Big\{\mathtt{O}+k_1\ell \mathtt{e}_1+k_2\ell\mathtt{e}_2\Big|\ k_1,k_2\in\mathds{Z}\Big\}.
\end{equation}
Given any point $\mathtt{p}\in \mathtt{PL}$, define $\mathtt{Q}_{\mathtt{p}}^{\ell}$ as
\begin{equation}
\mathtt{Q}_{\mathtt{p}}^{\ell}:=\Big\{\mathtt{p}+a_1\ell \mathtt{e}_1+a_2\ell\mathtt{e}_2\Big|\ a_1,a_2\in\mathds{R},\quad  |a_1|\leq \frac{1}{2},\ |a_2|\leq\frac{1}{2}\Big\}.
\end{equation}
 
As we integrate over $v_\l$ and $v_\n$, the other functions in \eqref{eq:single_collision_reduced_fullspace} will be perturbed accordingly, therefore we decompose the domain $\mathtt{PL}\times \mathds{R}^3$ into the union of small pieces
\begin{equation}
\mathtt{PL}=\bigcup_{\mathtt{p}_{\n}\in \ell \mathtt{Z^2}}\mathtt{Q}_{\mathtt{p}_\n}^{\ell},\quad \mathds{R}^3=\bigcup_{\mathtt{p}_{\l}\in \ell\mathds{Z}^3}\plq.
\end{equation}
Therefore to prove Lemma \ref{lem:single_collsion}, it is sufficient to prove for any $\mathtt{p}_\n$ and $\mathtt{p}_\l$
\begin{equation}
\begin{split}
&\int_{\mathtt{Q}_{\mathtt{p}_\n}^{\ell}\times \plq} \mathbf{OV}(v_{\n},v_{\l})h(s_{\mathcal{P}\l},\x_{\mathcal{P}\l},v_{\mathfrak{l}})\mathcal{H}\Big(\mathtt{d}(v_{\n},v_{\l})\Big)|v_{\mathfrak{l}}-v_{\n}|^{-1}dv_{\l}d\pi_{\mathtt{PL}}(v_\n)\\
&\leq 10^{|\mathcal{O}[\l,\n]|}\Big(C_{\gamma,\delta}\lVert h\rVert_{\timeldelta}\Big)\int_{\mathtt{Q_{\mathtt{p}_\n}^\ell}\times Q_{\mathtt{p}_{\l}}^{\ell}} \mathbf{OV}(v_{\n},v_{\l})M_{\gamma(0)}(v_{\l})\mathcal{H}\Big(\mathtt{d}(v_{\n},v_{\l})\Big)|v_{\mathfrak{l}}-v_{\n}|^{-1}dv_{\l}d\pi_{\mathtt{PL}}(v_\n)
\end{split}
\end{equation}
Given any $\mathtt{Q_p^\ell}$, define $\widehat{\mathcal{H}}(v_{\l})$ as
\begin{equation}\label{eq:widehat_H}
\widehat{\mathcal{H}}(v_{\l}):=\int_{\mathtt{Q_p^\ell}}\mathbf{OV}(v_\n,v_\l)\mathcal{H}\Big(\mathtt{d}(v_\n,v_\l)\Big)|v_\l-v_\n|^{-1}d\pi_{\mathtt{PL}}(v_\n)
\end{equation}
Therefore to prove Lemma \ref{lem:single_collsion}, it is sufficient to prove for any $\mathtt{p}_\l$
\begin{equation}
\int_{Q_{\mathtt{p}_\l}^\ell}\widehat{\mathcal{H}}(v_\l)[h(s_{\mathcal{P}\l})]_{\delta}(\xp,v_\l)dv_\l\leq 10^{|\mathcal{O}[\l,\n]|}\Big(C_{\gamma,\delta}\lVert h\rVert_{\timeldelta}\Big)\int_{Q_{\mathtt{p}_\l}^\ell}\widehat{\mathcal{H}}(v_\l)M_{\gamma(0)}dv_\l,
\end{equation}
which can be reduced to the following regularity estimate.

\textit{Step 3: Estimate of regularity.}

Now we prove for small shift $\boldsymbol{\eta}\in\mathds{R}^3$ with $|\boldsymbol{\eta}|\leq \ell$, we have
\begin{equation}\label{eq:regularity_H_hat}
\widehat{\mathcal{H}}(v_{\l}+\boldsymbol{\eta})\geq C_{\gamma,\delta}3^{-|\mathcal{O}[\l,\n]|}\exp\Big(-4\gamma(0)\big(|v_\l|+1\big)^2\Big)\widehat{\mathcal{H}}(v_\l).
\end{equation}
Expand the function $\widehat{\mathcal{H}}(v_{\l}+\boldsymbol{\eta})$ according to the definition of $\widehat{\mathcal{H}}$ in \eqref{eq:widehat_H}
\begin{equation}
\begin{split}
\widehat{\mathcal{H}}(v_{\l}+\boldsymbol{\eta})&=\int_{\mathtt{Q_p^\ell}}\mathbf{OV}(v_\n,v_\l+\boldsymbol{\eta})\mathcal{H}\Big(\mathtt{d}(v_\n,v_\l+\boldsymbol{\eta})\Big)|v_\l+\boldsymbol{\eta}-v_\n|^{-1}d\pi_{\mathtt{PL}}(v_\n)
\end{split}
\end{equation}
Notice the regularity of $\mathbf{OV}$ in \eqref{eq:overlap_bundle_regularity_single_collision} implies
\begin{equation}
\inf_{v_\n\in \mathtt{Q_p}^{\ell}}\mathbf{OV}(v_{\n},v_{\l}+\boldsymbol{\eta})\geq 3^{-|\mathcal{O}[\l,\n]|}\sup_{v_\n\in \mathtt{Q_p}^{\ell}}\mathbf{OV}(v_{\n},v_{\l}).
\end{equation}
Therefore to prove \eqref{eq:regularity_H_hat}, it is enough to prove
\begin{equation}\label{eq:regularity_of_Hhat_reduced}
\int_{\mathtt{Q}_{\mathtt{p}}^{\ell}}\mathcal{H}\Big(\mathtt{d}(v_\n,v_\l+\boldsymbol{\eta})\Big)|v_\l+\boldsymbol{\eta}-v_\n|^{-1}d\pi_{\mathtt{PL}}(v_\n)\geq C_{\gamma,\delta}\exp\Big(-4\gamma(0)(|v_\l|+1)^2\Big)\int_{\mathtt{Q}_{\mathtt{p}}^{\ell}}\mathcal{H}\Big(\mathtt{d}(v_\n,v_\l)\Big)|v_\l-v_\n|^{-1}d\pi_{\mathtt{PL}}(v_\n)
\end{equation}
The definition of $\mathtt{d}(v_{\n},v_{\l})$ in \eqref{eq:orthogonal_distance} implies
\begin{equation}\label{eq:distance_geometric_inequality}
\mathtt{d}^2(v_\n,v_\l+\boldsymbol{\eta})\leq |v_{\n}|^2+|v_\l+\boldsymbol{\eta}|^2\leq \mathtt{d}^2(v_{\n},v_{\l})+|v_\l|^2+|v_\l+\boldsymbol{\eta}|^2\leq \mathtt{d}^2(v_{\n},v_{\l})+2(|v_\l|+1)^2.
\end{equation}
Consequently
\begin{equation}
\mathcal{H}\big(\mathtt{d}(v_\n,v_\l+\boldsymbol{\eta})\big)\geq \exp\Big(-2\gamma(0)(|v_\l|+1)^2\Big)\mathcal{H}\Big(\mathtt{d}(v_\n,v_\l)\Big).
\end{equation}
Thus
\begin{equation}\label{eq:regularity_of_Hhat_reduced_reduce}
\int_{\mathtt{Q}_{\mathtt{p}}^{\ell}}\mathcal{H}\Big(\mathtt{d}(v_\n,v_\l+\boldsymbol{\eta})\Big)|v_\l+\boldsymbol{\eta}-v_\n|^{-1}d\pi_{\mathtt{PL}}(v_\n)\geq C_{\gamma,\delta}\exp\Big(-2\gamma(0)(|v_\l|+1)^2\Big)\int_{\mathtt{Q}_{\mathtt{p}}^{\ell}}\mathcal{H}\Big(\mathtt{d}(v_\n,v_\l)\Big)|v_\l+\boldsymbol{\eta}-v_\n|^{-1}d\pi_{\mathtt{PL}}(v_\n)
\end{equation}
If $\textup{dist}(v_\l,\mathtt{Q}_{\mathtt{p}}^{\ell})\geq 10\ell$, then (recall $\boldsymbol{\eta}$ is chosen small enough determined by $\gamma$ and $\delta$)
\begin{equation}
|v_{\l}+\boldsymbol{\eta}-v_{\n}|\geq C_{\gamma,\delta}|v_{\l}-v_{\n}|,\quad \textup{for any $v_\n\in Q_{\mathtt{p}_\n}^{\ell}$},
\end{equation}
and we can conclude the proof of \eqref{eq:regularity_H_hat}.

If $\textup{dist}(v_\l,\mathtt{Q}_{\mathtt{p}}^{\ell})<10\ell$, then as the same way of proving \eqref{eq:distance_geometric_inequality}, we can prove
\begin{equation}
\inf_{v_\n\in \mathtt{Q}_{\mathtt{p}}^{\ell}}\mathcal{H}\big(\mathtt{d}(v_\n,v_\l)\big)\geq \exp\Big(-2\gamma(0)\big(|v_\l|+1\big)^2\Big)\sup_{v_\n\in \mathtt{Q}_{\mathtt{p}}^{\ell}}\mathcal{H}\big(\mathtt{d}(v_\n,v_\l)\big).
\end{equation}
Then the singularity of $|v_{\l}+\boldsymbol{\eta}-v_\n|^{-1}$ and $|v_\l-v_\n|^{-1}$ can be resolved using the integration of $v_{\n}$ on the 2-dimensional cube $\mathtt{Q}_{\mathtt{p}}^{\ell}$
\begin{equation}
\begin{split}
\int_{\mathtt{Q}_{\mathtt{p}}^{\ell}}\mathcal{H}\Big(\mathtt{d}(v_\n,v_\l)\Big)|v_\l+\boldsymbol{\eta}-v_\n|^{-1}d\pi_{\mathtt{PL}}(v_\n)&\geq \exp\Big(-2\gamma(0)\big(|v_\l|+1\big)^2\Big)\sup_{v_\n\in \mathtt{Q}_{\mathtt{p}}^{\ell}}\mathcal{H}\big(\mathtt{d}(v_{\n},v_\l)\big)\int_{\mathtt{Q}_{\mathtt{p}}^{\ell}}|v_{\l}+\boldsymbol{\eta}-v_{\n}|^{-1}dv_{\n}\\
&\geq C_{\gamma,\delta}\exp\Big(-2\gamma(0)\big(|v_\l|+1\big)^2\Big)\sup_{v_\n\in \mathtt{Q}_{\mathtt{p}}^{\ell}}\mathcal{H}\big(\mathtt{d}(v_{\n},v_\l)\big)\int_{\mathtt{Q}_{\mathtt{p}}^{\ell}}|v_{\l}-v_{\n}|^{-1}dv_{\n}.
\end{split}
\end{equation}
This concludes the proof of \eqref{eq:regularity_H_hat}.

\end{proof}

Now we list the rest of lemmas needed for homogenization operations. These can be proved by directly modifying the proof of Lemma \ref{lem:double_collision}-\ref{lem:single_collsion}, therefore here we omit the proof.

\begin{lemma}\label{lem:single_overlap}
Consider any constants $(\gamma>0,\delta>0,T>0)$, and any non-negative function $h(t,x,v)$ within the space $L^{\infty}\ldelta\big([0,T]\times \mathds{R}^3\times \mathds{R}^3\big)$. There exists a constant $\gamma(0)>0$ such that the following statement is true. For any decorated tree $\T$, any set $\mathcal{L}_M\subset \mathcal{L}(\T)$ of Maxwellian leaves, assume a collision node $\n\in \mathcal{B}^+(\T)$ satisfies condition \ref{cond:SIN} and has an overlap-parent node $\o_*\in \mathcal{O}(\T)$. Then for any $(t,x,v)\in [0,T]\times \mathds{R}^3\times \mathds{R}^3$
\begin{equation}\label{eq:single_high_overlap}
\begin{split}
&\mathbf{C}_{h}\Big[\T,\mathcal{L}_M=M_{\gamma(0)}\Big](t,x,v)\leq 10^{|\mathcal{O}[\l,\n]|}\Big(C_{\gamma,\delta}\lVert h\rVert_{\timeldelta}\Big)\mathbf{C}_{h}\Big[\T,\mathcal{L}_M\cup\{\l\}=M_{\gamma(0)}\Big](t,x,v)
\end{split}
\end{equation}
\end{lemma}

\begin{lemma}\label{lem:single_null}
Consider any constants $(\gamma>0,\delta>0,T>0)$, and any non-negative function $h(t,x,v)$ within the space $L^{\infty}\ldelta\big([0,T]\times \mathds{R}^3\times \mathds{R}^3\big)$. There exist constants $0<\gamma_*<\gamma(0)<\gamma$ and $\delta_*>0$ such that the following statement is true. For any decorated tree $\T$, any set $\mathcal{L}_M\subset \mathcal{L}(\T)$ of Maxwellian leaves, assume a collision node $\n\in \mathcal{B}^+(\T)$ satisfies condition \ref{cond:SIN} and has no collision-parent node or overlap-parent node. Then we have
\begin{equation}\label{eq:single_high_norm}
\begin{split}
&\Big\lVert \mathbf{C}_{h}\Big[\T,\mathcal{L}_M=M_{\gamma(0)}\Big](t,\cdot,\cdot)\Big\rVert_{\ldeltastar}\leq 10^{|\mathcal{O}[\l,\n]|}\Big(C_{\gamma}\lVert h\rVert_{\timeldelta}\Big)\Big\lVert \mathbf{C}_{h}[\T,\mathcal{L}_M\cup\{\l\}=M_{\gamma(0)}](t,\cdot,\cdot)\Big\rVert_{\ldeltastar}.
\end{split}
\end{equation}
\end{lemma}

\begin{lemma}\label{lem:indi_homo}
Consider any constants $(\gamma>0,\delta>0,T>0)$, and any non-negative function $h(t,x,v)$ within the space $L^{\infty}\ldelta\big([0,T]\times \mathds{R}^3\times \mathds{R}^3\big)$. There exist constants $0<\gamma_*<\gamma(0)<\gamma$ and $\delta_*>0$ such that the following statement is true. For any decorated tree $\T$, any set $\mathcal{L}_M\subset \mathcal{L}(\T)$ of Maxwellian leaves, if a non-Maxwellian leaf $\l\in \mathcal{L}(\T)\setminus\mathcal{L}_M$ satisfies condition \ref{cond:IND} and has an overlap-parent node $\o_*$, then
\begin{equation}\label{eq:indi_homo_overlap}
\mathbf{C}_{h}\Big[\T,\mathcal{L}_M=M_{\gamma(0)}\Big](t,x,v)\leq 10^{|\mathcal{O}[\l]|}\Big(C_{\gamma,\delta}\lVert h\rVert_{\timeldelta}\Big)\mathbf{C}_{h}\Big[\T,\mathcal{L}_M\cup\{\l\}=M_{\gamma(0)}\Big](t,x,v).
\end{equation}
If a non-Maxwellian leaf $\l\in \mathcal{L}(\T)\setminus \mathcal{L}_M$ has no collision-parent node or overlap-parent node, then
\begin{equation}\label{eq:indi_homo_norm}
\Big\lVert \mathbf{C}_{h}\Big[\T,\mathcal{L}_M=M_{\gamma(0)}\Big](t,\cdot,\cdot)\Big\rVert_{\ldeltastar}\leq 10^{|\mathcal{O}[\l]|}\Big(C_{\gamma,\delta}\lVert h\rVert_{\timeldelta}\Big)\Big\lVert \mathbf{C}_{h}[\T,\mathcal{L}_M\cup\{\l\}=M_{\gamma(0)}](t,\cdot,\cdot)\Big\rVert_{\ldeltastar}.
\end{equation}

\end{lemma}

\section{Well-posedness and Weak Stability of Boltzmann Equation}\label{sec:proof_of_main}

\subsection{Uniform $L_t^{\infty}\ldeltastar$-Bound via Homogenization Procedure}

In this section, we establish the estimate of the $L_{t}^{\infty}\ldeltastar$-norm for the terms $\mathbf{C}_{f^0}[\T]$ in Definition \ref{def:C_expansion}. These terms are used in the definition of $f(t)$ in \eqref{eq:formal_summation}. Specifically, the estimate in this section shows the summation in \eqref{eq:formal_summation} absolutely converges, and $f(t)$ is well-defined in $L^{\infty}\ldeltastar$.

\begin{definition}\label{def:three_types}
Given a decorated tree $\T$
\begin{itemize}
    \item we call it a Type-I decorated tree if: in the primary sub-tree $\tprime$, the two child nodes of $\mathfrak{r}(\tprime)$ are two branching nodes,
    \item we call it a Type-II decorated tree if: in the primary sub-tree $\tprime$, the two child nodes of $\mathfrak{r}(\tprime)$ are one branching node and one leaf
    \item we call it a Type-III decorated tree if: in the primary sub-tree $\tprime$, the root $\mathfrak{r}(\tprime)$ is a leaf
\end{itemize}
\end{definition}

\begin{proposition}\label{prop:uniform_L1_estimate}
Consider any $(\gamma>0,\delta>0)$ and any $T>0$, there exist positive parameters $(\gamma(0),\gamma_*,\delta_*)$ with $0<\gamma_*<\gamma(0)<\gamma$ such that the following property is true. For any $h\in \timeldelta([0,T]\times \mathds{R}^3\times \mathds{R}^3)$ and any decorated tree $\T$
\begin{equation*}
\begin{split}
&\Big\lVert \mathbf{C}_{h}[\T]\Big\rVert_{L^{\infty}\ldeltastar([0,T]\times \mathds{R}^3\times \mathds{R}^3)}\leq \Big(C_{\gamma,\delta}\lVert h\rVert_{\timeldelta}\Big)^{|\mathcal{L}(\T)|}\Big\lVert \mathbf{C}_{M_{\gamma(0)}}[\T]\Big\rVert_{L^{\infty}\ldeltastar([0,T]\times \mathds{R}^3\times \mathds{R}^3)}.
\end{split}
\end{equation*}
\end{proposition}

\begin{proof}

\textit{Part 1: Absolute continuity.}

We apply the homogenization procedure in Definition \ref{def:homogenization_procedure} to the decorated tree $\T$ with $\mathcal{L}_M(\T)=\varnothing$. For an arbitrary homogenization operation used in the homogenization procedure, denote $\mathcal{L}_M(\T)$ as the set of Maxwellian leaves before the operation, and $\mathcal{L}_M'(\T)$ as the set of Maxwellian leaves after the operation. 

According to the lemmas proved in Section \ref{sec:homogenization_operation}, either $\mathbf{C}_{h}[\T,\mathcal{L}_M(\T)=M_{\gamma(0)}]$ is already absolutely continuous with respect to the Lebesgue measure (Lemma \ref{lem:double_null}, \ref{lem:single_null}, and \ref{lem:indi_homo}), or $\mathbf{C}_{h}[\T,\mathcal{L}_M(\T)=M_{\gamma(0)}]$ and $\mathbf{C}_{h}[\T,\mathcal{L}_M'(\T)=M_{\gamma(0)}]$ have the same relation with respect to the Lebesgue measure. After all the homogenization operations, the Boltzmann integral becomes $\mathbf{C}_{M_{\gamma(0)}}[\T]$, which is absolutely continuous. Thus $\mathbf{C}_{h}[\T]$ is also absolutely continuous.

\underline{\textit{Part 2: Uniform $L^{\infty}\ldeltastar([0,T]\times \mathds{R}^3\times \mathds{R}^3)$ estimate.}}

Apply the homogenization procedure to the decorated tree $\T$. According to the estimates in Section \ref{sec:homogenization_operation} on the multiplicative errors of homogenization operations, there exist positive integers $A_1,A_2$ such that
\begin{equation*}
\begin{split}
&\Big\lVert \mathbf{C}_{h}[\T](t,\cdot,\cdot)\Big\rVert_{\ldeltastar}\leq 10^{A_1}\Big(C_{\gamma,\delta}\lVert h\rVert_{\timeldelta}\Big)^{A_2}\Big\lVert \mathbf{C}_{M_{\gamma(0)}}[\T](t,\cdot,\cdot)\Big\rVert_{\ldeltastar}
\end{split}
\end{equation*}
A precise estimate of the constants $A_1,A_2$ is needed to prove Proposition \ref{prop:uniform_L1_estimate}. The counting of $A_2$ is direct, since for any homogenization operation, the increment of $A_2$ is the number of leaves homogenized. This means 
\begin{equation*}
A_2=|\mathcal{L}(\T)|.
\end{equation*}
For the estimate of $A_1$, notice that for any homogenization operation (with $\l$ denoted as one leaf homogenized by this operation), the increment of $A_1$ is 
\begin{equation}\label{eq:overlap_counting}
\Big|\mathcal{O}[\l]\Big|+\mathds{1}_{\l \textup{ has a collision-parent $\n$}}\Big|\mathcal{O}[\n]\Big|+\mathds{1}_{\n\textup{ has another non-Maxwellian collision-child $\l_*$}}\Big|\mathcal{O}[\l_*]\Big|.
\end{equation}
After each homogenization operation, the newly homogenized leaves will not be homogenized by any other operations, and the collision-parent of these leaves now has only Maxwellian leaves. Therefore the overlap nodes in \eqref{eq:overlap_counting} will not be counted in the increments of $A_1$ for other homogenization operations. Consequently
\begin{equation}
A_1\leq |\mathcal{O}(\T)|\leq |\mathcal{L}(\T)|.
\end{equation}
This concludes the proof of the proposition.

\end{proof}

\begin{proposition}\label{prop:point_homogenization}
Consider any $(\gamma>0,\delta>0)$ and any $T>0$, and any $h\in L^{\infty}\ldelta([0,T]\times \mathds{R}^3\times \mathds{R}^3)$. For any Type-I decorated tree $\T$, and any $(t,x,v)\in [0,T]\times \mathds{R}^3\times \mathds{R}^3$ we have
\begin{equation}
\begin{split}
&\Bigg|\mathbf{C}_{h}[\T](t,x,v)\Bigg|\leq \Big(C_{\gamma,\delta}\lVert h\rVert_{\timeldelta}\Big)^{|\mathcal{L}(\T)|}\Bigg|\mathbf{C}_{M_{\gamma(0)}}[\T](t,x,v)\Bigg|
\end{split}
\end{equation}
\end{proposition}

\begin{proof}
Recall the graph $\mathcal{G_T}$ constructed in Definition \ref{def:tree_decompose}, and the sequence $\{\T_i^+\}_{0\leq i\leq m}$ derived by deleting all the overlap nodes. Apply the homogenization procedure in Definition \ref{def:homogenization_procedure} to $\T$ with $\mathcal{L}_M(\T)=\varnothing$, until for any $1\leq i\leq m$ we have $\mathcal{L}_M(\T_i^+)=\mathcal{L}(\T_i^+)$. Then we notice in the primary sub-tree $\T_0^+$, for any leaf $\l\in \mathcal{L}(\T_0^+)$, it has a collision-parent $\n$ and $\n$ also has a collision-parent, according to the Type-I decorated tree in Definition \ref{def:three_types}. 

Then we continue the homogenization procedure, and we only need to use Lemmas \ref{lem:double_collision} and  \ref{lem:single_collsion} to estimate the error
\begin{equation*}
\begin{split}
\mathbf{C}_{h}[\T](t,x,v)\leq 10^{A_1}\Big(C_{\gamma,\delta}\lVert h\rVert_{\timeldelta}\Big)^{A_2} \mathbf{C}_{M_{\gamma(0)}}[\T](t,x,v),
\end{split}
\end{equation*}
with $A_1\leq |\mathcal{L}(\T)|$ and $A_2\leq |\mathcal{L}(\T)|$. This concludes the proof of the proposition.
\end{proof}

\subsection{Proof of Well-posedness and Energy Conservation}

\begin{proof}
For any time $0\leq t\leq T_*$, recall that $f(t)$ is defined in \eqref{eq:formal_summation} as the following summation
\begin{equation}\label{eq:formal_definition}
f(t):=\sum_{\T:\textup{decorated tree}}(-1)^{|\mathcal{O}(\T)|}\mathbf{C}_{f^0}[\T](t).
\end{equation}

\underline{\textit{Part 1: Convergence of summation in $L^{\infty}\ldeltastar([0,T_*]\times \mathds{R}^3\times \mathds{R}^3)$.}}

Classical results show that for any decorated tree $\T$, we have
\begin{equation}\label{eq:bound_for_maxwellian}
\Big\lVert \mathbf{C}_{M_{\gamma(0)}}[\mathcal{T}]\Big\rVert_{L^{\infty}\ldeltastar}\leq C_{\gamma,\delta}^{|\mathcal{L}(\T)|}T_*^{|\mathcal{L}(\T)|-1}
\end{equation}
Proposition \ref{prop:uniform_L1_estimate} shows the $L^{\infty}\ldeltastar$-norm of $\mathbf{C}_{f^0}[\T]$ can be bounded from above by the norm of $\mathbf{C}_{M_{\gamma(0)}}[\T]$ with multiplicative errors. Therefore we can apply Proposition \ref{prop:uniform_L1_estimate} to \eqref{eq:bound_for_maxwellian}, and choose $T_*>0$ small enough
\begin{equation}
\Big\lVert \mathbf{C}_{f^0}[\T]\Big\rVert_{L^{\infty}\ldeltastar}\leq \Big(C_{\gamma,\delta}\lVert f^0\rVert_{\ldelta}\Big)^{|\mathcal{L}(\T)|}T_*^{|\mathcal{L}(\T)|-1}\leq C_{\gamma,\delta}\lVert f^0\rVert_{\ldelta}\Big(\frac{1}{10}\Big)^{|\mathcal{L}(\T)|}.
\end{equation}
Since the number of decorated trees $\T$ with $|\mathcal{L}(\T)|=N$ is bounded by $O(8^N)$, the summation in \eqref{eq:formal_definition} converges in $L^{\infty}\ldeltastar$. Consequently the function $f(t)$ in \eqref{eq:formal_definition} lies in the space $L^{\infty}\ldeltastar$.

\underline{\textit{Part 2: $f(t)$ as an energy-conserving mild solution.}}

It is sufficient to prove the following decomposition for all time $0\leq t\leq T_*$ and smooth and compactly supported test functions $\Phi(x,v)$
\begin{equation}\label{eq:duhamel_decomposition}
\begin{split}
&\Big\langle \Phi,\mathbf{Q}\Big(\sum_{\T_l}(-1)^{|\mathcal{O}(\T_l)|}\mathbf{C}_{f^0}[\T_l](t),\sum_{\T_r}(-1)^{|\mathcal{O}(\T_r)|}\mathbf{C}_{f^0}[\T_r](t)\Big)\Big\rangle\\
&=\sum_{\T_l}\sum_{\T_r}\Big\langle \Phi, \mathbf{Q}\Big((-1)^{|\mathcal{O}(\T_l)|}\mathbf{C}_{f^0}[\T_l](t),(-1)^{|\mathcal{O}(\T_r)|}\mathbf{C}_{f^0}[\T_r](t)\Big)\Big\rangle.
\end{split}
\end{equation}
Equation \eqref{eq:duhamel_decomposition} yields for any times $0\leq s\leq t\leq T_*$ and test function $\varphi$
\begin{equation}
\begin{split}
&\Big\langle \varphi, \mathbf{S}_{t-s}\mathbf{Q}\Big(\sum_{\T_l}(-1)^{|\mathcal{O}(\T_l)|}\mathbf{C}_{\fl^0}[\T_l](s),\sum_{\T_r}(-1)^{|\mathcal{O}(\T_r)|}\mathbf{C}_{\fl^0}[\T_r](s)\Big) \Big\rangle\\
&=\Big\langle \mathbf{S}_{-(t-s)}\varphi, \mathbf{Q}\Big(\sum_{\T_l}(-1)^{|\mathcal{O}(\T_l)|}\mathbf{C}_{\fl^0}[\T_l](s),\sum_{\T_r}(-1)^{|\mathcal{O}(\T_r)|}\mathbf{C}_{\fl^0}[\T_r](s)\Big) \Big\rangle\\
&=\sum_{\T_l}\sum_{\T_r}\Big\langle \mathbf{S}_{-(t-s)}\varphi, \mathbf{Q}\Big((-1)^{|\mathcal{O}(\T_l)|}\mathbf{C}_{\fl^0}[\T_l](s),(-1)^{|\mathcal{O}(\T_r)|}\mathbf{C}_{\fl^0}[\T_r](s)\Big) \Big\rangle,
\end{split}
\end{equation}
where $\mathbf{S}_{-(t-s)}\varphi$ is seen as the test function $\Phi$ in \eqref{eq:duhamel_decomposition}. Consequently $f(t)$ in \eqref{eq:formal_definition} satisfies for any test function $\varphi(x,v)$
\begin{equation}
\begin{split}
&\Big\langle \varphi,\mathbf{S}_tf^0+\int_0^t\mathbf{S}_{t-s}\mathbf{Q}\Big(\sum_{\T_l}(-1)^{|\mathcal{O}(\T_l)|}\mathbf{C}_{f^0}[\T_l](s),\sum_{\T_r}(-1)^{|\mathcal{O}(\T_r)|}\mathbf{C}_{f^0}[\T_r](s)\Big)ds\Big\rangle\\
=&\Big\langle \varphi,\mathbf{S}_tf^0+\sum_{\T_l}\sum_{\T_r}\int_0^t\mathbf{S}_{t-s}\mathbf{Q}\Big((-1)^{|\mathcal{O}(\T_l)|}\mathbf{C}_{f^0}[\T_l](s),(-1)^{|\mathcal{O}(\T_r)|}\mathbf{C}_{f^0}[\T_r](s)\Big)ds\Big\rangle\\
&=\Big\langle \varphi,\sum_{\T}(-1)^{|\mathcal{O}(\T)|}\mathbf{C}_{f^0}[\T](t)\Big\rangle=\Big\langle \varphi,f(t)\Big\rangle.
\end{split}
\end{equation}
This shows the two sides of \eqref{eq:mild_formulation} are equal almost everywhere, thus $\fl(t)$ in \eqref{eq:formal_definition} is a mild solution of the Boltzmann equation \eqref{eq:Boltzmann}. Now we prove \eqref{eq:duhamel_decomposition}. The LHS of \eqref{eq:duhamel_decomposition} equals 
\begin{equation*}
\begin{split}
&\frac{1}{2}\int \big((v_*-v)\cdot\omega\big)_+\Big(\Phi(x,v')+\Phi(x,v_*')-\Phi(x,v_*)-\Phi(x,v)\Big)\Big(\sum_{\T_l}(-1)^{|\mathcal{O}(\T_l)|}\mathbf{C}_{f^0}[\T_l](t,x,v)\Big)\Big(\sum_{\T_r}(-1)^{|\mathcal{O}(\T_r)|}\mathbf{C}_{f^0}[\T_r](t,x,v_*)\Big)d\omega dv_*dvdx\\
&=\int \Phi_{\textup{integral}}(x,v,v_*)\Big(\sum_{\T_l}(-1)^{|\mathcal{O}(\T_l)|}\mathbf{C}_{f^0}[\T_l](t,x,v)\Big)\Big(\sum_{\T_r}(-1)^{|\mathcal{O}(\T_r)|}\mathbf{C}_{f^0}[\T_r](t,x,v_*)\Big)dv_*dvdx,
\end{split}
\end{equation*}
where $\Phi_{\textup{integral}}$ is defined as
\begin{equation}
\Phi_{\textup{integral}}:=\frac{1}{2}\int \big((v_*-v)\cdot\omega\big)_+\Big(\Phi(x,v')+\Phi(x,v_*')-\Phi(x,v)-\Phi(x,v_*)\Big)d\omega.
\end{equation}
There exists a constant $C_{\Phi}>0$ depending on $\Phi$ such that 
\begin{equation}
|\Phi_{\textup{integral}}(x,v,v_*)|\leq C_{\Phi}e^{|v|+|v_*|}.
\end{equation}
The uniform estimate of the $L_{t}^{\infty}\ldeltastar$-norm of $\mathbf{C}_{f^0}[\T]$ in Proposition \ref{prop:uniform_L1_estimate} implies
\begin{equation}
\begin{split}
&\int \Phi_{\textup{integral}}(x,v,v_*)\Big(\sum_{\T_l}(-1)^{|\mathcal{O}(\T_l)|}\mathbf{C}_{f^0}[\T_l](t,x,v)\Big)\Big(\sum_{\T_r}(-1)^{|\mathcal{O}(\T_r)|}\mathbf{C}_{f^0}[\T_r](t,x,v_*)\Big)dv_*dvdx\\
=&\sum_{\T_l}\sum_{\T_r}\int \Phi_{\textup{integral}}(x,v,v_*)(-1)^{|\mathcal{O}(\T_l)|}\mathbf{C}_{f^0}[\T_l](t,x,v)(-1)^{|\mathcal{O}(\T_r)|}\mathbf{C}_{f^0}[\T_r](t,x,v_*)dv_*dvdx\\
&=\sum_{\T_l}\sum_{\T_r}\Big\langle \Phi, \mathbf{Q}\Big((-1)^{|\mathcal{O}(\T_l)|}\mathbf{C}_{f^0}[\T_l](t),(-1)^{|\mathcal{O}(\T_r)|}\mathbf{C}_{f^0}[\T_r](t)\Big)\Big\rangle.
\end{split}
\end{equation}
The energy-conservation follows directly from the fact that $f\in L_{t}^{\infty}\ldeltastar$. For non-negativity, we can reorganize the sum of Boltzmann integrals \eqref{eq:formal_summation} into a summation index by primary sub-trees
\begin{equation}
f(t)=\sum_{\T_+:\textup{collisional tree}}\Big(\sum_{\T:\textup{ has primary sub-tree $\tprime=\T_+$}}(-1)^{|\mathcal{O}(\T)|}\mathbf{C}_{f^0}[\T](t)\Big).
\end{equation}
For each $\T_+$, the summation can be verified to be non-negative using Lemma \ref{lem:full_to_prime}.

\underline{\textit{Part 3: Uniqueness.}}

Suppose $g(t)$ is a mild solution of the Boltzmann equation in $L^{\infty}\ldeltastar([0,T_*]\times \mathds{R}^3\times \mathds{R}^3)$ different from $f(t)$. Define $t_*\geq 0$ as the time when $f(t)$ and $g(t)$ become different $t_*:=\inf \big\{t>0 \ \big|\ \lVert f(t)-g(t)\rVert_{\ldeltastar}\not=0\big\}$. Only in this proof, we will use notations $f(t)$ and $g(t)$ as the following for simplicity
\begin{equation}
f(t):=f(t_*+t),\quad g(t):=g(t_*+t).
\end{equation}
Since $f(t)$ and $g(t)$ are mild solutions of the Boltzmann equation \eqref{eq:Boltzmann}, the functions $f(t)$ and $g(t)$ are also mild solutions of the Boltzmann equation
\begin{equation}\label{eq:mild form f and g}
\left\{
\begin{split}
&f(t)=\mathbf{S}_tf(0)+\int_0^t\mathbf{S}_{t-s}\mathbf{Q}\Big(f(s),f(s)\Big)ds,\\
&g(t)=\mathbf{S}_tf(0)+\int_0^t\mathbf{S}_{t-s}\mathbf{Q}\Big(g(s),g(s)\Big)ds.\\
\end{split}
\right.
\end{equation}

\begin{lemma}\label{lem:iterated_duhamel}
Consider any positive integer $N\geq 1$, there exists a set $\mathcal{S}_{> N}$ of decorated trees with $|\mathcal{L}(\T)|\geq N+1$, and sets $\mathcal{L}^{(N)}(\T)\subset \mathcal{L}(\T)$ associated with each $\T\in \mathcal{S}_{> N}$ such that
\begin{equation}\label{eq:iterated_duhamel}
g(t)=\sum_{\T:|\mathcal{L}(\T)|\leq N}(-1)^{|\mathcal{O}(\T)|}\mathbf{C}_{f(0)}[\T](t)+\sum_{\T\in \mathcal{S}_{> N}}(-1)^{|\mathcal{O}(\T)|}\mathbf{C}_{f(0)}\big[\T,\mathcal{L}^{(N)}(\T)=g\big](t).
\end{equation}
\end{lemma}

\begin{proof}
This lemma can be proved by induction on $N\geq 1$. When $N=1$, the mild formulation of $g(t)$ in \eqref{eq:mild form f and g} implies 
\begin{equation}
g(t)=\mathbf{C}_{f(0)}[\circ](t)+\mathbf{C}_{g}[\CollisionFork](t)-\mathbf{C}_{g}[\OverlapFork](t)=\mathbf{C}_{f(0)}[\circ](t)+\mathbf{C}_{f(0)}\Big[\CollisionFork,\mathcal{L}(\CollisionFork)=g\Big](t)-\mathbf{C}_{f(0)}\Big[\OverlapFork,\mathcal{L}(\OverlapFork)=g\Big](t).
\end{equation}
In this case
\begin{equation}
\mathcal{S}_{>1}:=\Big\{\CollisionFork,\OverlapFork\Big\},\quad \mathcal{L}^{(1)}(\CollisionFork)=\mathcal{L}(\CollisionFork),\quad \mathcal{L}^{(1)}(\OverlapFork)=\mathcal{L}(\OverlapFork).
\end{equation}
Assume for an integer $N\geq 1$, it satisfies \eqref{eq:iterated_duhamel}. Now we prove the statement for $N_*=N+1$. Insert \eqref{eq:iterated_duhamel} into the mild formulation \eqref{eq:mild form f and g} of $g(t)$
\begin{equation}\label{eq:inserted_duhamel}
\begin{split}
g(t)=\mathbf{S}_{t}f(0)&+\sum_{\T_l:|\mathcal{L}(\T_l)|\leq N}\sum_{\T_r:|\mathcal{L}(\T_r)|\leq N}(-1)^{|\mathcal{O}(\T_l)|+|\mathcal{O}(\T_r)|}\int_0^t\mathbf{S}_{t-s}\mathbf{Q}\Big(\mathbf{C}_{f(0)}[\T_l](s),\mathbf{C}_{f(0)}[\T_r](s)\Big)ds\\
&+\sum_{\T_l:|\mathcal{L}(\T_l)|\leq N}\sum_{\T_r\in \mathcal{S}_{> N}}(-1)^{|\mathcal{O}(\T_l)|+|\mathcal{O}(\T_r)|}\int_0^t\mathbf{S}_{t-s}\mathbf{Q}\Big(\mathbf{C}_{f(0)}[\T_l](s),\mathbf{C}_{f(0)}[\T_r,\mathcal{L}^{(N)}(\T_r)=g](s)\Big)ds\\
&+\sum_{\T_l\in \mathcal{S}_{> N}}\sum_{\T_r:|\mathcal{L}(\T_r)|\leq N}(-1)^{|\mathcal{O}(\T_l)|+|\mathcal{O}(\T_r)|}\int_0^t\mathbf{S}_{t-s}\mathbf{Q}\Big(\mathbf{C}_{f(0)}[\T_l,\mathcal{L}^{(N)}(\T_l)=g](s),\mathbf{C}_{f(0)}[\T_r](s)\Big)ds\\
&+\sum_{\T_l\in \mathcal{S}_{>N}}\sum_{\T_r\in \mathcal{S}_{>N}}(-1)^{|\mathcal{O}(\T_l)|+|\mathcal{O}(\T_r)|}\int_0^t\mathbf{S}_{t-s}\mathbf{Q}\Big(\mathbf{C}_{f(0)}[\T_l,\mathcal{L}^{(N)}(\T_l)=g](s),\mathbf{C}_{f(0)}[\T_r,\mathcal{L}^{(N)}(\T_r)=g](s)\Big)ds
\end{split}
\end{equation}
The first line of the RHS of \eqref{eq:inserted_duhamel} contains the following summation
\begin{equation}
\sum_{\T:|\mathcal{L}(\T)|\leq N+1}(-1)^{|\mathcal{O}(\T)|}\mathbf{C}_{f(0)}[\T](t).
\end{equation}
The rest of the RHS of \eqref{eq:inserted_duhamel} can be organized as a summation over $\T\in \mathcal{S}_{>N+1}$, which is constructed as the following (here we temporarily write $\mathcal{L}^{(N)}(\T)=\varnothing$ if $|\mathcal{L}(\T)|\leq N$)
\begin{equation}
\begin{split}
&\mathcal{S}_{>N+1}:=\Big\{\T\Big| \mathcal{T}_i\in \mathcal{S}_{>N}\textup{ or }|\mathcal{L}(\T_i)|\leq N,\ \textup{for any }i\in\{l,r\} \Big\}\setminus\Big\{\mathcal{T}\Big|\ |\mathcal{L}(\T)|\leq N+1\Big\},\\
&\mathcal{L}^{(N+1)}(\T):=\mathcal{L}^{(N)}(\T_l)\cup\mathcal{L}^{(N)}(\T_r).
\end{split}
\end{equation}
\end{proof}

Lemma \ref{lem:iterated_duhamel} allows us to compare $f(t)$ and $g(t)$ with arbitrarily high accuracy. For arbitrary $N\geq 1$, we can compare $f(t)$ and $g(t)$
\begin{equation}
\begin{split}
&f(t)=\underbrace{\sum_{\T:|\mathcal{L}(\T)|\leq N}(-1)^{|\mathcal{O}(\T)|}\mathbf{C}_{f(0)}[\T](t)}_{\textup{\blue{common part}}}+\sum_{\T:|\mathcal{L}(\T)|\geq N+1}(-1)^{|\mathcal{O}(\T)|}\mathbf{C}_{f(0)}[\T](t)\\
&g(t)=\underbrace{\sum_{\T:|\mathcal{L}(\T)|\leq N}(-1)^{|\mathcal{O}(\T)|}\mathbf{C}_{f(0)}[\T](t)}_{\textup{\blue{common part}}}+\underbrace{\sum_{\T\in \mathcal{S}_{> N}}(-1)^{|\mathcal{O}(\T)|}\mathbf{C}_{f(0)}\big[\T,\mathcal{L}^{(N)}(\T)=g\big](t)}_{\textup{\blue{some leaves replaced by $g$}}}.
\end{split}
\end{equation}
Choose $(\gamma_{\dagger},\delta_{\dagger})$ small enough determined by $(\gamma_*,\delta_*)$
\begin{equation}
\lVert f(t)-g(t)\rVert_{\mathbf{L}_{\gamma_\dagger,\delta_\dagger}^1}\leq \sum_{\T:|\mathcal{L}(\T)|\geq N+1}\Big\lVert\mathbf{C}_{f(0)}[\T](t)\Big\rVert_{\mathbf{L}_{\gamma_\dagger,\delta_\dagger}^1}+\sum_{\T\in \mathcal{S}_{> N}}\Big\lVert\mathbf{C}_{f(0)}\big[\T,\mathcal{L}^{(N)}(\T)=g\big](t)\Big\rVert_{\mathbf{L}_{\gamma_\dagger,\delta_\dagger}^1}
\end{equation}
Defining function $G$ as
\begin{equation}
G(t,x,v):=\max\Big\{\big(\mathbf{S}_tf(0)\big)(x,v),g(t,x,v)\Big\},
\end{equation}
we further have
\begin{equation}
\lVert f(t)-g(t)\rVert_{\mathbf{L}_{\gamma_\dagger,\delta_\dagger}^1}\leq 2\sum_{\T:|\mathcal{L}(\T)|\geq N+1}\Big\lVert\mathbf{C}_{G}[\T](t)\Big\rVert_{\mathbf{L}_{\gamma_\dagger,\delta_\dagger}^1}.
\end{equation}
Taking $N\rightarrow +\infty$, we have $f(t)=g(t)$.

\end{proof}

\subsection{Proof of Weak Stability Estimate in Theorem \ref{th:well-posedness}}

In this subsection we prove the weak stability estimate \eqref{eq:mixing_PDE_statement} stated in Theorem \ref{th:well-posedness}. The proof is based on the uniform $\timeldeltastar$ estimate derived in Proposition \ref{prop:uniform_L1_estimate}, which allows us to safely truncate the series expansion of $f_n(t)$ and $f(t)$. Then we develop a new formulation of the Carleman representation in Definition \ref{def:carleman_measure}, which allows us to decouple the singularities of initial data.

\textit{Step 1: Reduction into term-wise convergence with arbitrarily slow rate.}

For arbitrary $M\geq 1$, we can truncate the series-expansion of $f_n(t)$ and $f(t)$ by $|\mathcal{L}(\T)|\leq M$
\begin{equation}
\begin{split}
&\int_{\mathds{R}^3} \varphi(v)f_n(t,x,v)dv=\sum_{\mathcal{T}:|\mathcal{L}(\T)|\leq M}(-1)^{|\mathcal{O}(\T)|}\int_{\mathds{R}^3} \varphi(v)\mathbf{C}_{f_n^0}[\mathcal{T}](t,x,v)dv+\int \textbf{Err}_{f_n^0,M+1}(t,x,v)\varphi(v)dv\\
&\int_{\mathds{R}^3} \varphi(v)f(t,x,v)dv=\sum_{\mathcal{T}:|\mathcal{L}(\T)|\leq M}(-1)^{|\mathcal{O}(\T)|}\int_{\mathds{R}^3} \varphi(v)\mathbf{C}_{f^0}[\mathcal{T}](t,x,v)dv+\int \textbf{Err}_{f^0,M+1}(t,x,v)\varphi(v)dv
\end{split}
\end{equation}
with the error terms $\textbf{Err}_{f_n^0,M+1}$ and $\textbf{Err}_{f^0,M+1}$ defined as
\begin{equation}
\begin{split}
& \textbf{Err}_{f_n^0,M+1}(t):=\sum_{\mathcal{T}:|\mathcal{L}(\T)|\geq M+1}(-1)^{|\mathcal{O}(\T)|}\mathbf{C}_{f_n^0}[\mathcal{T}](t),\quad \textbf{Err}_{f^0,M+1}(t):=\sum_{\mathcal{T}:|\mathcal{L}(\T)|\geq M+1}(-1)^{|\mathcal{O}(\T)|}\mathbf{C}_{f^0}[\mathcal{T}](t)
\end{split}
\end{equation}
As we take $M\rightarrow +\infty$, the norms $\lVert \textbf{Err}_{f^0,M+1}\rVert_{\ldeltastar}$ and $\lVert \textbf{Err}_{f_n^0,M+1}\rVert_{\ldeltastar}$ converge to $0$ uniformly in $n$. Consequently the following convergence is also uniform in $n$
\begin{equation}
\lim_{M\rightarrow +\infty}\int \textbf{Err}_{f_n^0,M+1}(t,x,v)\varphi(v)dv=0,\quad \lim_{M\rightarrow +\infty}\int \textbf{Err}_{f^0,M+1}(t,x,v)\varphi(v)dv=0.
\end{equation}
Now it is sufficient to prove for any decorated tree $\mathcal{T}$
\begin{equation}\label{eq:mixing_termwise}
\lim_{n\rightarrow +\infty}\int_{\mathds{R}^3} \varphi(v)\mathbf{C}_{f_n^0}[\mathcal{T}](t,x,v)dv=\int_{\mathds{R}^3} \varphi(v)\mathbf{C}_{f^0}[\mathcal{T}](t,x,v)dv.
\end{equation}

\textit{Step 2: Truncation of large velocity of each term}

For any $R\geq 0$, define $\mathbf{C}_{f_n^0}^{\leq R}[\T](t,x,v)$ and $\mathbf{C}_{f^0}^{\leq R}[\T](t,x,v)$ as the truncated collision integrals
\begin{equation}
\begin{split}
&\mathbf{C}_{f_n^0}^{\leq R}[\T](t,x,v):=\int \Bigg[\prod_{\l\in \mathcal{L}(\T)}\Big(f_n^0\big(\x_{\l},v_{\l}\big)\mathds{1}_{|v_{\l}|\leq R}\Big)\Bigg]\mathds{1}_{0\leq s_{\mathfrak{r}(\T)}\leq t}d[v]_{\nnew^+(\mathcal{T})}d[s]_{\mathcal{B}(\T)}\\
&\mathbf{C}_{f^0}^{\leq R}[\T](t,x,v):=\int \Bigg[\prod_{\l\in \mathcal{L}(\T)}\Big(f^0\big(\x_{\l},v_{\l}\big)\mathds{1}_{|v_{\l}|\leq R}\Big)\Bigg]\mathds{1}_{0\leq s_{\mathfrak{r}(\T)}\leq t}d[v]_{\nnew^+(\mathcal{T})}d[s]_{\mathcal{B}(\T)}
\end{split}
\end{equation}
Now we prove 
\begin{equation}\label{eq:uniform_velocity_truncation}
\lim_{R\rightarrow +\infty}\sup_{n<+\infty}\Bigg|\int \varphi(v)\Big(\mathbf{C}_{f_n^0}[\T](t,x,v)-\mathbf{C}_{f_n^0}^{\leq R}[\T](t,x,v)\Big)dv\Bigg|=0.
\end{equation}
Using the fact that
\begin{equation}
\begin{split}
&\prod_{\l\in \mathcal{L}(\T)}f_n^0\big(\x_{\l},v_{\l}\big)-\prod_{\l\in \mathcal{L}(\T)}\Big(f_n^0\big(\x_{\l},v_{\l}\big)\mathds{1}_{|v_{\l}|\leq R}\Big)\\
&=\sum_{\ldif:\ldif\subset \mathcal{L}(\T),\ldif\not=\varnothing} \underbrace{\Bigg[\prod_{\l\in \mathcal{L}(\T)\setminus\ldif}\Big(f_n^0\big(\x_{\l},v_{\l}\big)\mathds{1}_{|v_{\l}|\leq R}\Big)\Bigg]}_{\textup{\blue{velocity below truncation}}} \underbrace{\Bigg[\prod_{\l\in \ldif}\Big(f_n^0\big(\x_{\l},v_{\l}\big)\mathds{1}_{|v_{\l}|> R}\Big)\Bigg]}_{\textup{\blue{velocity above truncation}}},
\end{split}
\end{equation}
we can decompose the difference $\mathbf{C}_{f_n^0}[\T]-\mathbf{C}_{f_n^0}^{\leq R}[\T]$
\begin{equation}
\begin{split}
&\mathbf{C}_{f_n^0}[\T]-\mathbf{C}_{f_n^0}^{\leq R}[\T]\\
&=\sum_{\ldif:\ldif\subset \mathcal{L}(\T),\ldif\not=\varnothing}\int \underbrace{\Bigg[\prod_{\l\in \mathcal{L}(\T)\setminus\ldif}\Big(f_n^0\big(\x_{\l},v_{\l}\big)\mathds{1}_{|v_{\l}|\leq R}\Big)\Bigg]}_{\textup{\blue{velocity below truncation}}} \underbrace{\Bigg[\prod_{\l\in \ldif}\Big(f_n^0\big(\x_{\l},v_{\l}\big)\mathds{1}_{|v_{\l}|> R}\Big)\Bigg]}_{\textup{\blue{velocity above truncation}}}\mathds{1}_{0\leq s_{\mathfrak{r}(\T)}\leq t}d[v]_{\nnew^+(\mathcal{T})}d[s]_{\mathcal{B}(\T)}\\
&\leq e^{-\frac{\gamma}{2}R^2}\sum_{\ldif:\ldif\subset \mathcal{L}(\T),\ldif\not=\varnothing}\int \Bigg[\prod_{\l\in \mathcal{L}(\T)\setminus\ldif}\Big(f_n^0\big(\x_{\l},v_{\l}\big)\mathds{1}_{|v_{\l}|\leq R}e^{\frac{\gamma}{2}|v_{\l}|}\Big)\Bigg] \Bigg[\prod_{\l\in \ldif}\Big(f_n^0\big(\x_{\l},v_{\l}\big)\mathds{1}_{|v_{\l}|> R}e^{\frac{\gamma}{2}|v_{\l}|^2}\Big)\Bigg]\mathds{1}_{0\leq s_{\mathfrak{r}(\T)}\leq t}d[v]_{\nnew^+(\mathcal{T})}d[s]_{\mathcal{B}(\T)}
\end{split}
\end{equation}

\textit{Step 3: Reversed Carleman measure.}

\begin{definition}
For any collisional decorated tree $\T$, the \underline{reversed Carleman measure} $d^{\uparrow}[v]_{\mathcal{N}^+(\T)}$ is defined recursively. If $\mathcal{T}$ is a single node $\T=\circ$, then $\revcar{\circ}=dv$. If the root $\mathfrak{r}(\T)$ of $\T$ is a branching collision node, then denote $\T_l$ as the left sub-tree and $\T_r$ as the right sub-tree
\begin{equation}
\revcar{\T}:= \pi|v_{l}-v_r|\overline{d\pi}_{\mathtt{S}_{v_l,v_r}}(v_{\mathfrak{r}(\T)})\revcar{\T_l}\revcar{\T_r},\quad v_l:=v_{\mathfrak{r}(\T_l)},\quad v_r:=v_{\mathfrak{r}(\T_r)}.
\end{equation}
Take an arbitrary decorated tree $\T$, with the sequence of collisional decorated trees $\{\T_{i}^+\}_{0\leq i\leq m}$ in Definition \ref{def:tree_decompose}. Define $d^{\uparrow}[v]_{\mathcal{N}^+(\T)}$ as
\begin{equation}\label{eq:carleman_measure_of_velocities_reversed}
d^{\uparrow}[v]_{\mathcal{N}^+(\T)}:=\Big(\prod_{\o\in \mathcal{O}(\T)}\pi|v_{\oold}-v_{\onew}|\Big)d^{\uparrow}[v]_{\mathcal{N}^+(\T_{0}^+)}\prod_{1\leq i\leq m}d^{\uparrow}[v]_{\mathcal{N}^+(\T_i^+)}.
\end{equation}
\end{definition}

By induction we can prove the reversed Carleman measure is equivalent to the Carleman measure in Definition \ref{def:carleman_measure}
\begin{equation}\label{eq:reversed_carleman_equivalence}
\int H\Big([v]_{\mathcal{N}^+(\T)}\Big)d[v]_{\mathcal{N}^+(\T)}=\int H\Big([v]_{\mathcal{N}^+(\T)}\Big)d^{\uparrow}[v]_{\mathcal{N}^+(\T)}.
\end{equation}
Indeed, given a decorated tree $\T$ with the left sub-tree $\T_l$ and the right sub-tree $\T_r$, assume the root $\mathfrak{r}(\T)$ is a collision node. Suppose $\T_l$ and $\T_r$ separately satisfy \eqref{eq:reversed_carleman_equivalence}. Using $v_l$ and $v_r$ to denote $v_{\mathfrak{r}(\T_l)}$ and $v_{\mathfrak{r}(\T_r)}$, we have
\begin{equation}
\begin{split}
\int H\Big([v]_{\mathcal{N}^+(\T)}\Big)d[v]_{\mathcal{N}^+(\T)}&=\int H\Big([v]_{\mathcal{N}^+(\T)}\Big)d[v]_{\mathcal{N}_{\new}^+(\T_l)}d[v]_{\mathcal{N}_{\new}^+(\T_r)} |v_l-v_{\mathfrak{r}(\T)}|^{-1}d\pi_{v_{\mathfrak{r}(\T)}+(v_l-v_{\mathfrak{r}(\T)})^{\perp}}(v_r)dv_ldv_{\mathfrak{r}(\T)}\\
&=\pi\int H\Big([v]_{\mathcal{N}^+(\T)}\Big) |v_l-v_r|\overline{d\pi}_{\mathtt{S}_{v_l,v_r}}(v_{\mathfrak{r}(\T)})d[v]_{\mathcal{N}^+(\T_l)}d[v]_{\mathcal{N}^+(\T_r)}
\end{split}
\end{equation}
The reversed Carleman measure $\revcar{\T}$ can be further parametrized by (see Figure \ref{fig:collision_sphere-reverse})
\begin{equation}
\Big(\prod_{\n\in \mathcal{B}^+(\T)}d\upsigma_{\n}\Big)\prod_{\l\in \mathcal{L}(\T)}dv_{\l}.
\end{equation}

\begin{figure}[h]
    \centering
    \includegraphics[width=0.6\linewidth]{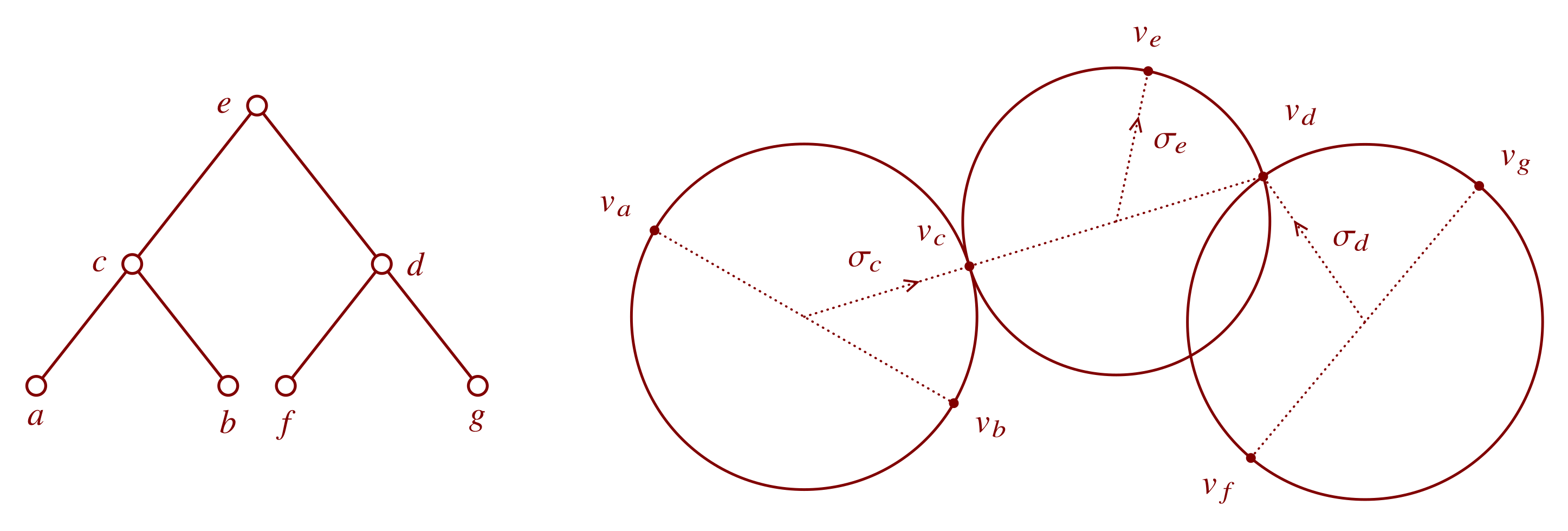}
    \caption{}
    \label{fig:collision_sphere-reverse}
\end{figure}
Therefore there exists a continuous function $J_{\mathcal{T}}$ of variables $[\upsigma]_{\mathcal{B}^+(\T)},[v]_{\mathcal{L}(\T)}$ such that
\begin{equation}
\int H\Big([v]_{\mathcal{N}^+(\T)}\Big)d[v]_{\mathcal{N}^+(\T)}=\int H\Big([v]_{\mathcal{N}^+(\T)}\Big)J_{\mathcal{T}}\Big([\upsigma]_{\mathcal{B}^+(\T)},[v]_{\mathcal{L}(\T)}\Big)\Big(\prod_{\n\in \mathcal{B}^+(\T)}d\upsigma_{\n}\Big)\prod_{\l\in \mathcal{L}(\T)}dv_{\l}
\end{equation}
\textit{Step 4: Convergence of the truncated Boltzmann integral.}

Notice that for each $\l\in \mathcal{L}(\T)$, the initial position $\x_\l$  is a function of $\{\upsigma_{\n}\}$, $\{v_\l\}$, and $\{s_{\n}\}$. We show for any $\{\upsigma_{\n}\}_{\n\in \mathcal{B}^+(\T)}$ and $\{s_{\n}\}_{s\in \mathcal{B}(\T)}$
\begin{equation}
\lim_{n\rightarrow +\infty}\int \varphi(v)J_{\mathcal{T}}\prod_{\l\in \mathcal{L}(\T)}\Big(f_n^0(\x_\l,v_\l)\mathds{1}_{|v_\l|\leq R}\Big)\prod_{\l\in \mathcal{L}(\T)}dv_\l=\int \varphi(v)J_{\mathcal{T}}\prod_{\l\in \mathcal{L}(\T)}\Big(f^0(\x_\l,v_\l)\mathds{1}_{|v_\l|\leq R}\Big)\prod_{\l\in \mathcal{L}(\T)}dv_\l.
\end{equation}
Decompose the truncated Boltzmann integral into integrations over small cubes
\begin{equation}
\begin{split}
&\sum_{[\mathtt{p}]_{\mathcal{L}(\T)}\in \big(\ell\mathds{Z}^3\big)^{|\mathcal{L}(\T)|}}\int \varphi(v)J_{\mathcal{T}}\Bigg[\prod_{\l\in \mathcal{L}(\T)}\Big(f_n^0\big(\x_{\l},v_{\l}\big)\mathds{1}_{|v_{\l}|\leq R}\mathds{1}_{v_{\l}\in Q_{\mathtt{p}_\l}}\Big)\Bigg]\prod_{\l\in \mathcal{L}(\T)}dv_{\l},
\end{split}
\end{equation}
We have
\begin{equation}
\begin{split}
&\lim_{\ell\rightarrow 0}\Bigg\{\sum_{[\mathtt{p}]_{\mathcal{L}(\T)}\in \big(\ell\mathds{Z}^3\big)^{|\mathcal{L}(\T)|}}\int \varphi(v)J_{\mathcal{T}}\Bigg[\prod_{\l\in \mathcal{L}(\T)}\Big(f_n^0\big(\x_{\l}^{\mathtt{p}},v_{\l}\big)\mathds{1}_{|v_{\l}|\leq R}\mathds{1}_{v_{\l}\in Q_{\mathtt{p}_\l}}\Big)\Bigg]\prod_{\l\in \mathcal{L}(\T)}dv_{\l}\\
&\hspace{5cm}-\int \varphi(v)J_{\mathcal{T}}\prod_{\l\in \mathcal{L}(\T)}\Big(f_n^0(\x_\l,v_\l)\mathds{1}_{|v_\l|\leq R}\Big)\prod_{\l\in \mathcal{L}(\T)}dv_\l\Bigg\}=0.
\end{split}
\end{equation}
with 
\begin{equation}
\x_l^{\mathtt{p}}\textup{defined as the initial position determined by the fixed $\{\upsigma_{\n}\}_{\n\in \mathcal{B}^+(\T)}$, $\{s_{\n}\}_{s\in \mathcal{B}(\T)}$, and }\{v_{\l}\}_{\l\in \mathcal{L}(\T)}=\{\mathtt{p}_{\l}\}_{\l\in \mathcal{L}(\T)}.
\end{equation}
So we only need to prove 
\begin{equation}
\lim_{n\rightarrow +\infty}\int \varphi(v)J_{\mathcal{T}}\Bigg[\prod_{\l\in \mathcal{L}(\T)}\Big(f_n^0\big(\xp_{\l},v_{\l}\big)\mathds{1}_{|v_{\l}|\leq R}\mathds{1}_{v_{\l}\in Q_{\mathtt{p}_\l}}\Big)\Bigg]\prod_{\l\in \mathcal{L}(\T)}dv_\l=\int \varphi(v)J_{\mathcal{T}}\Bigg[\prod_{\l\in \mathcal{L}(\T)}\Big(f^0\big(\xp_{\l},v_{\l}\big)\mathds{1}_{|v_{\l}|\leq R}\mathds{1}_{v_{\l}\in Q_{\mathtt{p}_\l}}\Big)\Bigg]\prod_{\l\in \mathcal{L}(\T)}dv_\l,
\end{equation}
which can be derived by the condition on the initial data.

\section{Kinetic Limit and Weak Stability of Hard Sphere Gas with Initial Data}

In this section, we prove the kinetic limit and weak stability of hard sphere gas with initial data $f^0\in \ldelta(\mathds{T}^3\times \mathds{R}^3)$, stated in Theorem \ref{th:kinetic_limit}. More precisely with $T_*$ determined by $\gamma$ and $\lVert f^0\rVert_{\mathbf{L}_{\gamma,+\infty}^1}$, the empirical measure $\pi_T^{\eps}$ in \eqref{eq:empirical_measure} converges in probability for any $T\in[0,T_*]$ to the unique mild solution $f(T)$ of Boltzmann equation with initial data $f^0$
\begin{equation}\label{eq:kinetic_limit_sec6}
\lim_{\eps\rightarrow 0}\mathds{P}_{\eps}\Big[\Big|\pi_T^{\eps}(\psi)-\int_{\mathds{T}^3\times \mathds{R}^3}f(T,x,v)\psi(x,v)dxdv\Big|>\delta\Big]=0,\quad \textup{for any test function $\psi$ and any $\delta>0$}.
\end{equation}
The weak stability follows directly combining \eqref{eq:kinetic_limit_sec6} the weak stability of Boltzmann equation in Theorem \ref{th:well-posedness} by a diagonal argument.

The existence of a unique solution $f(T)$ is given by Theorem \ref{th:well-posedness}, since $f^0\in L_{\gamma,+\infty}^1(\mathds{T}^3\times \mathds{R}^3)$ can be seen as a periodic-in-space function on $\mathds{R}^3\times \mathds{R}^3$. We choose to work with hard sphere gas on the torus $\mathds{T}^3$ because the mass of $f^0$ on $\mathds{T}^3$ is directly finite. In this section, we assume $\lVert f^0\rVert_{\mathbf{L}_{\gamma,+\infty}^1}\leq 1$ for convenience of exposition, while the proof for $\lVert f^0\rVert_{\mathbf{L}_{\gamma,+\infty}^1}> 1$ is exactly the same. 

\begin{definition}\label{def:graph_of_dynamical_interaction}
For arbitrary positive integer $m\geq 1$, and given $m$ hard sphere particles with initial configurations $\{z_{i,0}\}_{1\leq i\leq m}$, their dynamics follow the laws of motion \eqref{eq:hard_sphere_law_1}-\eqref{eq:hard_sphere_law_2} and we have a set of trajectories $\{\mathbf{z}_{i}([0,T])\}_{1\leq i\leq m}$. Given a set of trajectories, the \underline{graph of dynamical interactions} is built as the following: the set of vertices is $\{1,...,m\}$, and we add an edge $\{i,j\}$ if particles $i$ and $j$ collide during the time $[0,T]$. If there are multiple collisions between $i$ and $j$ then we add multiple egdes.

We write $(x_{i,t},v_{i,t})$ as the position and velocity of a particle at time $t$ with label $i$, and $z_{i,t}$ as the abbreviation of it. Given any positive integer $N$, we write $(X_{N,t},V_{N,t})$ as the positions and velocities of the $N$ particles at time $t$. We write $\mathbf{z}_i$ as the trajectory of the particle with label $i$ during a given time interval $[0,T]$, and $\mathbf{Z}_{N}$ as the trajectories of $N$ particles.
\end{definition}

\begin{figure}[h]
    \centering
    \includegraphics[width=0.8\linewidth]{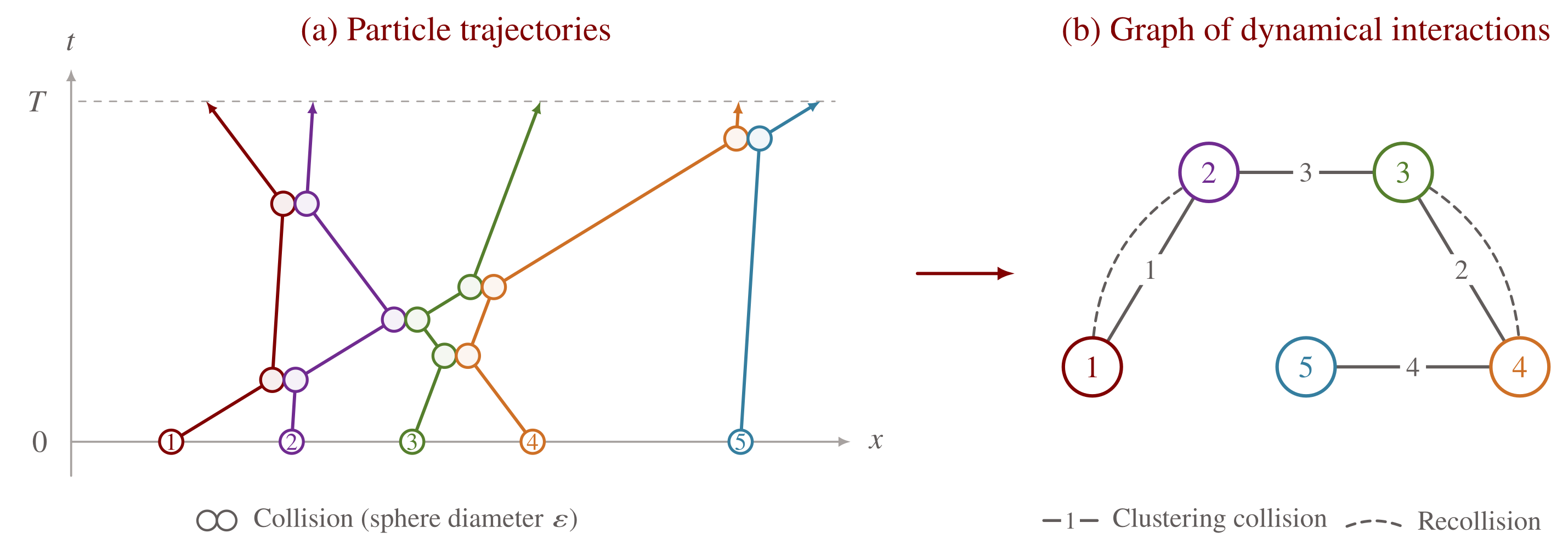}
    \caption{The left graph is for the physical trajectories of particles. The right graph is the corresponding graph of dynamical interactions.}
    \label{fig:dynamical_interactions}
\end{figure}

Given any parameters $(\eps>0,T>0)$ and any smooth and compactly supported test function $\psi(x,v)$, define the cumulant generating functional $\Lambda_{T}^{\eps}(e^{\psi})$ as
\begin{equation}\label{eq:cumulant_generating}
\Lambda_{T}^{\eps}(e^{\psi}):=\frac{1}{\mu_{\eps}}\log\mathds{E}_{\eps}\Bigg[\exp\Bigg(\sum_{1\leq i\leq \mathbf{N}}\psi(z_{i,T})\Bigg)\Bigg]
\end{equation}
The cumulant generating functional is specifically useful because the derivatives of it describe the behavior of the empirical measure $\pi_T^{\eps}$ in \eqref{eq:empirical_measure}, see \cite[Equation (51) and (54)]{bodineau2022cluster}
\begin{equation}\label{eq:derivative_of_CGF}
\p_u\Lambda_T^{\eps}(e^{u\psi})\big|_{u=0}=\mathds{E}_{\eps}\Big[\pi_T^{\eps}(\psi)\Big],\quad \frac{1}{\mu_{\eps}}\p_{u}^2\Lambda_T^{\eps}(e^{u\psi})\big|_{u=0}=\mathds{E}_{\eps}\Bigg[\Big(\pi_T^{\eps}(\psi)-\mathds{E}_{\eps}[\pi_T^{\eps}(\psi)]\Big)^2\Bigg].
\end{equation}
Furthermore, by taking higher order derivatives of $\Lambda_T^{\eps}(e^{u\psi})$ and their limiting behavior as $\eps\rightarrow 0$, we can recover all the cumulants of hard sphere gas, and this leads to the study of dynamical fluctuations and large deviations. See for example \cite{Bodineau_Laure_Gallagher_Simonella_2023,qi2024global}. For the derivation of Boltzmann equation \eqref{eq:Boltzmann}, it is enough to study the first and second order derivatives. Precisely we need to prove
\begin{equation}\label{eq:CGF_goal}
\underbrace{\lim_{\eps\rightarrow 0}\p_{u}\Lambda_T^{\eps}(e^{u\psi})\big|_{u=0}=\int f(T,x,v)\psi(x,v)dvdx}_{\textup{\blue{convergence of expection of $\pi_t^{\eps}$}}},\quad \underbrace{\lim_{\eps\rightarrow 0}\frac{1}{\mu_{\eps}}\p_{u}^2\Lambda_T^{\eps}(e^{u\psi})\big|_{u=0}=0}_{\textup{\blue{concentration of $\pi_t^{\eps}$}}}.
\end{equation}

Define the modified partition function $\mathcal{Z}_{T}^{\eps}(e^{\psi})$ as (recall $\mathcal{D}_N^{\eps}$ in \eqref{eq:exclusion_domain})
\begin{equation}\label{eq:modified_partition_function}
\mathcal{Z}_{T}^{\eps}(e^{\psi}):=1+\sum_{N\geq 1}\frac{\mu_{\eps}^N}{N!}\int_{\mathds{T}^{3N}\times \mathds{R}^{3N}}\mathds{1}_{\mathcal{D}_{N}^{\eps}}\prod_{1\leq i\leq N}\Bigg(e^{\psi}(z_{i,T})f^0(z_{i,0})\Bigg)dZ_{N,0},
\end{equation}
and the cumulant generating functional equals
\begin{equation}
\Lambda_T^{\eps}(e^{\psi})=\frac{1}{\mu_{\eps}}\log\Bigg(\frac{\mathcal{Z}_{T}^{\eps}(e^{\psi})}{\mathcal{Z}_T^{\eps}(1)}\Bigg)=\frac{1}{\mu_{\eps}}\log\mathcal{Z}_T^{\eps}(e^{\psi})-\frac{1}{\mu_{\eps}}\log\mathcal{Z}_T^{\eps}(1).
\end{equation}
Therefore to prove the kinetic limit in Theorem \ref{th:kinetic_limit}, it is sufficient to prove
\begin{equation}\label{eq:derivative_of_partition_function}
\lim_{\eps\rightarrow 0}\p_u\Big(\frac{1}{\mu_{\eps}}\log\mathcal{Z}_T^{\eps}(e^{\psi})\Big)\Big|_{u=0}=\int f(T,x,v)\psi(x,v)dvdx,\quad \lim_{\eps\rightarrow 0}\p_u^2\Big(\frac{1}{\mu_{\eps}}\mathcal{Z}_T^{\eps}(e^{\psi})\Big)\Big|_{u=0}=0.
\end{equation}
In Section \ref{sec:cgf_to_cluster} we prove a uniform-in-$\eps$ bound on $\mu_{\eps}^{-1}\log\mathcal{Z}_T^{\eps}(e^{\psi})$, expressing it as an integral over particle trajectories with collisions and overlaps. In Section \ref{sec:recollision} we show the uniform-in-$\eps$ convergence of $\mu_{\eps}^{-1}\log\mathcal{Z}_T^{\eps}(e^{\psi})$ to a limiting functional $\Lambda_T^{0}(e^{\psi})$. In Section \ref{sec:derivation} we show 
\begin{equation}
\underbrace{\lim_{\eps\rightarrow 0}\p_{u}\Lambda_T^{0}(e^{u\psi})\big|_{u=0}=\int f(T,x,v)\psi(x,v)dvdx}_{\textup{\blue{convergence of expection of $\pi_t^{\eps}$}}},\quad \underbrace{\lim_{\eps\rightarrow 0}\frac{1}{\mu_{\eps}}\p_{u}^2\Lambda_T^{0}(e^{u\psi})\big|_{u=0}=0}_{\textup{\blue{concentration of $\pi_t^{\eps}$}}},
\end{equation}
which concludes the proof of kinetic limit in Theorem \ref{th:kinetic_limit}. The weak stability follows by a diagonal argument applied to Theorem \ref{th:well-posedness}.

\subsection{Cumulant Generating Functional as an Integral over Clusters}\label{sec:cgf_to_cluster}

Define $\mathds{D}\big([0,T];\mathds{T}^{3}\times \mathds{R}^3\big)$ as the space of particle trajectories, which consists of trajectories $(x_{t},v_t)_{0\leq t\leq T}$ that are piece-wise linear continuous in position $x\in\mathds{T}^3$ and piece-wise constant in velocity $v\in\mathds{R}^3$. A cluster $\lambda$ is a set of particle trajectories , or more precisely a finite subset of $\mathds{D}\big([0,T];\mathds{T}^3\times \mathds{R}^3\big)$, where initially particles do not overlap, then follow the law of hard sphere gas in \eqref{eq:hard_sphere_law_1}-\eqref{eq:hard_sphere_law_2}, and do not collide with particles in other clusters. The elements in $\lambda$ are indexed by $\{1,...,|\lambda|\}$
\begin{equation}
\lambda=\{(x_{1,t},v_{1,t})_{0\leq t\leq T},...,(x_{|\lambda|,t},v_{|\lambda|,t})_{0\leq t\leq T}\},\quad \textup{with each $(x_{i,t},v_{i,t})_{0\leq t\leq T}\in \mathds{D}\big([0,T];\mathds{T}^3\times \mathds{R}^3\big)$}.
\end{equation}
Given arbitrary parameters $(\eps>0,T>0)$, define the indicator function $\indicator$ as 
\begin{equation}
\indicator:=
\left\{
\begin{array}{cc}
 1    & \textup{the graph of dynamical interactions in Definition \ref{def:graph_of_dynamical_interaction} of $\lambda$ is connected} \\
 0    & \textup{otherwise}
\end{array}
\right.
\end{equation}

Denote $\boldsymbol{\mathcal{G}}_k$ as the set of graphs on the vertices $\{1,...,k\}$, and $\mathcal{C}_k$ as the set of connected graphs on the vertices $\{1,...,k\}$. Define function $\varphi_{\eps}(\lambda_1,...,\lambda_k)$ as
\begin{equation}
\varphi_{\eps}(\lambda_1,...,\lambda_k):=
\left\{
\begin{array}{lr}
 1    & \textup{ if $k=1$} \\
 \sum_{\mathcal{G}\in \mathcal{C}_k}\prod_{\{j,j'\}\in E(\mathcal{G})}\big(-\mathds{1}_{\lambda_j\sim_{\eps}\lambda_{j'}}\big)    & \textup{ if $k>1$}
\end{array}
\right.
\end{equation}
with the indicator function $\mathds{1}_{\lambda_j\sim_{\eps}\lambda_{j'}}$ defined as
\begin{equation}
\left\{
\begin{array}{lr}
1     & \textup{at some time $[0,T]$, one particle in $\lambda_j$ has distance $\leq \eps$ to a particle in $\lambda_{j'}$} \\
0     & \textup{otherwise}
\end{array}
\right.
\end{equation}
Given parameters $\gamma>0$, we choose a small enough constant $\chard$, and when we consider smooth test function $\psi:\mathds{R}^3\times \mathds{R}^3\mapsto\mathds{R}$ we always assume
\begin{equation}\label{eq:test_function_bound}
\lVert \psi\rVert_{L_{x,v}^{\infty}}\leq \chard.
\end{equation}
Given any parameters $(\eps>0,T>0)$ and any test function $\psi(x,v)$, define cluster measure $\vepst^{(\psi)}(\lambda)$ as
\begin{equation}\label{eq:cluster_measure_def}
\vepst^{(\psi)}(\lambda):=\frac{\mu_{\eps}^{|\lambda|}}{|\lambda|!}\indicator\prod_{i\leq |\lambda|}\Big(\exp\big(\psi(z_{i,T})\big)f^0\big(z_{i,0}\big)\Big)dZ_{|\lambda|,0}
\end{equation}

\begin{lemma}\label{lem:modified_partition_function}
For any smooth function $\psi(x,v)$ satisfying \eqref{eq:test_function_bound}, the modified partition function $\mathcal{Z}_T^{\eps}(e^{\psi})$ in \eqref{eq:modified_partition_function} satisfies
\begin{equation}
\mathcal{Z}_T^{\eps}(e^{\psi})=\exp\Bigg(\sum_{m\geq 1}\frac{1}{m!}\int \varphi_{\eps}(\lambda_1,...,\lambda_m)d\vepst^{(\psi)}(\lambda_1)...d\vepst^{(\psi)}(\lambda_m)\Bigg).
\end{equation}
\end{lemma}

\begin{proof}

The indicator function $\mathds{1}_{\mathcal{D}_{N}^{\eps}}$ can be decomposed as the following, with $\mathcal{P}_{\mathbf{Z}_N}^k$ as the set of partitions of $\{\mathbf{z}_i\}_{1\leq i\leq N}$
\begin{equation}\label{eq:decomposition_of_exclusion}
\mathds{1}_{\mathcal{D}_N^{\eps}}=\sum_{k\leq N}\sum_{\{\lambda_1,...,\lambda_k\}\in\mathcal{P}_{\mathbf{Z}_N}^k}\Bigg(\prod_{1\leq j<j'\leq k}\mathds{1}_{\lambda_j\not\sim_{\eps}\lambda_{j'}}\Bigg)\prod_{1\leq j\leq k}\Bigg(\indicatork{j}\Bigg)
\end{equation}
Denote $F^{(\psi)}(\lambda)$ as
\begin{equation}
F^{(\psi)}(\lambda):=\prod_{1\leq i\leq |\lambda|}e^{\psi}(z_{i,T})f^0(z_{i,0})
\end{equation}
Apply the decomposition of $\mathcal{D}_N^{\eps}$ in \eqref{eq:decomposition_of_exclusion} to the expression of the modified partition functions in \eqref{eq:modified_partition_function}
\begin{equation}
\mathcal{Z}_{T}^{\eps}(e^{\psi})=1+\sum_{N\geq 1}\frac{\mu_{\eps}^N}{N!}\sum_{k\geq 1}\sum_{\{\lambda_1,...,\lambda_k\}\in \mathcal{P}_{\mathbf{Z}_N}^{k}}\int_{\mathds{T}^{3N}\times \mathds{R}^{3N}}\Bigg(\prod_{1\leq j<j'\leq k}\mathds{1}_{\lambda_j\not\sim_{\eps}\lambda_{j'}}\Bigg)\prod_{1\leq j\leq k}\Bigg(\indicatork{j}F^{(\psi)}(\lambda_j)\Bigg)dZ_{N,0}
\end{equation}
Given the total number $N$ of particles and the total number $k$ of clusters, we can sum over all possible sizes $\{n_1,...,n_k\}$ of clusters with $n_1+...+n_k=N$, and $\mathcal{Z}_{T}^{\eps}(e^{\psi})$ further equals
\begin{equation}
\begin{split}
&1+\sum_{N\geq 1}\frac{1}{N!}\sum_{k\geq 1}\frac{1}{k!}\sum_{\substack{n_1,...,n_k\\ n_1+...+n_k=N}}\frac{N!}{n_1!...n_k!}\mu_{\eps}^{n_1+...+n_k}\int_{\mathds{T}^{3N}\times \mathds{R}^{3N}}\Bigg(\prod_{1\leq j<j'\leq k}\mathds{1}_{\lambda_j\not\sim_{\eps}\lambda_{j'}}\Bigg)\prod_{1\leq j\leq k}\Bigg(\indicatork{j}F^{(\psi)}(\lambda_j)dZ_{|\lambda_j|,0}\Bigg)\\
&=1+\sum_{k\geq 1}\frac{1}{k!}\int\Bigg(\prod_{1\leq j<j'\leq k}\mathds{1}_{\lambda_j\not\sim_{\eps}\lambda_{j'}}\Bigg)\prod_{1\leq j\leq k}\Bigg(\sum_{n_j\geq 1}\frac{\mu_{\eps}^{n_j}}{n_j!}\indicatork{j}F^{(\psi)}(\lambda_j)dZ_{n_j,0}\Bigg)
\end{split}
\end{equation}
Using the notion of cluster measure in \eqref{eq:cluster_measure_def}, the modified partition function $\mathcal{Z}_{T}^{\eps}(e^{\psi})$ in \eqref{eq:modified_partition_function} can be expressed by integration over cluster measures
\begin{equation}\label{eq:partition_formula}
1+\sum_{k\geq 1}\frac{1}{k!}\int \Bigg(\prod_{1\leq j<j'\leq k}\mathds{1}_{\lambda_j\not\sim_{\eps}\lambda_{j'}}\Bigg)d\vepst^{(\psi)}(\lambda_1)...d\vepst^{(\psi)}(\lambda_k)
\end{equation}
Now we further decompose 
\begin{equation}\label{eq:overlap_decomposition}
\prod_{1\leq j<j'\leq k}\mathds{1}_{\lambda_{j}\not\sim_{\eps}\lambda_{j'}}=\prod_{1\leq j<j'\leq k}\Big(1-\mathds{1}_{\lambda_{j}\sim_{\eps}\lambda_{j'}}\Big)=\sum_{\mathcal{G}\in \boldsymbol{\mathcal{G}}_{k}}\prod_{\{j,j'\}\in E(\mathcal{G})}\Big(-\mathds{1}_{\lambda_{j}\sim_{\eps}\lambda_{j'}}\Big)
\end{equation}
Apply \eqref{eq:overlap_decomposition} to \eqref{eq:partition_formula}, the modified partition function $\mathcal{Z}_{T}^{\eps}(e^{\psi})$ equals
\begin{equation}
1+\sum_{k\geq 1}\frac{1}{k!}\int \sum_{\mathcal{G}\in \boldsymbol{\mathcal{G}}_{k}}\prod_{\{j,j'\}\in E(\mathcal{G})}\Big(-\mathds{1}_{\lambda_{j}\sim_{\eps}\lambda_{j'}}\Big)d\vepst^{(\psi)}(\lambda_1)...d\vepst^{(\psi)}(\lambda_k)
\end{equation}
Given the total number $k$ of clusters, we can sum over the total number $n$ of connected components of $\mathcal{G}$, and sum over all the possible sizes $\{m_1,...,m_n\}$ of the connected components
\begin{equation}
\begin{split}
&1+\sum_{k\geq 1}\frac{1}{k!}\sum_{1\leq n\leq k}\frac{1}{n!}\sum_{\substack{\{m_1,...,m_n\}\\ m_1+...+m_n=k}}\frac{k!}{m_1!...m_n!}\int\prod_{1\leq \ell\leq n}\Bigg(\sum_{\mathcal{C}\in \mathbf{\mathcal{C}}_{m_\ell}}\prod_{\{j,j'\}\in E(\mathcal{C})}\big(-\mathds{1}_{\lambda_{\ell,j}\sim_{\eps}\lambda_{\ell,j'}}\big)d\vepst^{(\psi)}(\lambda_{\ell,1})...d\vepst^{(\psi)}(\lambda_{\ell,m_{\ell}})\Bigg)\\
&=1+\sum_{n\geq 1}\frac{1}{n!}\Bigg(\sum_{m\geq 1}\frac{1}{m!}\int \varphi_{\eps}(\lambda_1,...,\lambda_m)d\vepst^{(\psi)}(\lambda_1)...d\vepst^{(\psi)}(\lambda_{m})\Bigg)^n\\
&=\exp\Bigg(\sum_{m\geq 1}\frac{1}{m!}\int \varphi_{\eps}(\lambda_1,...,\lambda_m)d\vepst^{(\psi)}(\lambda_1)...d\vepst^{(\psi)}(\lambda_{m})\Bigg).
\end{split}
\end{equation}
This concludes the proof of Lemma \ref{lem:modified_partition_function}.
\end{proof}

Given any function $f:\mathds{T}^3\times \mathds{R}^3\mapsto \mathds{R}$, we define $[f]:\mathds{R}^3\mapsto\mathds{R}_{\geq 0}$ as
\begin{equation}
[f](v):=\sup_{x\in\mathds{T}^3}|f(x,v)|.
\end{equation}
\begin{lemma}\label{lem:cluster_bound}
For any initial data $f^0\in \mathbf{L}_{\gamma,+\infty}^1(\mathds{T}^3\times \mathds{R}^3)$, any smooth function $\psi$ satisfying \eqref{eq:test_function_bound}, any small enough $T$ and $\eps$ 
\begin{equation}\label{eq:cluster_bound}
\int d\vepst^{(\psi)}(\lambda_1) ...d\vepst^{(\psi)}(\lambda_k)\ |\varphi_{\eps}(\lambda_1,...,\lambda_k)|\leq C_{\gamma}^k\mu_{\eps}k!(T+\eps)^{k-1}.
\end{equation}
\end{lemma}

\begin{proof}

We first prove the case of $k=1$ for Equation \eqref{eq:cluster_bound} (we add factor $e^{\frac{10}{\gamma}|\lambda|}$ because it will be useful in the case of $k>1$)
\begin{equation}\label{eq:k=1_cluster_bound}
\int e^{\frac{10}{\gamma}|\lambda|}d\vepst(\lambda)\leq C_{\gamma}\mu_{\eps}
\end{equation}
Apply the definition of the cluster measure in \eqref{eq:cluster_measure_def} to \eqref{eq:k=1_cluster_bound}, and use the fact that $\lVert \psi\rVert_{L_{x,v}^{\infty}}$ is bounded by a constant determined by $\gamma$
\begin{equation}
\begin{split}
\int e^{\frac{10}{\gamma}|\lambda|}d\vepst(\lambda)&\leq \sum_{n\geq 1}\frac{\mu_{\eps}^n}{n!}e^{\frac{10}{\gamma}n}\int_{\mathds{T}^{3n}\times \mathds{R}^{3n}}\indicator \prod_{1\leq i\leq n}\Big(e^{\psi}(z_{i,T})f^0(z_{i,0})\Big)dZ_{n,0}\\
&\leq \sum_{n\geq 1}\frac{\mu_{\eps}^n}{n!}e^{\frac{10}{\gamma}n}\int_{\mathds{T}^{3n}\times \mathds{R}^{3n}}\indicator \prod_{1\leq i\leq n}\Big(e^{[\psi]}(v_{i,T})[f^0](v_{i,0})\Big)dZ_{n,0}\\
&\leq \sum_{n\geq 1}\frac{\mu_{\eps}^n}{n!}e^{\frac{10}{\gamma}n}\int_{\mathds{T}^{3n}\times \mathds{R}^{3n}}\indicator C_{\gamma}^n\prod_{1\leq i\leq n}[f^0](v_{i,0})dZ_{n,0}
\end{split}
\end{equation}

We can fix all the initial velocities  $V_{n,0}$ and integrate over the initial positions $X_{n,0}$
\begin{equation}\label{eq:cluster_reduction_1}
\int e^{\frac{10}{\gamma}|\lambda|}d\vepst(\lambda)\leq \sum_{n\geq 1}\frac{\mu_{\eps}^n}{n!}C_{\gamma}^n\int \prod_{1\leq i\leq n}[f^0](v_{i,0})\Big(\int_{\mathds{T}^{3n}}\indicator dX_{n,0} \Big) dV_{n,0}
\end{equation}
\cite[Equation (43)]{bodineau2022cluster} provides the following inequality. It estimates, given fixed initial velocities $V_{n,0}$, the size of initial positions $X_{n,0}$ making the graph of dynamical interactions of $n$ particles becomes connected (with $\mathcal{E}_{\lambda}$ as the total kinetic energy of cluster $\lambda$)
\begin{equation}\label{eq:integrate_relative_positions}
\int_{\mathds{T}^{3n}}\indicator dX_{n,0}\leq \Big(\frac{C}{\mu_{\eps}}\Big)^{n-1}\Big(\mathcal{E}_{\lambda}+n\Big)^{n-1}n^{n-1}\frac{T^{n-1}}{(n-1)!}
\end{equation}
Consequently we apply \eqref{eq:integrate_relative_positions} to \eqref{eq:cluster_reduction_1}
\begin{equation}\label{eq:cluster_reduction_2}
\begin{split}
&\sum_{n\geq 1}\frac{\mu_{\eps}^{n}}{n!}C_{\gamma}^n\int \prod_{1\leq i\leq n}[f^0](v_{i,0})\Big(\frac{C}{\mu_{\eps}}\Big)^{n-1}\Big(\mathcal{E}_{\lambda}+n\Big)^{n-1}n^{n-1}\frac{T^{n-1}}{(n-1)!}dV_{n,0}\\
&\leq \sum_{n\geq 1}\frac{\mu_{\eps}}{n!}C_{\gamma}\Big(C_{\gamma}T\Big)^{n-1}\int \prod_{1\leq i\leq n}[f^0](v_{i,0})\Big(\mathcal{E}_{n}+n\Big)^{n-1}dV_{n,0}
\end{split}
\end{equation}
According to \cite[Equation (44)]{bodineau2022cluster}
\begin{equation}\label{eq:cluster_reduction_3}
\begin{split}
\int \prod_{1\leq i\leq n}[f^0](v_{i,0})\Big(\mathcal{E}_{n}+n\Big)^{n-1}dV_{n,0}&\leq \int \prod_{1\leq i\leq n}\Big([f^0](v_{i,0})e^{\gamma|v_{i,0}|^2}\Big)\Bigg(e^{-\frac{\gamma}{2}\mathcal{E}_{n}}\Big(\mathcal{E}_{n}+n\Big)^{n-1}\Bigg)e^{-\frac{\gamma}{2}\mathcal{E}_{n}}dV_{n,0}\\
&\leq C_{\gamma}^nn^n\int \prod_{1\leq i\leq n}\Big([f^0](v_{i,0})e^{\gamma|v_{i,0}|^2}\Big)e^{-\frac{\gamma}{2}\mathcal{E}_{n}}dV_{n,0}\leq \Big(C_{\gamma}\lVert f^0\rVert_{\mathbf{L}_{\gamma,+\infty}^1}\Big)^nn^n
\end{split}
\end{equation}
Apply \eqref{eq:cluster_reduction_3} to \eqref{eq:cluster_reduction_2}
\begin{equation}
\int e^{\frac{10}{\gamma}|\lambda|}d\vepst^{(\psi)}(\lambda)\leq \sum_{n\geq 1}\frac{\mu_{\eps}}{n!}C_{\gamma}\Big(C_{\gamma}T\Big)^{n-1}\Big(C_{\gamma}\lVert f^0\rVert_{\mathbf{L}_{\gamma,+\infty}^1}\Big)^nn^n.
\end{equation}
Choosing $T$ small enough, we can prove the upper bound for a single cluster with $k=1$. The proof for $k>1$ can be performed in the same idea: we first fix all the initial velocities, and then integrate the relative positions of clusters. For details, see the 'case $k>1$' in the proof of \cite[Proposition II.6]{bodineau2022cluster}.
\end{proof}

Applying Lemma \ref{lem:cluster_bound} to Lemma \ref{lem:modified_partition_function}, we have the uniform-in-$\eps$ bound
\begin{equation}\label{CGF_expanded}
\frac{1}{\mu_{\eps}}\log\mathcal{Z}_T^{\eps}(e^{\psi})=\frac{1}{\mu_{\eps}}\sum_{k\geq 1}\frac{1}{k!}\int \varphi_{\eps}(\lambda_1,...,\lambda_k)d\vepst^{(\psi)}(\lambda_1)...d\vepst^{(\psi)}(\lambda_k),\quad |\frac{1}{\mu_{\eps}}\log\mathcal{Z}_T^{\eps}(e^{\psi})|\leq C_{\gamma}.
\end{equation}

\subsection{Elimination of Recollision}\label{sec:recollision}
Given arbitrary parameters $(\eps>0,T>0)$, define the indicator function $\indicatormink{}$ as 
\begin{equation}\label{eq:minimal_cluster}
\indicatormink{}:=
\left\{
\begin{array}{cc}
 1   & \textup{the graph of dynamical interactions of $\lambda$ in Definition \ref{def:graph_of_dynamical_interaction} is connected and no cycle} \\
 0    & \textup{otherwise}
\end{array}
\right.
\end{equation}
We denote the measure of minimal clusters $d\vepst^{(\psi),\min}$
\begin{equation}
\vepst^{(\psi),\min}(\lambda):=\frac{\mu_{\eps}^{|\lambda|}}{|\lambda|!}\indicatormink{}\prod_{i\leq |\lambda|}\Big(\exp\big(\psi(z_{i,T})\big)f^0\big(z_{i,0}\big)\Big)dZ_{|\lambda|,0} 
\end{equation}
Given a sequence of clusters $\{\lambda_{i}\}_{1\leq i\leq k}$, we can build a graph of dynamical interactions for clusters: the set of vertices is $\{1,...,k\}$. We add an edge $\{i,j\}$ if during the time $[0,T]$, there exists a particle $q$ in $\lambda_i$ and a particle $\Bar{q}$ in $\lambda_j$ such that
\begin{equation}
\min_{t\in[0,T]}|x_{q,t}-x_{\Bar{q},t}|\leq \eps.
\end{equation}
If there are multiple times that particles in $\lambda_i$ and $\lambda_j$ overlap with each other, then we add multiple edges.

Define function $\varphi_{\eps}^{\min}$ as
\begin{equation}\label{eq:minimal_overlap}
\varphi_{\eps}^{\min}(\lambda_1,...,\lambda_k):=
\left\{
\begin{array}{lr}
(-1)^{k-1}     & \textup{the graph of dynamical interactions for $\{\lambda_i\}_{1\leq i\leq k}$ is a tree} \\
 0    & \textup{otherwise}
\end{array}
\right.
\end{equation}

\begin{lemma}\label{lem:eliminate_recollision}
For any smooth test function $\psi$ satisfying \eqref{eq:test_function_bound}, we have
\begin{equation}\label{eq:eliminate_recollision}
\lim_{\eps\rightarrow 0}\frac{1}{\mu_{\eps}}\Bigg|\int d\vepst^{(\psi)}(\lambda_1)...d\vepst^{(\psi)}(\lambda_k)\varphi_{\eps}(\lambda_1,...,\lambda_k)-\int d\vepst^{(\psi),\min}(\lambda_1)...d\vepst^{(\psi),\min}(\lambda_k)\varphi_{\eps}^{\min}(\lambda_1,...,\lambda_k)\Bigg|=0
\end{equation}
\end{lemma}

\begin{proof}
For any $n\geq 1$, define $\vepst^{(\psi),(n)}$ as 
\begin{equation}
\vepst^{(\psi),(n)}(\lambda):=\mathds{1}_{|\lambda|=n}\frac{\mu_{\eps}^{n}}{n!}\indicator\prod_{i\leq n}\Big(\exp\big(\psi(z_{i,T})\big)f^0\big(z_{i,0}\big)\Big)dZ_{n,0}.
\end{equation}
We also define $d\vepst^{(\psi),\min,(n)}$
\begin{equation}
\vepst^{(\psi),\min,(n)}(\lambda):=\mathds{1}_{|\lambda|=n}\frac{\mu_{\eps}^{n}}{n!}\indicatormink{}\prod_{i\leq n}\Big(\exp\big(\psi(z_{i,T})\big)f^0\big(z_{i,0}\big)\Big)dZ_{n,0}.
\end{equation}
To prove \eqref{eq:eliminate_recollision}, it is enough to show for any $k\geq 1$ and any $\{n_{i}\}_{1\leq i\leq k}$ 
\begin{equation}
\lim_{\eps\rightarrow 0}\frac{1}{\mu_{\eps}}\Bigg|\int d\vepst^{(\psi),(n_1)}(\lambda_1)...d\vepst^{(\psi),(n_k)}(\lambda_k)\varphi_{\eps}(\lambda_1,...,\lambda_k)-\int d\vepst^{(\psi),\min,(n_1)}(\lambda_1)...d\vepst^{(\psi),\min,(n_k)}(\lambda_k)\varphi_{\eps}^{\min}(\lambda_1,...,\lambda_k)\Bigg|=0
\end{equation}

The proof of Lemma \ref{lem:cluster_bound} shows the following summation absolutely converges
\begin{equation}
\vepst^{(\psi)}(\lambda)=\sum_{n\geq 1}\vepst^{(\psi),(n)}(\lambda).
\end{equation}
The geometric estimate in \cite[Proposition 5.3 and 5.4]{bodineau2023long} shows for arbitrary initial velocities $V_{n,0}^*$, the recollisions can be eliminated
\begin{equation}
\lim_{\eps\rightarrow 0}\frac{1}{\mu_{\eps}}\int \Big|d\vepst^{(\psi),(n)}(\lambda)-d\vepst^{(\psi),\min,n}(\lambda)\Big|\delta_{V_{n,0}^*}(V_{N,0})dX_{n,0}=0.
\end{equation}
Dominated convergence theorem imply
\begin{equation}
\lim_{\eps\rightarrow 0}\frac{1}{\mu_{\eps}}\int \Big|d\vepst^{(\psi),(n)}(\lambda)-d\vepst^{(\psi),\min,(n)}(\lambda)\Big|=0,
\end{equation}
and proves 
\begin{equation}
\lim_{\eps\rightarrow 0}\frac{1}{\mu_{\eps}}\Bigg|\int d\vepst^{(\psi),(n_1)}(\lambda_1)...d\vepst^{(\psi),(n_k)}(\lambda_k)\varphi_{\eps}(\lambda_1,...,\lambda_k)-\int d\vepst^{(\psi),\min,(n_1)}(\lambda_1)...d\vepst^{(\psi),\min,(n_k)}(\lambda_k)\varphi_{\eps}(\lambda_1,...,\lambda_k)\Bigg|=0.
\end{equation}
So now we only need to prove
\begin{equation}
\lim_{\eps\rightarrow 0}\frac{1}{\mu_{\eps}}\Bigg|\int d\vepst^{(\psi),(n_1),\min}(\lambda_1)...d\vepst^{(\psi),(n_k),\min}(\lambda_k)\varphi_{\eps}(\lambda_1,...,\lambda_k)-\int d\vepst^{(\psi),\min,(n_1)}(\lambda_1)...d\vepst^{(\psi),\min,(n_k)}(\lambda_k)\varphi_{\eps}^{\min}(\lambda_1,...,\lambda_k)\Bigg|=0
\end{equation}
Fix all the positions and velocities of $\{\lambda_1,...,\lambda_k\}$, while leaving the centers of masses for clusters as free. This allow us to move clusters rigidly, because overlaps do not change the geometry of a cluster. Consider a graph $\mathcal{G}\in\boldsymbol{\mathcal{G}}_k$ with a cycle. We pick a sub-graph $\mathcal{C}$ of $\mathcal{G}$ as a tree, and denote $\{\lambda_{p_i},\lambda_{q_i}\}_{1\leq i\leq k-1}$ as the clusters involved in each overlap contained by $\mathcal{C}$. 

Starting from $i=1$, we move the two clusters $(\lambda_{p_i},\lambda_{q_i})$ rigidly to create an overlap between them. \cite[Equation (46)-(47)]{bodineau2022cluster}  estimates the bound for such an integral. There exists an index $j\leq k$ such that, we need to move a cluster $\lambda^*$ rigidly such that it has overlaps with two different clusters. In this case, the admissible value of the center of mass of $\lambda^*$ must be within the intersection of two thin tubes with radius $\eps$. Each tube's direction is given by the difference between two velocities of two particles, see \cite[Equation (41)]{bodineau2022cluster}.

To gain smallness from here, we can define a truncated domain: $\mathds{1}_{\mathcal{B}}(\lambda_1,...,\lambda_m)=1$  if and only if at any time $t\in[0,T]$, all the directions of velocities and relative velocities are separated from each other by $|\log\eps|^{-1}$.
As $\eps\rightarrow 0$, we have $|\mathcal{B}^c|\rightarrow 0$. And within $\mathcal{B}$, we can gain smallness from the overlap forming a loop, because all the tubes containing the admissible relative positions have directions separated by $|\log \eps|^{-1}$, therefore the intersection of two tubes contribute a small factor. This concludes the proof.
\end{proof}

Apply Lemma \ref{lem:eliminate_recollision} to \eqref{CGF_expanded}
\begin{equation}\label{eq:CGF_minimal}
\frac{1}{\mu_{\eps}}\log\mathcal{Z}_T^{\eps}(e^{\psi})=\frac{1}{\mu_{\eps}}\sum_{k\geq 1}\frac{1}{k!}\int \varphi_{\eps}^{\min}(\lambda_1,...,\lambda_k)d\vepst^{(\psi),\min}(\lambda_1)...d\vepst^{(\psi),\min}(\lambda_k)
\end{equation}

\subsection{Derivation of Boltzmann Equation by Reducing Clusters into Boltzmann Integrals}\label{sec:derivation}

Given a sequence of minimal clusters $\{\lambda_1,...,\lambda_k\}$ satisfying $\varphi_{\eps}^{\min}(\lambda_1,...,\lambda_k)\not=0$, we label the particles by $\{1,...,n\}$ with $n=|\lambda_1|+...+|\lambda_k|$. The collisions within clusters and the overlaps between clusters can be recorded by a \textup{decorated ordered tree} $\mathcal{T}^{\prec,\pm}$ (different from decorated binary tree in Definition \ref{def:dec_binary_tree}!) with
\begin{itemize}
    \item the set of vertices $V(\mathcal{T}^{\prec,\pm})=\{1,...,n\}$.
    \item the set of edge $E(\mathcal{T}^{\prec,\pm})$ contains $\{i,j\}$ if and only if the particles labeled by $i$ and $j$ have a collision or overlap with each other. Each edge $e\in E(T_{\prec}^{\pm})$ is decorated as either collision or overlap, and we associate $d_e\in\{+1,-1\}$ with edge $e$. If $e$ is a collision, then $d_e=1$. If $e$ is an overlap, then $d_{e}=-1$.
    \item the edges of $\mathcal{T}^{\prec,\pm}$ is ordered as $\{e_1,...,e_{n-1}\}$.
\end{itemize}
Apply \cite[Remark III.2]{bodineau2022cluster} with $\boldsymbol{\mathcal{T}_n^{\prec,\pm}}$ as the set of all decorated ordered trees with $n$ vertices
\begin{equation}
\begin{split}
&\frac{1}{\mu_{\eps}}\sum_{k\geq 1}\frac{1}{k!}\int d\vepst^{(\psi),\min}(\lambda_1)...d\vepst^{(\psi),\min}(\lambda_k)\varphi_{\eps}^{\min}(\lambda_1,...,\lambda_k)\\
&=\sum_{n\geq 1}\sum_{\mathcal{T}^{\prec,\pm}\in\boldsymbol{\mathcal{T}_n^{\prec,\pm}}}\Big(\prod_{e\in E(\mathcal{T}^{\prec,\pm})}d_{e}\Big)\underbrace{\frac{\mu_{\eps}^{n-1}}{n!}\int \mathds{1}_{\mathbf{Z}_n\textup{ compatible with $\mathcal{T}^{\prec,\pm}$ as $\eps$-trajectories}}\prod_{1\leq i\leq n}\Big(e^{\psi}(z_{i,T})f^0(z_{i,0})\Big)dZ_{N,0}}_{\textup{\blue{we will transform it into a form more similar to Boltzmann integrals}}}
\end{split}
\end{equation}
Given any decorated ordered tree $\mathcal{T}^{\prec,\pm}$, as $\eps\rightarrow 0$ the initial configurations $Z_{N,0}$ of trajectories compatible with $\mathcal{T}^{\prec,\pm}$ concentrate within a more and more singular set. However by re-parametrizing trajectories, we can identify a stable limit. Given any decorated ordered tree $\mathcal{T}^{\prec,\pm}$, the trajectories of particles with diameter $\eps$ compatible with $\mathcal{T}^{\prec,\pm}$ can be parametrized by
\begin{itemize}
    \item The center of mass $y$ of $n$ particles at time $t=0$
    \item Velocities $V_{n,0}$ of $n$ particles at time $n=0$
    \item The collision/overlap angles $\Omega_{n-1}:=\{\omega_{e}\}_{e\in E(\mathcal{T}^{\prec,\pm})}$ of each collision or overlap
    \item The collision/overlap times $S_{n-1}:=\{s_{e}\}_{e\in E(\mathcal{T}^{\prec,\pm})}$ of each collision or overlap
\end{itemize}
Given the parameters and going forward in time, we can add a new particle to an already existing particle at time $s_{e}$ with collision/overlap angle $\omega_{e}$. If $d_e=1$ then the velocities change, otherwise not, see Figure \ref{fig:construct_trajectory}.
\begin{figure}[h]
    \centering
    \includegraphics[width=0.9\linewidth]{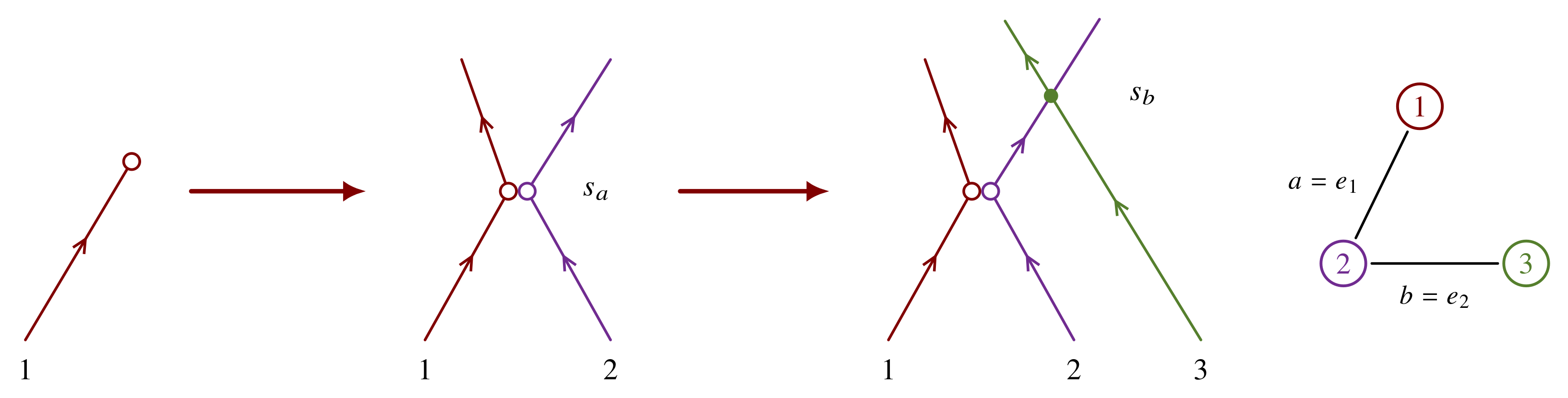}
    \caption{Construct a trajectory given $\mathcal{T}^{\prec,\pm}$ and the parameters.}
    \label{fig:construct_trajectory}
\end{figure}

Apply \cite[Equation (65)]{bodineau2022cluster}
\begin{equation}\label{eq:Z_to_omega}
\begin{split}
&\frac{\mu_{\eps}^{n-1}}{n!}\int \mathds{1}_{\mathbf{Z}_n\textup{ compatible with $\mathcal{T}^{\prec,\pm}$ as $\eps$-trajectories}}\prod_{1\leq i\leq n}\Big(e^{\psi}(z_{i,T})f^0(z_{i,0})\Big)dZ_{N,0}\\
&=\frac{1}{n!}\int \mathds{1}_{\mathbf{Z}_n\textup{ has no extra collision/overlap}}\prod_{e=\{i,j\}\in E(\mathcal{T}^{\prec,\pm})}\big((v_{i,s_{e}^-}-v_{j,s_{e}^-})\cdot\omega_e\big)_+\prod_{1\leq i\leq n}\Big(e^{\psi}(z_{i,T})f^0(z_{i,0})\Big) dS_{n-1}d\Omega_{n-1}dV_ndy
\end{split}
\end{equation}
The re-parametrization in \eqref{eq:Z_to_omega} allows us to express the partition function in \eqref{eq:CGF_minimal} as
\begin{equation}
\begin{split}
&\frac{1}{\mu_{\eps}}\log \mathcal{Z}_T^{\eps}(e^{\psi})=\sum_{n\geq 1}\sum_{\mathcal{T}^{\prec,\pm}\in\boldsymbol{\mathcal{T}_n^{\prec,\pm}}}\Big(\prod_{e\in E(\mathcal{T}^{\prec,\pm})}d_{e}\Big)\frac{1}{n!}\int \mathds{1}_{\mathbf{Z}_n\textup{ has no extra collision/overlap}}\\
&\hspace{3cm}\times \prod_{e=\{i,j\}\in E(\mathcal{T}^{\prec,\pm})}\big((v_{i,s_{e}^-}-v_{j,s_{e}^-})\cdot\omega_e\big)_+\prod_{1\leq i\leq n}\Big(e^{\psi}(z_{i,T})f^0(z_{i,0})\Big) dS_{n-1}d\Omega_{n-1}dV_ndy,
\end{split}
\end{equation}
As we take $\eps\rightarrow 0$, in the trajectories constructed by $\mathcal{T}^{\prec,\pm}$ and the relevant parameters, at each new collision or overlap, the new particle we add will asymptotically converge to the same position as the particle it collides or overlaps. This allows us to identify a stable limit trajectory $\mathbf{Z}_n$ as $\eps\rightarrow 0$
\begin{equation}
\begin{split}
&\lim_{\eps\rightarrow 0}\frac{1}{\mu_{\eps}}\log \mathcal{Z}_T^{\eps}(e^{\psi})=\sum_{n\geq 1}\sum_{\mathcal{T}^{\prec,\pm}\in\boldsymbol{\mathcal{T}_n^{\prec,\pm}}}\Big(\prod_{e\in E(\mathcal{T}^{\prec,\pm})}d_{e}\Big)\frac{1}{n!}\int \mathds{1}_{\mathbf{Z}_n\textup{ has no extra collision/overlap}}\\
&\hspace{3cm}\times \prod_{e=\{i,j\}\in E(\mathcal{T}^{\prec,\pm})}\big((v_{i,s_{e}^-}-v_{j,s_{e}^-})\cdot\omega_e\big)_+\prod_{1\leq i\leq n}\Big(e^{\psi}(z_{i,T})f^0(z_{i,0})\Big) dS_{n-1}d\Omega_{n-1}dV_ndy
\end{split}
\end{equation}
and this limit functional is denoted as $\Lambda_T^{0}(e^{\psi})$. The convergence is uniform for all $\psi$ satisfying \eqref{eq:test_function_bound}, thus 
\begin{equation}
\lim_{\eps\rightarrow 0}\frac{1}{\mu_{\eps}}\p_u^2\Lambda_T^{\eps}(e^{u\psi})=\lim_{\eps\rightarrow 0}\frac{1}{\mu_{\eps}}\p_u^2\Lambda_T^{0}(e^{u\psi})=0,
\end{equation}
which proves the second equation in \eqref{eq:CGF_goal}.

For the first equation in \eqref{eq:CGF_goal}, take derivatives of the limiting functional
\begin{equation}
\begin{split}
&\p_u\Lambda_T^{0}(e^{u\psi})\big|_{u=0}=\sum_{n\geq 1}\sum_{\mathcal{T}^{\prec,\pm}\in\boldsymbol{\mathcal{T}_n^{\prec,\pm}}}\Big(\prod_{e\in E(\mathcal{T}^{\prec,\pm})}d_{e}\Big)\frac{1}{n!}\int \mathds{1}_{\mathbf{Z}_n\textup{ has no extra collision/overlap}}\\
&\hspace{3cm}\times \prod_{e=\{i,j\}\in E(\mathcal{T}^{\prec,\pm})}\big((v_{i,s_{e}^-}-v_{j,s_{e}^-})\cdot\omega_e\big)_+\Big(\prod_{1\leq i\leq n}f^0(z_{i,0})\Big)\Big(\sum_{1\leq i\leq n}\psi(z_{i,T})\Big) dS_{n-1}d\Omega_{n-1}dV_ndy
\end{split}
\end{equation}
Use permutation symmetry between particles
\begin{equation}\label{eq:CGF_derivatives_final}
\begin{split}
&\p_u\Lambda_T^{0}(e^{u\psi})\big|_{u=0}=\sum_{n\geq 1}\sum_{\mathcal{T}^{\prec,\pm}\in\boldsymbol{\mathcal{T}_n^{\prec,\pm}}}\Big(\prod_{e\in E(\mathcal{T}^{\prec,\pm})}d_{e}\Big)\frac{1}{(n-1)!}\int \mathds{1}_{\mathbf{Z}_n\textup{ has no extra collision/overlap}}\\
&\hspace{3cm}\times \prod_{e=\{i,j\}\in E(\mathcal{T}^{\prec,\pm})}\big((v_{i,s_{e}^-}-v_{j,s_{e}^-})\cdot\omega_e\big)_+\Big(\prod_{1\leq i\leq n}f^0(z_{i,0})\Big)\psi(z_{1,T})dS_{n-1}d\Omega_{n-1}dV_ndy
\end{split}
\end{equation}
At this stage, the RHS of \eqref{eq:CGF_derivatives_final} is very close to the Boltzmann integral defined in Definition \ref{def:C_expansion} used to define the unique solution $f(T)$ of Boltzmann equation with initial data $f^0$.
\begin{equation}
f(t)=\sum_{\T:\textup{decorated trees}}\mathbf{C}_{f^0}[\T].
\end{equation}
The methods in \cite[Proposition IV.1]{bodineau2022cluster} allow us to reduce the summation over $\mathcal{T}^{\prec,\pm}$ into a summation over decorated (binary trees), see also Figure \ref{fig:clusters_to_pseudo_trajectory.png}. In this way, we can identify \eqref{eq:CGF_derivatives_final} as a summation over Boltzmann integrals. This concludes the proof of \eqref{eq:CGF_goal}, and thus Theorem \ref{th:kinetic_limit}.

\begin{figure}[h]
    \centering
    \includegraphics[width=1\linewidth]{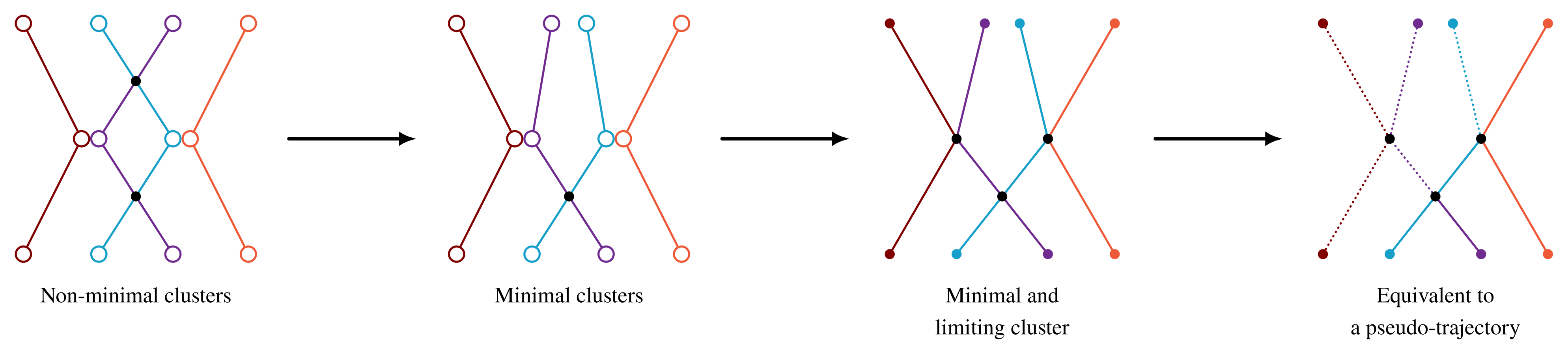}
    \caption{}
    \label{fig:clusters_to_pseudo_trajectory.png}
\end{figure}

\appendix
\section{Geometric Change of Variables}

In this appendix, we prove or recall the proof of several changes of variables useful in this paper. We use the following equality frequently
\begin{equation}\label{eq:dirac-to-hausdorf}
\int \delta\big(H(v)\big)dv=\int \frac{1}{|\nabla H(v)|}d\mathcal{H}_{H=0}(v).
\end{equation}

\begin{lemma}\label{lem:Carleman_change_of_variables}
\textbf{\textup{[Carleman Change of Variables]}} The integration over $d\omega$ and $dv_*$ can be changed into the integration over $dv_*'$ and $v'$ as
\begin{equation}
\big((v_*-v)\cdot\omega\big)_+d\omega dv_*=|v'-v|^{-1}d\pi_{v+(v'-v)^{\perp}}(v_*')dv'.
\end{equation}
and
\begin{equation}
d\pi_{v+(v'-v)^{\perp}}(v_*')dv'=\delta\Big((v'-v)\cdot(v_*'-v)\Big)|v'-v|dv'dv_*'=\frac{|v'-v|}{|v_*'-v|}d\pi_{v+(v_*'-v)^{\perp}}(v')dv_*'.
\end{equation}
\end{lemma}
\begin{proof}
First decompose $v_*-v$ into a component $v^{\parallel}$ parallel to $\omega$ and a component $v^{\perp}$ orthogonal to $\omega$
\begin{equation}\label{eq:polar_decomposition_app}
d\omega dv_*=d\omega dv^{\parallel}dv^{\perp},\quad  dv^{\parallel}:=\underbrace{d\Big((v_*-v)\cdot\omega\omega\Big)}_{\textup{\blue{parallel to $\omega$}}},\quad dv^{\perp}:=\underbrace{d\Big((v_*-v)-(v_*-v)\cdot\omega\omega\Big)}_{\textup{\blue{orthogonal to $\omega$}}}.
\end{equation}
In this case $d\omega dv^{\parallel}$ can be seen as a polar coordinate of $\mathds{R}^3$. Using the fact that $|v^{\parallel}|=|v-v'|$, we obtain
\begin{equation}\label{eq:polar_appendix}
d\omega dv^{\parallel}=|v'-v|^{-2}dv'=\big((v_*-v)\cdot\omega\big)_+^{-2}dv'.
\end{equation}
Apply \eqref{eq:polar_appendix} to \eqref{eq:polar_decomposition_app}
\begin{equation}
\big((v_*-v)\cdot\omega\big)_+d\omega dv_*=|v'-v|^{-1}dv'dv^{\perp}=|v'-v|^{-1}d\pi_{v+(v'-v)^{\perp}}(v_*')dv',
\end{equation}
which concludes the proof of the lemma.
\end{proof}

\begin{lemma}\label{lem:geometric_change_of_variable}
For any plane $\mathtt{PL}\subset \mathds{R}^3$, the following change of variables is true,
\begin{equation}
|v'-v|^{-1}d\pi_{v+(v'-v)^{\perp}}(v_*')dv'd\pi_{\mathtt{PL}}(v)=\pi\overline{d\pi}_{\mathtt{S}_{v_*',v'}\cap \mathtt{PL}}(v)dv_*'dv'.
\end{equation}
\end{lemma}

\begin{proof}
Apply \eqref{eq:dirac-to-hausdorf} 
\begin{equation}
|v'-v|^{-1}d\pi_{v+(v'-v)^{\perp}}(v_*')=\delta\Big((v'-v)\cdot (v_*'-v)\Big)dv_*',
\end{equation}
and thus
\begin{equation}
\begin{split}
|v'-v|^{-1}d\pi_{v+(v'-v)^{\perp}}(v_*')dv'd\pi_{\mathtt{PL}}(v)&=\delta\Big((v'-v)\cdot(v_*'-v)\Big)dv_*'dv'd\pi_{\mathtt{PL}}(v)\\
&=\delta\Big((v'-v)\cdot(v_*'-v)\Big)d\pi_{\mathtt{PL}}(v)dv_*'dv'
\end{split}
\end{equation}
Noticing
\begin{equation}
\nabla_{v}\Big((v'-v)\cdot (v_*'-v)\Big)=\Big(\nabla_v(v'-v)\Big)\cdot (v_*'-v)+(v'-v)\cdot\Big(\nabla_v(v_*'-v)\Big)=2v-v'-v_*'.
\end{equation}
Therefore
\begin{equation}
\begin{split}
|v'-v|^{-1}d\pi_{v+(v'-v)^{\perp}}(v_*')dv'd\pi_{\mathtt{PL}}(v)&=\frac{1}{\textup{diam}(\mathtt{S}_{v',v_*'}\cap \mathtt{PL})}d\pi_{\mathtt{S}_{v_*',v'}\cap L}(v)dv_*'dv'=\pi\overline{d\pi}_{\mathtt{S}_{v_*',v'}\cap \mathtt{PL}}(v)dv_*'dv'
\end{split}
\end{equation}
\end{proof}

\begin{lemma}\label{lem:geometric_change_of_variable_2}
\textbf{\textup{[Change of Variables]}} We have the following change of variables
\begin{equation}
|v'-v|^{-1}d\pi_{v+(v'-v)^{\perp}}(v_*')dv'dv=\pi|v'-v_*'|\overline{d\pi}_{\mathtt{S}_{v_*',v'}}(v)dv_*'dv'
\end{equation}
\end{lemma}
\begin{proof}
As already shown in the proof of Lemma \ref{lem:geometric_change_of_variable}
\begin{equation}
|v'-v|^{-1}d\pi_{v+(v'-v)^{\perp}}(v_*')dv'dv=\delta \Big((v'-v)\cdot(v_*'-v)\Big)dvdv_*'dv'.
\end{equation}
The Dirac mass above can be transformed into the measure of $v$ on the sphere $\mathtt{S}_{v_*',v'}$ formed by $v_*'$ and $v'$
\begin{equation}
\begin{split}
\delta \Big((v'-v)\cdot(v_*'-v)\Big)dvdv_*'dv'&=\frac{1}{|2v-v'-v_*'|}d\pi_{\mathtt{S}_{v',v_*'}}(v)dv_*'dv'=\pi|v'-v_*'|\overline{d\pi}_{\mathtt{S}_{v',v_*'}}(v)dv_*'dv'
\end{split}
\end{equation}
\end{proof}

\paragraph{AI Statement} The initial formulation of this project and the mathematical contents are entirely produced by the author, starting from early 2025, without AI assistance. The writing of this paper is also entirely produced by the author except for: AI tools were used for proofreading, and transforming the figures hand-written by the author into latex-drawn figures.

\bibliography{bibIT.bib}
\bibliographystyle{plain}
\end{document}